\documentclass[11pt,reqno]{amsart}
\usepackage[utf8]{inputenc}
\usepackage{amsfonts,amsthm,amsmath,amssymb, bbold, amscd, amssymb, amscd, mathrsfs, stmaryrd, hyperref,  bbm, enumerate, url}
\usepackage[all, cmtip]{xy}
\usepackage[text={6.5in,9in},centering,letterpaper,margin=1in]{geometry}
\usepackage{MnSymbol}
\usepackage{float}
\usepackage{array,diagbox}
\usepackage{graphicx}
\usepackage{enumerate}
\usepackage{hyperref}
\usepackage{color}
\usepackage{mathtools}
\usepackage{tikz}
\usepackage{tikz-cd}
\usepackage{pgf}
\usetikzlibrary{math}
\usetikzlibrary{decorations.markings}
\usetikzlibrary{decorations.pathreplacing}
\usetikzlibrary{arrows,shapes,positioning}
\tikzstyle directed=[postaction={decorate,decoration={markings,
    mark=at position #1 with {\arrow{>}}}}]
\tikzstyle rdirected=[postaction={decorate,decoration={markings,
    mark=at position #1 with {\arrow{<}}}}]

\tikzset{anchorbase/.style={baseline={([yshift=-0.5ex]current bounding box.center)}},
    tinynodes/.style={font=\tiny,text height=0.75ex,text depth=0.15ex},
    smallnodes/.style={font=\scriptsize,text height=0.75ex,text depth=0.15ex},
    >={Latex[length=1mm, width=1.5mm]}
  }
  \tikzset{
    partial ellipse/.style args={#1:#2:#3}{
        insert path={+ (#1:#3) arc (#1:#2:#3)}
    }
}
\tikzcdset{arrow style=tikz, diagrams={>=stealth}}
  
\usetikzlibrary{patterns}
\usepackage[
    maxbibnames=99,
    backend=bibtex,
    style=alphabetic,
    giveninits=true
]{biblatex}
\DeclareFieldFormat{pages}{#1}
\renewbibmacro{in:}{%
  \ifentrytype{article}
    {}
    {\bibstring{in}%
     \printunit{\intitlepunct}}}
\DeclareFieldFormat
[article,inbook,incollection,inproceedings,patent,thesis,unpublished]
  {title}{\mkbibemph{#1}}
\DeclareFieldFormat{journaltitle}{#1\isdot}
\DeclareFieldFormat[article]{volume}{\mkbibbold{#1}}
\DeclareFieldFormat[article]{number}{\bibstring{number}\addnbspace #1}

\renewbibmacro*{journal+issuetitle}{%
  \usebibmacro{journal}%
  \setunit*{\addspace}%
  \iffieldundef{series}
    {}
    {\newunit
     \printfield{series}%
     \setunit{\addspace}}%
  \printfield{volume}%
  \setunit{\addspace}%
  \usebibmacro{issue+date}%
  \setunit{\addcomma\space}%
  \printfield{number}%
  \setunit{\addcolon\space}%
  \usebibmacro{issue}%
  \setunit{\addcomma\space}%
  \printfield{eid}
  \newunit}

\newtheoremstyle{mystyle}
  {}
  {}
  {\itshape}
  {}
  {\bfseries}
  {.}
  { }
  {\thmname{#1}\thmnumber{ #2}\thmnote{ (#3)}}

\tikzset{anchorbase/.style={baseline={([yshift=-0.5ex]current bounding box.center)}}}
\usetikzlibrary{calc}
\usetikzlibrary{decorations.markings}
\usetikzlibrary{decorations.pathreplacing}
\usetikzlibrary{arrows,shapes,positioning}
\tikzstyle directed=[postaction={decorate,decoration={markings,
    mark=at position #1 with {\arrow{>}}}}]
\tikzstyle rdirected=[postaction={decorate,decoration={markings,
    mark=at position #1 with {\arrow{<}}}}]

\newcommand{\capfig}[4]{ 
\begin{scope}[shift={(#1,#2-5.9)},scale=#3]
\path[fill=blue,opacity=.2] (1,0) arc[start angle=0, end angle=180,x radius=1,y radius=.5] 
		to [out=90,in=180] (0,1.5) to [out=0,in=90] (1,0);
	\path[fill=blue,opacity=.2] (1,0) arc[start angle=360, end angle=180,x radius=1,y radius=.5] 
		to [out=90,in=180] (0,1.5) to [out=0,in=90] (1,0);
	\draw[very thick] (0,0) ellipse (1 and 0.5);
	\draw[very thick] (-1,0) to [out=90,in=180] (0,1.5) to [out=0,in=90] (1,0);
\end{scope} 
}
\newcommand{\braid}[4]{
\begin{scope}[shift={(#1,#2-5.9)},scale=#3]

 \draw[very thick] (0,0) to (1,2);
 \draw[very thick] (0,2) to (.4,1.2);
 \draw[very thick] (.6,0.8) to (1,0);

\end{scope} 
}
\newcommand{\cupcap}[4]{
\begin{scope}[shift={(#1,#2-5.9)},scale=#3]
\draw[very thick] (0,0) to [in=100,out=80](1,0);
\draw[very thick] (0,2) to [in=-100, out=-80](1,2);
\end{scope} 
}
\newcommand{\lines}[4]{
\begin{scope}[shift={(#1,#2-5.9)},scale=#3]
\draw[very thick] (0,0) to (0,2);
\draw[very thick] (1,0) to (1,2);
\end{scope} 
}

\newcommand{\nsphere}[4]{
\begin{scope}[shift={(#1,#2-5.9)},scale=#3]
	\path [fill=blue,opacity=0.3] (0,0) circle (1);
	\draw (-1,0) .. controls (-1,-.4) and (1,-.4) .. (1,0);
	\draw[dashed] (-1,0) .. controls (-1,.4) and (1,.4) .. (1,0);
	\draw[very thick] (0,0) circle (1);
	
\end{scope}

}
\newcommand{\fsphere}[3]{
\draw[white, line width=1mm] (#1-#3,#2) to [out=270,in=180] (#1,#2-#3*2/3) to [out=0,in=270] (#1+#3,#2);
\draw (#1-#3,#2) to [out=270,in=180] (#1,#2-#3*2/3) to [out=0,in=270] (#1+#3,#2);
\draw[thick] (#1,#2) circle (#3);
}

\newcommand{\bsphere}[3]{
\draw[dashed] (#1-#3,#2) to [out=90,in=180] (#1,#2+#3*2/3) to [out=0,in=90] (#1+#3,#2);
}

\newcommand{\fgsphere}[3]{
\draw[white, line width=1mm] (#1-#3,#2) to [out=270,in=180] (#1,#2-#3*2/3) to [out=0,in=270] (#1+#3,#2);
\draw[opacity=.3] (#1-#3,#2) to [out=270,in=180] (#1,#2-#3*2/3) to [out=0,in=270] (#1+#3,#2);
\draw[thick, opacity=.3] (#1,#2) circle (#3);
}

\newcommand{\bgsphere}[3]{
\draw[dashed,opacity=.3] (#1-#3,#2) to [out=90,in=180] (#1,#2+#3*2/3) to [out=0,in=90] (#1+#3,#2);
}

\newcommand{\sphere}[3]{
\bsphere{#1}{#2}{#3}
\fsphere{#1}{#2}{#3}
}
\newcommand{\trefoil}[4]{
\begin{scope}[shift={(#1,#2-5.9)},scale=#3]
\draw[#4] (-0.3,5.15) to [out=150,in=270] (-1,6) to   [out=90,in=180] (0,6.8) to [out=0,in=110] (0.85,6.3)
(-0.7,6) to [out=20,in=160] (1,6) to   [out=340,in=60] (1.4,5) to [out=240,in=330] (0.3,4.85)
(0.85,5.7) to [out=240,in=30] (0,5) to   [out=210,in=300] (-1.4,5) to [out=120,in=200] (-1.15,5.85);
\end{scope}
}

\newcommand{\hopflink}[4]{
\begin{scope}[shift={(#1,#2)},scale=#3]
\begin{scope}[yscale=.66]
\draw[#4]
(.5,0) to [out=90,in=0] (-.5,1) to [out=180,in=90] (-1.5,0);
\draw[white, line width=1mm]
(-.5,0) to [out=90,in=180] (.5,1) to [out=0,in=90] (1.5,0);
\draw[#4]
(-.5,0) to [out=90,in=180] (.5,1) to [out=0,in=90] (1.5,0) to [out=270,in=0] (.5,-1) to [out=180,in=270] (-.5,0);
\draw[white, line width=1mm]
(.5,0) to [out=270,in=0] (-.5,-1);
\draw[#4]
(.5,0) to [out=270,in=0] (-.5,-1) to [out=180,in=270] (-1.5,0);
\end{scope}
\end{scope}
}
\newcommand{\hopflinkcol}[5]{
\begin{scope}[shift={(#1,#2)},scale=#3]
\begin{scope}[yscale=.66]
\draw[#5]
(.5,0) to [out=90,in=0] (-.5,1) to [out=180,in=90] (-1.5,0);
\draw[white, line width=1mm]
(-.5,0) to [out=90,in=180] (.5,1) to [out=0,in=90] (1.5,0);
\draw[#4]
(-.5,0) to [out=90,in=180] (.5,1) to [out=0,in=90] (1.5,0) to [out=270,in=0] (.5,-1) to [out=180,in=270] (-.5,0);
\draw[white, line width=1mm]
(.5,0) to [out=270,in=0] (-.5,-1);
\draw[#5]
(.5,0) to [out=270,in=0] (-.5,-1) to [out=180,in=270] (-1.5,0);
\end{scope}
\end{scope}
}

\newcommand{\unknot}[4]{
\begin{scope}[shift={(#1-1,#2)},scale=#3]
\begin{scope}[yscale=.66]
\draw[#4]
(-.5,0) to [out=90,in=180] (.5,1) to [out=0,in=90] (1.5,0) to [out=270,in=0] (.5,-1) to [out=180,in=270] (-.5,0);
\end{scope}
\end{scope}
}

\newcommand{\torushole}[4]{
\draw[#4] (#1-#3/2,#2+#3/6) to [out=300,in=180] (#1,#2-#3/6) to [out=0,in=240] (#1+#3/2,#2+#3/6);
\draw[#4] (#1-#3/3,#2) to [out=45,in=180] (#1,#2+#3/6) to [out=0,in=135] (#1+#3/3,#2);
}

\newcommand{\torus}[4]{
\draw[#4] (#1,#2) ellipse (#3 and #3*2/3);
\torushole{#1}{#2}{#3}{#4}
}

\newcommand{\lasagnafillingfigure}[1]{
    \begin{tikzpicture}[anchorbase,scale=#1]
        \begin{scope}[shift={(-18,-1)},xscale=3,yscale=12]
            \draw[dashed] (0,1) to [out=0,in=90] (1,0) to [out=270,in=0] (0,-1)
            (0,1) to [out=180,in=90] (-1,0) to [out=270,in=180] (0,-1);
        \end{scope}
        \begin{scope}[shift={(-18,-1)},xscale=9,yscale=12]
            \draw[thick] (0,-1) to [out=15,in=180] (2,-1.3) to [out=0,in=270] (3.8,0)
            to [out=90,in=0] (2,1.3) to [out=180,in=340] (0,1);
        \end{scope}
        \bsphere{-4}{-4}{10}
        \fsphere{-4}{-4}{10}
        \sphere{10.5}{-.5}{3}
        \draw[white, line width=1mm]
        (-8,4.5) to [out=80,in=290] (-7.3,11)
        (-2.9,11) to [out=280, in=180] (-1,9) to [out=0,in=260] (.7,11)
        (5.2,10.7) to [out=250, in=180] (6,8) to [out=0,in=260] (7.3,10)
        (10.3,10) to [out=250, in=90] (9,6) to [out=270,in=110] (12,0)
        (9,0) to [out=115, in=0] (0,5) to [out=180,in=20] (-3.3,4.5)
        ;
        \draw[blue, thick]
        (-8,4.5) to [out=80,in=290] (-7.3,11)
        (-2.9,11) to [out=280, in=180] (-1,9) to [out=0,in=260] (.7,11)
        (5.2,10.7) to [out=250, in=180] (6,8) to [out=0,in=260] (7.3,10)
        (10.3,10) to [out=250, in=90] (9,6) to [out=270,in=110] (12,0)
        (9,0) to [out=115, in=0] (0,5) to [out=180,in=20] (-3.3,4.5)
        ;
        \torushole{2}{7}{2}{thick, blue}
        \hopflink{10.5}{0}{1}{red,very thick}
        \begin{scope}[shift={(-6.9,4.5)},xscale=2.3,yscale=.4]
            \draw[red, very thick]
            (-.5,0) to [out=90,in=180] (.5,1) to [out=0,in=90] (1.5,0) to [out=270,in=0] (.5,-1) to [out=180,in=270] (-.5,0);
        \end{scope}
        \trefoil{3}{9}{1.5}{red,very thick}
        \hopflink{-5}{11}{1.5}{red,very thick}
        \unknot{9}{10}{1.5}{red,very thick}
        \node at (8,-8) {\small$F_1$};
        \node at (-8,3.4) {\tiny$v_1$};
        \node at (10.5,-1.5) {\tiny$v_k$};
        \end{tikzpicture}
\quad \sim \quad
\begin{tikzpicture}[anchorbase,scale=#1]
    \begin{scope}[shift={(-18,-1)},xscale=3,yscale=12]
        \draw[dashed] (0,1) to [out=0,in=90] (1,0) to [out=270,in=0] (0,-1)
        (0,1) to [out=180,in=90] (-1,0) to [out=270,in=180] (0,-1);
    \end{scope}
    \begin{scope}[shift={(-18,-1)},xscale=9,yscale=12]
        \draw[thick] (0,-1) to [out=15,in=180] (2,-1.3) to [out=0,in=270] (3.8,0)
        to [out=90,in=0] (2,1.3) to [out=180,in=340] (0,1);
    \end{scope}
    \bgsphere{-4}{-4}{10}
    \fgsphere{-4}{-4}{10}
    \sphere{1}{-5}{3}
    \sphere{10.5}{-.5}{3}
    \sphere{-8}{.6}{3}
    \draw[white, line width=1mm]
    (-9.5,.7) to [out=80,in=290] (-7.3,11)
    (-2.9,11) to [out=280, in=180] (-1,9) to [out=0,in=260] (.7,11)
    (5.2,10.7) to [out=250, in=180] (6,8) to [out=0,in=260] (7.3,10)
    (10.3,10) to [out=250, in=90] (9,6) to [out=270,in=110] (12,0)
    (9,0) to [out=115, in=0] (0,5) to [out=180,in=70] (-6.5,.7)
    (3.1,-4.7) to [out=75, in=0] (-.5,3) to [out=180,in=110] (-.8,-5.2)
    (1.5,-4.4) to [out=75, in=0] (-.1,.5) to [out=180,in=110] (.7,-5.2)
    ;
    \draw[blue, thick]
    (-9.5,.7) to [out=80,in=290] (-7.3,11)
    (-2.9,11) to [out=280, in=180] (-1,9) to [out=0,in=260] (.7,11)
    (5.2,10.7) to [out=250, in=180] (6,8) to [out=0,in=260] (7.3,10)
    (10.3,10) to [out=250, in=90] (9,6) to [out=270,in=110] (12,0)
    (9,0) to [out=115, in=0] (0,5) to [out=180,in=70] (-6.5,.7)
    ;
    \draw[green, thick]
    (3.1,-4.7) to [out=75, in=0] (-.5,3) to [out=180,in=110] (-.8,-5.2)
    (1.5,-4.4) to [out=75, in=0] (-.1,.5) to [out=180,in=110] (.7,-5.2)
    ;
    \torushole{2}{7}{2}{thick, blue}
    \torushole{-.3}{1.5}{2}{thick, green}
    \torus{-5}{-4}{2.3}{thick, green}
    \trefoil{-8}{1.5}{1}{red,very thick}
    \hopflink{10.5}{0}{1}{red,very thick}
    \unknot{0.5}{-5}{.75}{red,very thick}
    \unknot{3}{-4.5}{.75}{red,very thick}
    \begin{scope}[shift={(-6.9,4.5)},xscale=2.3,yscale=.4]
        \draw[red, very thick, opacity=.5]
        (-.5,0) to [out=90,in=180] (.5,1) to [out=0,in=90] (1.5,0) to [out=270,in=0] (.5,-1) to [out=180,in=270] (-.5,0);
    \end{scope}
    \trefoil{3}{9}{1.5}{red,very thick}
    \hopflink{-5}{11}{1.5}{red,very thick}
    \unknot{9}{10}{1.5}{red,very thick}
    \node at (8,-8) {\small$F_2$};
    \node at (-8,-8) {\small$F_3$};
    \node at (-8,-.5) {\tiny$v_i$};
    \node at (1.6,-6) {\tiny$v_j$};
    \node at (10.5,-1.5) {\tiny$v_k$};
    \end{tikzpicture}
    }

\theoremstyle{plain}
\newtheorem{Thm}{Theorem}[section]
\newtheorem{Lem}[Thm]{Lemma}
\newtheorem{Cor}[Thm]{Corollary}
\newtheorem{Prop}[Thm]{Proposition}

\newtheorem{Que}[Thm]{Question}

\theoremstyle{definition}
\newtheorem{Def}[Thm]{Definition}

\theoremstyle{remark}
\newtheorem{Rmk}[Thm]{Remark}

\newcommand{\bigdownarrow}[1]{\Big\downarrow\mathrlap{\,\scriptstyle #1}}
\newcommand{\biguparrow}[1]{\Big\uparrow\mathrlap{\,\scriptstyle #1}}

\newcommand{\R}{\mathbb{R}}
\newcommand{\Z}{\mathbb{Z}}
\newcommand{\N}{\mathbb{N}}

\def\cKhRNa{\underline{\operatorname{KhR}}_{N, \alpha}}

\newcommand{\ru}{to [out=0,in=270]}
\newcommand{\rd}{to [out=0,in=90]}
\newcommand{\ur}{to [out=90,in=180]}

\newcommand{\dr}{to [out=270,in=180]}

\begin{document}

\title{Spectral Sequences in Lasagna Theory}


\author{Amey Joshi}

\subjclass[2010]{Primary }

\address{Department of Mathematics, Michigan State University, East Lansing, MI 48824, USA}
\email{joshiam3@msu.edu}

\dedicatory{}

\

\maketitle
\begin{abstract}
    We use Khovanov-Floer theories and their spectral sequences to construct differentials on the corresponding skein-lasagna modules. Then we show that for 2-handlebodies, a spectral sequence of TQFTs coming from a Khovanov-Floer theory gives a spectral sequence of lasagna modules under the aforementioned differential. As a corollary, we recover the rank inequality between the Khovanov and Lee lasagna modules.
\end{abstract}


\section{Background and Introduction}
Following Khovanov's categorification of the Jones polynomial in \cite{khovanov1999categorificationjonespolynomial}, a new world of TQFTs and manifold invariants opened. Using Khovanov homology, Rasmussen defined the celebrated s-invariant in \cite{rasmussen2004khovanovhomologyslicegenus}. The s-invariant was used to provide the first-ever proof of the existence of exotic 4-manifolds that did not use gauge theory.
Skein-lasagna modules were first defined in \cite{Morrison_2022}. They used a functorial fix of the Khovanov homology to define an invariant of four-manifolds. These modules were significantly larger invariants than anything defined earlier. However, in \cite{rw24}, the authors used these modules to define a generalization of Rasmussen's s-invariant for links in the boundary of an arbitrary four-manifold. The authors used this new invariant to distinguish exotic structures. The idea was to extract a smaller and easier-to-compute invariant out of the skein-lasagna modules, which they achieved by using a Lee-lasagna module.
\par
Suppose $X$ is a four-manifold and $L\subset \partial X $, a link. Given a link invariant $Z$ satisfying certain axioms (referred to as``TQFTs" in the rest of the paper), one can define a corresponding lasagna module $S^{Z}(X, L)$ that is an invariant of the pair $(X,L)$.
These modules have a natural grading coming from the second cohomology:
\begin{equation}
    S^{Z}(X, L)= \bigoplus_{\alpha \in H_{2}(X,L)} S^{Z}_{\alpha}(X, L)
\end{equation}

In \cite{mn22}, the authors gave a formula to compute the skein lasagna modules for 2-handlebodies. To explore further possibilities, it is natural to analyze how lasagna modules depend on their input TQFT, especially under small perturbations of the differential. The formula for the skein lasagna modules of 2-handlebodies proved in \cite{mn22}, resembles a filtered colimit. Filtered colimits are known to commute with homology; hence, the following question arises:

\begin{Que}
    Does the inequality between the ranks of the Khovanov skein lasagna and the Lee lasagna modules in \cite{rw24}:
    \begin{equation}
        rank(S^{KhR_2}_{\alpha,q,h}(X,L,\mathbb{Q}))\geq  rank(gr_{q}S^{KhR_{Lee}}_{\alpha,h}(X,L,\mathbb{Q})) 
    \end{equation}
    (for a 2 handlebody $(X,L)$ and $q \in \mathbb{Z}$) come from a spectral sequence? \footnote{Here $q$ and $h$ are the quantum and homological gradings, respectively, coming from Khovanov homology.}
    \par In particular, is there a spectral sequence whose second page is the Khovanov-Rozansky skein lasagna that converges to the Lee lasagna module?
\end{Que}
One can also observe that there are multiple spectral sequences that relate two different TQFTs; many such examples are listed in \cite{b18}. Such spectral sequences are called Khovanov-Floer theories, since these spectral sequences start at a Khovanov-type homology and converge to some version of  Floer homology. Hence, we also ask:
\begin{Que} 
Does a spectral sequence of TQFTs give a spectral sequence of the corresponding lasagna modules? \footnote{   Note that we already know that this is not true in general. For example, there are examples of manifolds with one handle in which the Lee lasagna is larger than the Khovanov lasagna module.}
\end{Que}

In this paper, we give partial answers to both of these questions. 
\begin{Thm} \label{thm: main}
     Let $X$ be a 2-handlebody, formed by attaching 2-handles to a framed link $K\subset \partial B^4$. Let $L$ be a link disjoint from $K$. Let $\{ E_i, d_i \}$ be a spectral sequence of TQFTs coming from a Khovanov-Floer theory. Then there is a differential on the corresponding lasagna modules so that:
    
       $$  H_{*}(S^{E^{n}}(X,L))\simeq S^{E^{n+1}}(X,L) $$
   Such a spectral sequence is an invariant of the pair $(X,L)$. Henceforth, we will denote such a spectral sequence as $LSS(X,L,\mathbb{F})$, or just $LSS$ (short for lasagna spectral sequence) when the underlying pair $(X,L)$ is understood. 
    
\end{Thm}
\subsection{Outline: } We start by discussing the basics of lasagna modules in Section 2. In Section 3, we will discuss some algebraic preliminaries, as well as define some new categories that behave well with respect to colimits. In Section 4, we define the basics as well as the necessary modifications of the Khovanov-Floer theories. After pointing out that Khovanov-Floer theories satisfy the right axioms, we define lasagna modules for them and discuss their properties that will later be used to construct the $LSS$. In Section 5, we define a differential on the lasagna modules corresponding to each page of the spectral sequence and prove the main theorem. Finally, in Section 6, we discuss a stronger version of our result where we have a notion of filtered chain complexes arising from a quantum degree. In the last section, we first show how to obtain the rank inequality in \cite{rw24} as a corollary of our result. Then we also prove a connect-sum formula and provide some corollaries of the $LSS$ construction.
\subsection{Idea of the Proof}

To summarize, we prove the existence of the spectral sequence in Theorem \ref{thm: main} as follows:

\begin{align*} 
    H_{*}(S_{\alpha \in H_{2}(X,L) } ^{E^{n}}(X,L))\cong  H_{*}( {\text{Fcolim}}_{r\in \mathbb{N}^{k}} {\text{colim}}_{\sigma \in \mathfrak{S}_{\|\alpha \|+2r}} E^{n}(K(r)))  \\  \cong  F{\text{colim}}_{r\in \mathbb{N}^{k}} H_{*}({\text{colim}}_{\sigma \in \mathfrak{S}_{\|\alpha \|+2r}} E^{n}(K(r)))\\ \cong  {\text{Fcolim}}_{r\in \mathbb{N}^{k}} {\text{colim}}_{\sigma \in \mathfrak{S}_{\|\alpha \|+2r}} H_{*}(E^{n}(K(r)))\\ \cong   {\text{Fcolim}}_{r\in \mathbb{N}^{k}} {\text{colim}}_{\sigma \in \mathfrak{S}_{\|\alpha \|+2r}} E^{n+1}(K(r)) \\ \cong  S_{\alpha \in H_{2}(X,L) } ^{E^{n+1}}(X,L)
\end{align*}
To prove the first step, we introduce some new algebraic techniques (Section 3) and accordingly reformulate skein lasagna modules as well as the 2-handlebody formula in terms of the aforementioned algebraic terms (Section 5). The remaining steps follow from rather straightforward computations in Section 5. Since the differential is constructed on a general skein lasagna module (rather than just a 2-handlebody), it follows that the construction of the spectral sequence is independent of the choice of a 2-handlebody structure.

So, simply by observing the``filtered colimit" in the 2-handlebody formula in \cite{mn22}, one might expect a rather straightforward proof of the existence of the $LSS$. However, the details are rather tedious. The key issues that arise are:

\paragraph{
\textbf{Field of characteristic zero :} $1+1$ being invertible is crucial to many proofs in \cite{rw24} and \cite{mn22}. To give an example of one of these instances, the 2-handlebody formula considers a colimit with respect to ``braid group action". 1+1 being invertible is crucial for this group action to commute with the differentials.}
\paragraph{
\textbf{Problems with existing Khovanov-Floer Theories : } The work done in \cite{b18,saltz2018strongkhovanovfloertheoriesfunctoriality} is over $\mathbb{Z}_2$; hence we need to use the right functorial version of the Khovanov Homology over characteristic zero, and choose the right axioms of Khovanov-Floer theories. Apart from that, these papers only work for TQFTs for links in $\mathbb{R}^3$. To be able to define lasagna modules, however, we need to work with TQFTs for links in $S^3$. }
\paragraph{
\textbf{Comparing the differentials : } 
We will be using the universal property of colimits to define differentials of the $LSS$. However, while doing so, we encounter some limitations of the 2-handlebody formula in \cite{mn22}. To be precise, the isomorphism between skein lasagna and cabled Khovanov homology is not ``functorial". That is, the isomorphism does not come from a functor between indexing categories that induces an isomorphism on the colimits. Most of the work in Section 3 is to fix this issue. Finally, we reformulate the 2-handlebody formula in Section 5 so that the differentials on both the cabled homology and the skein lasagna modules (naturally coming from the universal property of colimits) agree.}
\paragraph{
\textbf{Convergence page vs. $E^\infty$ page : } 
Note that at first sight it may look like our result implies a direct rank inequality for the Khovanov and the Lee lasagna modules. However, this is not true since the $E^\infty$ page of a KFT is not the same as the Lee homology as a TQFT. Observe that a cobordism from the empty link to the unknot given by a genus 3 surface gives a zero map on $E^\infty$, while it gives multiplication by 8 on the Lee TQFT. Hence, we need to use quantum degree to relate lasagna modules coming from these two similar TQFTs.
}

\paragraph{
\textbf{Acknowledgments: } 
The author is grateful to his PhD advisors, Matthew Hedden and Matthew Stoffregen, for many detailed and stimulating discussions, which greatly influenced this work. The author is especially indebted to Mike Willis for generous guidance and invaluable insight at every stage of the project. The author also thanks Qiuyu Ren for carefully pointing out a mistake in an earlier construction.
}
\section{Skein Modules}

We start by defining what we mean by a ``TQFT". Throughout this paper, we use the following convention:\footnote{In the classical literature, these axioms would describe a symmetric monoidal (3+1) TQFT.}

\begin{Def} \label{def:TQFT} 
   Let $\Gamma$ be a discrete group, and $R$ be a commutative ring, then a TQFT for links in an abstract $\mathbb{R}^3$ is a functor 
\[
\begin{Bmatrix}
\textrm{link embeddings in an oriented X $\cong$ } \mathbb{R}^3
\\
\textrm{link cobordisms (possibly dotted) in oriented }
Y \cong X \times I 
\\
\textrm{ up to isotopy rel } \partial
\end{Bmatrix}
\xrightarrow{Z}
\begin{Bmatrix}
\ R\textrm{-modules} 
\\
\ ( \textrm{possibly} \ \mathbb{Z}\textrm{-graded, } \Gamma\textrm{-graded/filtered }  
\\
\   \ \Gamma\textrm{-hom./filt.} R\textrm{-morphisms})
\end{Bmatrix}
\]
that satisfies \footnote{TQFTs that satisfy the second property are generally referred to as "lax monoidal".} 
 \begin{align*}
    Z(\emptyset) & \cong R \\
    Z(L_{1}\sqcup L_{2})  &\cong  Z(L_{1})\bigotimes Z(L_{2})  \\
 \end{align*} 
\end{Def}

There is no canonical notion of a disjoint union for links in $S^3$. Hence, we define TQFTs for links in $S^3$ as follows:

\begin{Def}\label{def:s3tqft}
Let $\Gamma$ be a discrete group, $R$ be a commutative ring, then a TQFT for links in an abstract $S^3$ is a functor 
 \[
\begin{Bmatrix}
\textrm{link embeddings in an oriented }\  S\cong S^3
\\
\textrm{link cobordisms (possibly dotted) in oriented }
Y \cong S \times I 
\\
\textrm{ up to isotopy rel } \partial
\end{Bmatrix}
\xrightarrow{Z}
\begin{Bmatrix}
\ R\textrm{-modules} 
\\
\ ( \textrm{possibly} \ Z\textrm{-graded, } \Gamma\textrm{-graded/filtered }  
\\
\   \ Z\textrm{-pres., } \Gamma\textrm{-hom./filt.} R\textrm{-morphisms})
\end{Bmatrix}
\]

\end{Def}

Throughout this paper, we will be dealing with pairs $(X,L)$, so that $X$ is a four manifold with a link $L\subset \partial X$.
Following \cite{rw24}, define \textit{category of skeins} in $(X,L)$, denoted $\mathcal{C}(X,L)$, as follows:
\begin{Def}

 The objects of $C(X,L)$ consist of all 2-dimensional skeins (possibly dotted)\footnote{The original definition in \cite{rw24} does not include dotted surfaces but it is easy to see that the definition does not change even if dots are included} in $X$ rel $L$ i.e. \begin{equation}\label{eq:skein} \Sigma\subset X\backslash{\sqcup_{i=1}^kint(B_i)},\text{ with }\partial\Sigma|_{\partial X}=L 
\end{equation} \footnote{Since we are working in the smooth category, the tangent vector to $\Sigma$ at any point on the link $L_{i}$ is normal to $\partial B_i$}

Consider $\Sigma'$, another object with input balls $B_j'$, $j=1,\cdots,k'$. If $\sqcup B_i\subset\sqcup B_j'$ and each $B_j'$ either coincides with some $B_i$ or does not intersect any $B_i$ on its boundary, then the set of morphisms from $\Sigma$ to $\Sigma'$ consists of pairs of an isotopy class of surfaces $S\subset\sqcup B_j'\backslash\sqcup\,int(B_i)$ rel boundary with $S|_{\partial B_i}=\Sigma|_{\partial B_i}$, $S|_{\partial B_j'}=\Sigma'|_{\partial B_j'}$, and an isotopy class of isotopies from $S\cup\Sigma'$ to $\Sigma$ rel boundary. We will suppress the isotopy and write $[S]\colon\Sigma\to\Sigma'$ for this morphism. If the input balls $B_i,B_j'$ do not satisfy the condition above, there is no morphism from $\Sigma$ to $\Sigma'$. These morphisms are illustrated in the Figure \ref{fig:lasagnafillingequiv}, which we shall describe later. For now, one should note that the surface in green represents a morphism between two objects of $C(X,L)$

\end{Def}
Note that while working over a field of characteristic zero, one can interpret dotted cobordisms by the following relation: 
\begin{equation}
	\label{eq:dotdef}
\begin{tikzpicture} [fill opacity=0.2,anchorbase, scale=.375]
	\path[fill=blue,opacity=.2] (1,4) arc[start angle=0, end angle=180,x radius=1,y radius=.5] 
		to [out=270,in=180] (0,2.5) to [out=0,in=270] (1,4);
	\path[fill=blue,opacity=.2] (1,4) arc[start angle=360, end angle=180,x radius=1,y radius=.5] 
		to [out=270,in=180] (0,2.5) to [out=0,in=270] (1,4);
	\draw[very thick] (0,4) ellipse (1 and 0.5);
	\draw[very thick] (-1,4) to [out=270,in=180] (0,2.5) to [out=0,in=270] (1,4);
	\node[opacity=1] at (0,3) {\footnotesize$\bullet$};
\end{tikzpicture}
:=
\frac{1}{2}
\begin{tikzpicture}[scale=.5,anchorbase]
\begin{scope}
    \clip (-.65,4) arc[start angle=180, end angle=360,x radius=.65,y radius=.25]
    	to (1.25,4) to (1.25,.4) to (-1.25,.4) to  (-1.25,4) to (-.65,4);
\fill[blue,opacity=.3] (0,2.5) ellipse (1 and 2);
\draw[very thick] (0,2.5) ellipse (1 and 2);
\fill[white] (0,3) to [out=300,in=60] (0,2) to [out=120,in=240] (0,3);
\draw[very thick] (0,3) to [out=300,in=60] (0,2);
\draw[very thick] (0.1,1.8) to [out=125,in=235] (0.1,3.2);
\end{scope}
\fill[blue,opacity=.2] (0,4) ellipse (.65 and .25);
\draw[very thick] (0,4) ellipse (.66 and .25);
\end{tikzpicture} \, .
\end{equation}

Before we go on to define skein lasagna modules for a general TQFT, we note that most of the TQFTs, including Khovanov and Lee homologies, are defined for links in a fixed $S^3$. that is, a space diffeomorphic to $$\{(x,y,z,w)\subset \mathbb{R}^{4} \| x^2 + y^2 + z^2 + w^2= 1 \}$$.
Whenever the manifold $S$ in terms of \ref{def:s3tqft} is this standard $S^3$, we refer to a TQFT as a ``TQFT of the standard $S^3$."

\begin{Def} \label{def:TQFTextends}
    We say that a TQFT $Z$ on the standard $S^3$ extends to a TQFT $\overline{Z}$ of abstract $S^3$ if for every orientation-preserving diffeomorphism $\phi: Y \rightarrow S^3$ and a link $L\subset Y$, there is a corresponding isomorphism $\phi^{Z}(L): \overline{Z}(L) \rightarrow Z(\phi(L)) $
    and for every diffeomorphism $\Phi: Y\times I \rightarrow S^3 \times I $ and every cobordism $\Sigma \subset Y\times I$, the following diagram commutes:
\\
\[
\begin{tikzcd}
    \overline{Z}(L_1) \arrow{r}{\overline{Z}(\Sigma)} \arrow{d}{{\Phi |}^{Z}_{Y\times\{ 0\} }} &  \overline{Z}(L_2) \arrow{d}{{\Phi |}^{Z}_{Y\times \{ 1 \} }} \\
    Z({\Phi |}_{Y\times\{ 0\} }(L_1)) \arrow{r}{Z(\Sigma)} & Z({\Phi |}_{Y\times\{ 0\} }(L_2)) \\
\end{tikzcd}
\]
\end{Def}

From now on, we only work over a field $\mathbb{k}$ of characteristic zero.
\begin{Def}
    For a given TQFT $Z$ (for links in an abstract $S^3$), define a functor that assigns to every $\Sigma\in C(X,L)$, the module $\bigotimes_{i=1}^{k} Z(L_i)$ \footnote{To avoid confusion, recall that $L_i \subset \partial B_i$ are the links in the input balls, while $L\subset \partial X$ is the fixed boundary link such that for any $\Sigma \in C(X,L) $, $\Sigma \ \cap \ \partial X= L$}. We define the $Z$-lasagna module denoted by $S^{Z}(X,L,\mathbb{k})$ to be 
    \begin{equation}\label{eq:colim}
\mathcal{S}^Z(X;L,\mathbb{k}):=\mathrm{colim}_{[\Sigma ] \in C(X,L)} Z([\Sigma]).
\end{equation}\footnote{Note that, to evaluate this colimit, we need to realize the cobordism $\Sigma$ as a surface embedded in an abstract $S^3 \times I$. To do this, we take connect sum of all the input balls. The Sweep-around move assures that this construction is well-defined and independent of where the connect-sum operation is performed. For details, refer to \cite{rw24}.}
\begin{Rmk}
    This colimit is computed in the category of vector spaces over the field $\mathbb{k}$. 
\end{Rmk}

\end{Def}
Observe that a morphism exists between two skeins only when they represent the same element of $H_{2}(X,L)$ (after collapsing the input balls). Hence, as pointed out earlier, one gets a well-defined grading:
\begin{equation}
    S^{Z}(X, L, \mathbb{k})= \bigoplus_{\alpha \in H_{2}(X,L)} S^{Z}_{\alpha}(X, L, \mathbb{k})
\end{equation}

Note that even though this definition of lasagna modules is much more useful and natural to us, in order to visualize elements of this convoluted module, one should follow the definition in \cite{mn22} and regard the elements as skeins with input links labeled by an element of the corresponding TQFT. That is, to visualize elements of this module, we will consider equivalence classes of pairs $(\Sigma, \nu)$ (denoted by $[\Sigma, \nu]$) where the input links $\{L_i\}$ of $\Sigma$ are colored by $\nu = \otimes v_i (\text{where } v_i \in Z(L_i) )$.

For each element of $C(X,L)$, we consider the connect sum of the input balls as demonstrated in the figure \ref{fig:connectsum} reproduced from \cite{morrison2024invariantssurfacessmooth4manifolds}. Hence, we get a cobordism between two links inside $\#_j  B_j^{'} - \#_i B_i^{\circ} \cong S^{3} \times I $, thus giving us a well-defined functor.

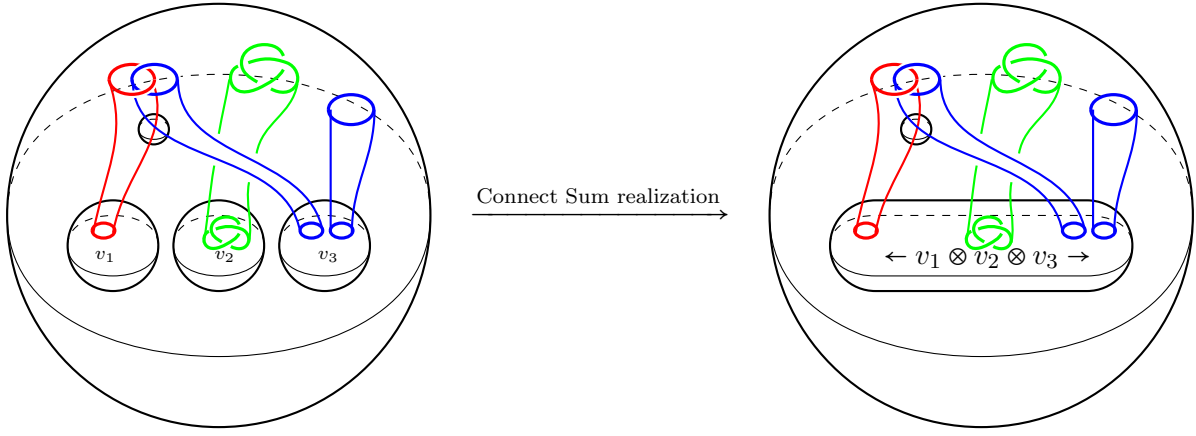
\begin{figure}[htbp]\label{fig:connectsum} 
\centering
\[ \begin{tikzpicture}[anchorbase,scale=.2]
    \bsphere{0}{2}{14}
    \sphere{-4.3}{7.7}{1}
    \sphere{0}{0}{3}
    \sphere{7}{0}{3}
    \sphere{-7}{0}{3}
    \trefoil{.5}{1}{1}{green,very thick}
    \node at (0.4,-.7) {\tiny$v_2$};
    \unknot{-7}{1}{.75}{red,very thick}
    \node at (-7.5,-.7) {\tiny$v_1$};
    \unknot{6.75}{1}{.75}{blue,very thick}
    \unknot{8.75}{1}{.75}{blue,very thick}
    \node at (7.25,-.7) {\tiny$v_3$};
    \draw[red, thick] (-8.3,.85) to [out=80,in=290] (-7.3,11)
    (-7,.85) to [out=80,in=290] (-4.4,11);
    \draw[blue, thick]
    (5.4,.85) to [out=100,in=270] (-5.7,11)
    (6.9,.85) to [out=100,in=270] (-2.7,11)
    (7.4,.85) to [out=90,in=270] (7.4,9)
    (8.8,.85) to [out=90,in=270] (10.3,9);
    \draw[green, thick]
    (0,7.4) to [out=80,in=260] (.7,11)
    (2.5,6.1) to [out=80, in=260] (5.2,10.7)
    (-1,0.2) to [out=80,in=260] (-.3,5.3)
    (2,0.2) to [out=80, in=260] (2,4.2)
    ;
    \trefoil{3}{9}{1.5}{green,very thick}
    \hopflinkcol{-5}{11}{1.5}{blue,very thick}{red,very thick}
    \unknot{9}{9}{1.5}{blue,very thick}
    \fsphere{0}{2}{14}
    \end{tikzpicture}\quad \xrightarrow{\text{Connect Sum realization}} \quad
    \begin{tikzpicture}[anchorbase,scale=.2]
      \bsphere{0}{2}{14}
      \sphere{-4.3}{7.7}{1}
       \node at (0.4,-.8) {$\leftarrow \tiny v_1 \otimes v_2 \otimes v_3 \rightarrow$};
      \draw[white, line width=1mm] (-10,0) \dr (-7,-2) to (7,-2) \ru (10,0);
      \draw  (-10,0) \dr (-7,-2) to (7,-2) \ru (10,0);
      \draw[thick] 
      (-10,0) \dr (-7,-3) to (7,-3) \ru (10,0) 
      (-10,0) \ur (-7,3)to (7,3) \rd (10,0);
      \draw[dashed] (-10,0) \ur (-7,2)to (7,2) \rd (10,0);
      \trefoil{.5}{1}{1}{green,very thick}
      \unknot{-7}{1}{.75}{red,very thick}
      \unknot{6.75}{1}{.75}{blue,very thick}
      \unknot{8.75}{1}{.75}{blue,very thick}
      \draw[red, thick] (-8.3,.85) to [out=80,in=290] (-7.3,11)
      (-7,.85) to [out=80,in=290] (-4.4,11);
      \draw[blue, thick]
      (5.4,.85) to [out=100,in=270] (-5.7,11)
      (6.9,.85) to [out=100,in=270] (-2.7,11)
      (7.4,.85) to [out=90,in=270] (7.4,9)
      (8.8,.85) to [out=90,in=270] (10.3,9);
      \draw[green, thick]
      (0,7.4) to [out=80,in=260] (.7,11)
      (2.5,6.1) to [out=80, in=260] (5.2,10.7)
      (-1,0.2) to [out=80,in=260] (-.3,5.3)
      (2,0.2) to [out=80, in=260] (2,4.2)
      ;
      \trefoil{3}{9}{1.5}{green,very thick}
      \hopflinkcol{-5}{11}{1.5}{blue,very thick}{red,very thick}
      \unknot{9}{9}{1.5}{blue,very thick}
      \fsphere{0}{2}{14}
      \end{tikzpicture}  
\]

\caption{Connect Sum of the input balls}
\label{fig:connectsum}
\end{figure}

Note that equivalently, one can think of  the skein lasagna module as an equivalence class of pairs $[\Sigma,\nu]$, where $\Sigma$ is a skein with boundary links $L_i$, and $\nu \in Z(\bigotimes L_i)$. The equivalence relation is given in the Figure \ref{fig:lasagnafillingequiv} reproduced from \cite{Morrison_2022}

\begin{figure}[ht]
	\[
	\lasagnafillingfigure{.20}
	\]
	\caption{}
	\label{fig:lasagnafillingequiv}
\end{figure}

As we will see in the fourth section, to prove our main result, it is much more convenient to deal with a slight modification of the category $C(X,L)$:
\begin{Def}\label{def:newcategory}
   We define the category $\overline{C(X,L)}$ as follows:
   \begin{itemize}
       \item \textbf{Objects:} $$\text{Ob}(\overline{C(X,L)}) \equiv \text{Ob}(C(X,L))$$ That is, this new category has the same objects as $C(X,L)$ (Hence, we call these objects skeins)
       \item For any two skeins $[\Sigma]$ and $[\Sigma]^{'}$ we have: $$\text{Hom}_{C(X,L)}([\Sigma],[\Sigma]^{'}) \subset \text{Hom}_{\overline{C(X,L)}}([\Sigma],[\Sigma]^{'})$$ That is, the new category has all the morphisms that were present in $C(X,L)$.
       \item We describe a new class of morphisms that did not previously exist in $C(X,L)$:
       \\ Suppose $[\Sigma]$ is a skein with input balls $\sqcup_{i=1}^{n}B_{i}$ and $[\Sigma^{'}]$ is a skein with input balls $\sqcup_{j=1}^{k}B_{j}^{'}\sqcup_{i=2}^{n}B_i$ such that:
       \begin{enumerate}
           \item $\sqcup_{j=1}^{k}B_{j}^{'} \subset B_1$
           \item $\Sigma^{'}|_{X-B_1}=\Sigma$
           \item Suppose $S=\Sigma^{'}\cap B_1$ and there exists a map $$F:B_1-\sqcup_{j=1}^{k}B_{j}^{'} \longrightarrow S^3 \times I $$\footnote{A more accurate description of the map $F$ is as follows: First choose paths between the input balls $\gamma_j$ in $X$ going from $B^{'}_j$ to $B^{'}_{j+1}$ disjoint from the surface $\Sigma$. Choose a tubular neighborhood of each $\gamma_j$ disjoint from the input balls as well as the skein. The map $F$ we are looking for is then given by collapsing the tubular neighborhood of each $\gamma_j$ followed by a diffeomorphism. Hence giving us a map of the form $$F:B_1-\sqcup_{j=1}^{k}B_{j}^{'} \xrightarrow{\text{collapsing } N(\gamma_i) } B_1-\#_{j=1}^{k}B_{j}^{'}  \xrightarrow{\text{diffeomorphism}} S^3 \times I $$   } given by connect sum quotient of the input balls $\{   B_{j}\}_{j=1}^{k}$ followed by a diffeomorphism onto the standard $S^3 \times I$; such that $F(S)=\sqcup_{i} \ (  L_{i}\times I)$\footnote{This condition is trying to say that the surface $S$ looks like a disjoint union of identity cobordisms on boundary links}
           
       \end{enumerate}
    For any such pairs $[\Sigma]$ and $[\Sigma^{'}]$, we define a new class of morphisms in  the category $\overline{C(X,L)}$ denoted by $G_{\overline{C(X,L)}}$. Maps in the class $G_{\overline{C(X,L)}}$ are denoted as follows: $$g_{\overline{C(X,L)}}:[\Sigma]\rightarrow [\Sigma^{'}]$$ As depicted in the Figure \ref{fig:newcatfg}, this map $g_{\overline{C(X,L)}}$ is supposed to act as an ``inverse" to the map $f_{\overline{C(X,L)}}$  that already existed in $C(X,L)$. 
      \item $\text{Hom}_{\overline{C(X,L)}}$ is the formal closure under composition of $\text{Hom}_{C(X,L)} \cup G_{\overline{C(X,L)}}$. That is every morphism in $\overline{C(X,L)}$ is of the form $$f_{1}\circ g_{1} \circ \dots \circ g_{n} \circ f_{n+1}$$ where each $f_{i} \in \text{Hom}_{C(X,L)} \subset \text{Hom}_{\overline{C(X,L)}}$ and each $g_{j} \in G_{\overline{C(X,L)}} $.
       
   \end{itemize}
\end{Def} 
\begin{figure}[htbp]
\centering

\[
    \begin{tikzpicture}[anchorbase, scale=0.5]
    \begin{scope}[xshift=-1cm]

    \sphere{0}{0}{4}
    
    \sphere{-1}{0}{0.5}
    \sphere{1}{0}{0.5}
    \unknot{-1}{3}{0.4}{very thick, blue}
    \trefoil{1.5}{7}{0.4}{very thick, green}
     
    \unknot{-0.2}{0.2}{0.2}{very thick, blue}
    \trefoil{1.2}{5}{0.2}{very thick, green}
    \draw[thick] (-2.2,3) to [in=90,out=-90] (-1.35,1.5);
     \draw[thick] (-1.38,3) to [in=90,out=-90] (-0.77,1.5);
     \draw[thick] (1.1,3) to [in=90,out=-90] (0.77,1.3);
     \draw[thick] (1.8,3) to [in=90,out=-90] (1.5,1.3);
     \draw[thick]    (-1.35,1.5)    to [in=90,out=-90] (-1.31,0.2);
    \draw[thick]    (-0.77,1.5)    to [in=90,out=-90]   (-0.9,0.2);
    \draw[thick]     (0.77,1.3)   to [in=90,out=-90]      (0.93,0.2);
     \draw[thick]     (1.5,1.3)   to [in=90,out=-90]     (1.4,0.2);
     \end{scope}

 \end{tikzpicture} \quad \equiv \quad
 \begin{tikzpicture}[anchorbase, scale=0.5]
    \begin{scope}[xshift=1cm]
    \sphere{0}{0}{4}
    \nsphere{0}{6}{2}
    \sphere{-1}{0}{0.5}
    \sphere{1}{0}{0.5}
    \unknot{-1}{3}{0.4}{very thick, blue}
    \trefoil{1.5}{7}{0.4}{very thick, green}
     \unknot{-0.2}{1.5}{0.3}{very thick, blue}
    \trefoil{1.2}{5.5}{0.3}{very thick, green}
    \unknot{-0.2}{0.2}{0.2}{very thick, blue}
    \trefoil{1.2}{5}{0.2}{very thick, green}
    \draw[thick] (-2.2,3) to [in=90,out=-90] (-1.35,1.5);
     \draw[thick] (-1.38,3) to [in=90,out=-90] (-0.77,1.5);
     \draw[thick] (1.1,3) to [in=90,out=-90] (0.77,1.3);
     \draw[thick] (1.8,3) to [in=90,out=-90] (1.5,1.3);
     \draw[thick]    (-1.35,1.5)    to [in=90,out=-90] (-1.31,0.2);
    \draw[thick]    (-0.77,1.5)    to [in=90,out=-90]   (-0.9,0.2);
    \draw[thick]     (0.77,1.3)   to [in=90,out=-90]      (0.93,0.2);
     \draw[thick]     (1.5,1.3)   to [in=90,out=-90]     (1.4,0.2);
    \end{scope}
    
 \end{tikzpicture} \quad \substack{g_{\overline{C(X,L)}}\\[-2.5em]\xleftarrow{\hspace{1.5cm}}\\[-2em] \xrightarrow{\hspace{1.5cm}} \\[-2em] f_{\overline{C(X,L)}} }
 \begin{tikzpicture}[anchorbase, scale=0.5]
    \begin{scope}[xshift=1cm]
    \sphere{0}{0}{4}
    \sphere{0}{0}{2}
    
    \unknot{-1}{3}{0.4}{very thick, blue}
    \trefoil{1.5}{7}{0.4}{very thick, green}
     \unknot{-0.2}{1.5}{0.3}{very thick, blue}
    \trefoil{1.2}{5.5}{0.3}{very thick, green}
    
    \draw[thick] (-2.2,3) to [in=90,out=-90] (-1.35,1.5);
     \draw[thick] (-1.38,3) to [in=90,out=-90] (-0.77,1.5);
     \draw[thick] (1.1,3) to [in=90,out=-90] (0.77,1.3);
     \draw[thick] (1.8,3) to [in=90,out=-90] (1.5,1.3);
     
    \end{scope}
    
 \end{tikzpicture}
\]

\caption{Depiction of the maps  $f_{\overline{C(X,L)}}$ and $g_{\overline{C(X,L)}}$ }
\label{fig:newcatfg}
\end{figure}
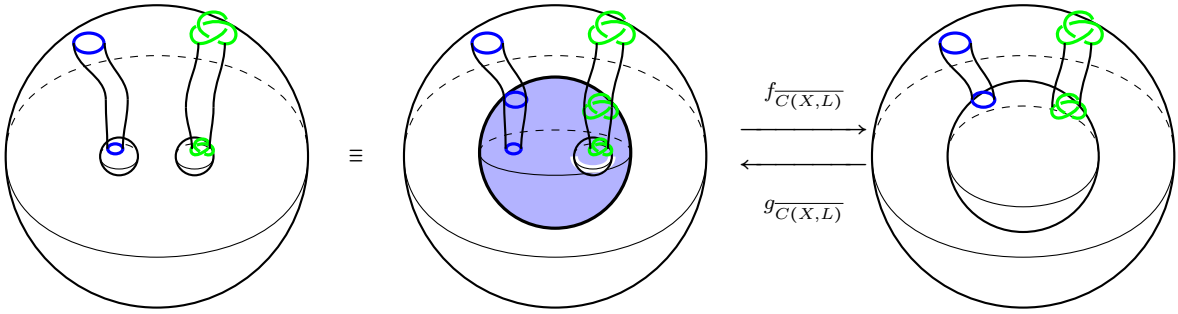
    
Now we give a reformulation of the skein lasagna modules with respect to the category $\overline{C(X,L)}$. Given a TQFT $Z$, consider a functor $\overline{Z}$ defined as follows: For any object $[\Sigma]$, $$\overline{Z}([\Sigma]) \coloneqq  Z([\Sigma])$$ and for every morphism $f \in \text{Hom}_{C(X,L)} \subset \text{Hom}(\overline{C(X,L)}$, $$\overline{Z}(f) \coloneqq Z(f)$$
Finally, with the notation of the Figure \ref{fig:newcatfg} in mind, we define $\overline{Z}$ over $G_{\overline{C(X,L)}}$ as follows: $$\overline{Z}(g_{\overline{C(X,L)}}) \coloneqq Z(f_{\overline{C(X,L)}})^{-1}$$ \footnote{ We only defined $\overline{Z}$ over $\text{Hom}_{C(X,L)}\cup G_{\overline{C(X,L}}$, however, it extends to $\text{Hom}_{\overline{C(X,L)}}$ by simply considering the appropriate compositions.}
Then skein lasagna modules can be equivalently defined as:
\begin{equation}\label{eq:newcolim}
\mathcal{S}^{Z}(X;L,\mathbb{k}) \coloneqq \mathrm{colim}_{[\Sigma ] \in \overline{C(X,L)}} \overline{Z}([\Sigma]).
\end{equation}
Note that the functor $\overline{Z}$ restricts to $Z$ on the subcategory $C(X,L)$.
As we will see in section 4, the colimits in the definition \ref{eq:colim} and the equation \ref{eq:newcolim} are isomorphic; hence we can use either of the two definitions for the skein lasagna modules. 
Henceforth, we will write $\overline{Z}$ simply as $Z$ whenever convenient.
The idea behind defining the new category $\overline{C(X,L)}$ is to eventually take an appropriate quotient, so that "big input balls" and "small input balls" represent the same skeins as shown in the figure \ref{fig:newcat}

\begin{figure}[htbp]

\[ 
  \begin{tikzpicture}[anchorbase,scale=.2]
    \bsphere{0}{2}{14}
    \sphere{-4.3}{7.7}{1}
    \sphere{0}{0}{4}
    \sphere{0}{0}{1.5}
    \sphere{7}{0}{2}
    \sphere{-7}{0}{2}
    \trefoil{.5}{2.5}{1}{green,very thick}
    
    \unknot{-7}{1}{.75}{red,very thick}
   
    \unknot{6.75}{1}{.75}{blue,very thick}
    \unknot{8.75}{1}{.75}{blue,very thick}
   
    \draw[red, thick] (-8.3,.85) to [out=80,in=290] (-7.3,11)
    (-7,.85) to [out=80,in=290] (-4.4,11);
    \draw[blue, thick]
    (5.4,.85) to [out=100,in=270] (-5.7,11)
    (6.9,.85) to [out=100,in=270] (-2.7,11)
    (7.4,.85) to [out=90,in=270] (7.4,9)
    (8.8,.85) to [out=90,in=270] (10.3,9);
    \draw[green, thick]
    (0,7.4) to [out=80,in=260] (.7,11)
    (2.5,6.1) to [out=80, in=260] (5.2,10.7)
    (-1,1.2) to [out=80,in=260] (-.3,5.3)
    (2,1.2) to [out=80, in=260] (2,4.2)
    ;
    \trefoil{3}{9}{1.5}{green,very thick}
    \hopflinkcol{-5}{11}{1.5}{blue,very thick}{red,very thick}
    \unknot{9}{9}{1.5}{blue,very thick}
    \fsphere{0}{2}{14}
    \end{tikzpicture}\quad \cong \quad  
    \begin{tikzpicture}[anchorbase,scale=.2]
    \begin{scope}[xshift=28cm]
    \bsphere{0}{2}{14}
    \sphere{-4.3}{7.7}{1}
    \sphere{0}{0}{1.5}
    \sphere{7}{0}{3}
    \sphere{-7}{0}{3}
    \trefoil{.5}{2.5}{0.6}{green,very thick}
    
    \unknot{-7}{1}{.75}{red,very thick}
   
    \unknot{6.75}{1}{.75}{blue,very thick}
    \unknot{8.75}{1}{.75}{blue,very thick}
   
    \draw[red, thick] (-8.3,.85) to [out=80,in=290] (-7.3,11)
    (-7,.85) to [out=80,in=290] (-4.4,11);
    \draw[blue, thick]
    (5.4,.85) to [out=100,in=270] (-5.7,11)
    (6.9,.85) to [out=100,in=270] (-2.7,11)
    (7.4,.85) to [out=90,in=270] (7.4,9)
    (8.8,.85) to [out=90,in=270] (10.3,9);
    \draw[green, thick]
    (0,7.4) to [out=80,in=260] (.7,11)
    (2.5,6.1) to [out=80, in=260] (5.2,10.7)
    (-0.5,0.2) to [out=80,in=260] (-.3,5.3)
    (1.5,0.2) to [out=80, in=260] (2,4.2)
    ;
    \trefoil{3}{9}{1.5}{green,very thick}
    \hopflinkcol{-5}{11}{1.5}{blue,very thick}{red,very thick}
    \unknot{9}{9}{1.5}{blue,very thick}
    \fsphere{0}{2}{14}
    \end{scope}
    \end{tikzpicture}
    \]
\caption{Equivalence relation that will be imposed in $\overline{C(X,L)}$ after an appropriate quotient}
\label{fig:newcat}
\end{figure}
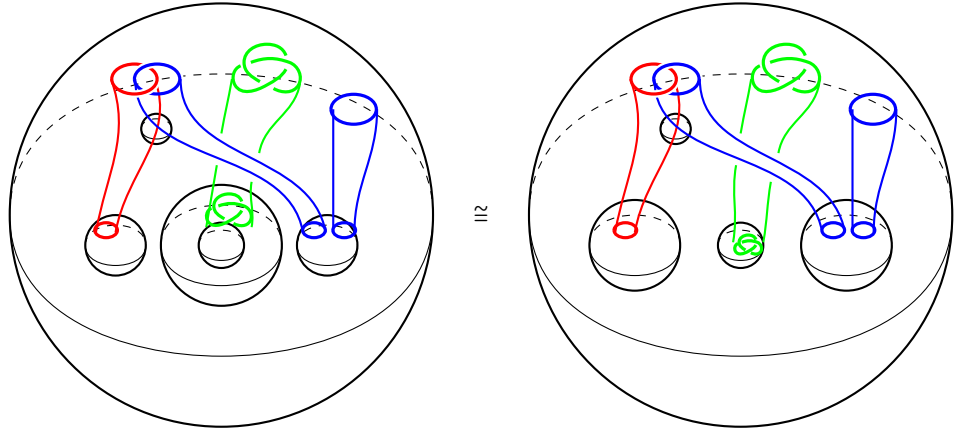

Now, we recall the 2-handle-body formula from \cite{mn22}. Later, we will give a more general statement for a general set of TQFTs.

Let $K$ be a framed link, with $n$ components and let $\textbf{k}=(k_1, \dots, k_n)$ be the framing vector. For two $n$-vectors $\textbf{r}_1$ and $\textbf{r}_2$, we introduce the following notation:
for a framed, oriented link $K=\cup K_{i}$ with $n$ components, let \textbf{$K(\textbf{r}_1,\textbf{r}_2)$} denote the $\textbf{r}_1 + \textbf{r}_2 $ cable of $K$ with $\textbf{r}_{1}^{i}$ strands going in the same direction as $K_i$ and the other $\textbf{r}^{i}_{2}$ going in the opposite direction.

\begin{Prop}[$2$-handle-body formula]\label{prop:2_hdby}
We recall that the cabled Khovanov-Rozansky homology of $K$ at level $\alpha$ is 
\[
\cKhRNa(K) =  \Bigl( \bigoplus\limits_{r\in\N^n} KhR_N (K(\textbf{r}-\alpha^-,\textbf{r}+\alpha^+))\{(1-N)(2r+|\alpha|)\} \Bigr)/ \sim
\]
where $\{ \_\}$ denotes a shift in the quantum degree, and the equivalence $\sim$ is the transitive and linear closure of the relations
\begin{equation}\label{eq:sim}
\beta_i(b)v \sim v, \ \ \psi^{[m]}_i(v) \sim 0 \text{ for } m < N-1,\ \ \psi^{[N-1]}_i(v) \sim v
\end{equation} \footnote{We will later discuss what these maps $\psi$ represent. To get an idea, one may think of $\psi^m$ as adding a genus $m$ surface connect sum with an annulus to the newly added strands. }
for all $i=1, \dots, n$; $b \in B_{r_i-\alpha^-, r_i+\alpha^+}$, and $v \in  KhR_N (K(r-\alpha^-,r+\alpha^+)).$
Then, for a 2-handlebody $(X,L)$ formed by attaching 2-handles to $B^4$ along a framed link $K\subset \partial B^4$ and a link $L$ disjoint from $K$, we get:
\begin{equation}
    S^{KhR_{N}}_{0,\alpha}(X,L,\mathbb{Q})\cong \cKhRNa(K)
\end{equation}
\end{Prop}

\section{Algebraic preliminaries}
In this section, we establish some algebraic background that we shall use in the subsequent sections. In particular, we provide the necessary algebraic recipe to perturb the category $\overline{C(X,L)}$, so that the 2-handlebody formula \ref{thm:2hdlbdy} can be realized as a ``finality of a filtered subcategory" result, which will help us rigorously construct the desired spectral sequences.
\subsection{Special colimits }
As we have already seen, lasagna modules are intrinsically colimits of some directed system. We discuss certain properties of some special types of colimits, following \cite[chapter IX]{mac1998categories}. \\
 
\begin{Def}
    An index category $\mathcal{J}$\footnote{In this context, an index category corresponds to the diagram to which we will assign objects and arrows in $C$, to form a system over which we will take a colimit. For example, in lasagna modules, $C(X,L)$ is such a category. We shall refer to morphisms in such a category as ``arrows."} is called \textbf{filtered} if 
\begin{enumerate}
    \item To any objects $j$, $j^{'}$ $\exists \ k $ and arrows $j \dashrightarrow k$ and $j^{'}\dashrightarrow k$ in the category $\mathcal{J}$.
    \item For any two arrows $u,v:i\rightarrow j$, there exists an arrow $w: j \dashrightarrow k$ such that the following diagram commutes: 
\[
      \begin{tikzcd}
        &j \arrow{rd}{w} & \\
        i \arrow{ru}{u} \arrow{rd}{v}& & k \\
        & j \arrow{ru}{w}& 
      \end{tikzcd}
  \]
\end{enumerate}
\end{Def}
\begin{Def} \label{def:final}
    A functor between two indexing categories $F:\mathcal{I} \rightarrow \mathcal{J}$ is called \textbf{final} if 
\begin{enumerate}
\item  For every object $k\in \mathcal{J}$,
$\exists \ i \in \mathcal{I}$ and an arrow $k\rightarrow Fi$ in the category $\mathcal{J}$.
\item Any two such arrows are related by the following commutative diagram: \par
\[
    \begin{tikzcd}
        k \arrow{dd} \arrow{r}{\text{Id}_{k}}&k \arrow{dd} \arrow{r}{\text{Id}_{k}}&k \arrow{dd} &k \arrow{dd} \arrow{r}{\text{Id}_{k}}&k \arrow{dd} \\
        &&&&\\
        Fi\arrow{r}{Fg_1}&\bullet & \bullet \arrow{l}{Fg_2} \arrow[dotted]{r}& \bullet & Fi \arrow{l}{Fg_3}  \\
    \end{tikzcd}
\]

\end{enumerate}
Here all the $g_i$'s are morphisms in $\mathcal{I}$.
\end{Def}
The above definition is suggesting that the subcategory $\mathcal{I}$ has "enough" objects and maps to get the following theorem:

\begin{Thm}
For an indexing functor $G: \mathcal{J} \rightarrow C$, a final functor $F:\mathcal{I}\rightarrow \mathcal{J}$ induces an isomorphism $${\text{colim}}_{\mathcal{J}}G(\mathcal{J})\rightarrow {\text{colim}}_{\mathcal{I}}G\circ F(\mathcal{I})$$ if the latter colimit exists.
\end{Thm}
Now we point out that a weaker condition is enough for a subsystem to give the same colimit. Firstly, observe that if the index category $\mathcal{J}$ has maps $f,g:A\rightarrow B$, then in the colimit with respect to the functor $G$, we get $G(f)v\sim v $ and $G(g)v\sim v $. If the category $C$ is $Vec_{\mathbb{k}}$, then we observe that $$[\frac{1}{2}G(f)v +\frac{1}{2}G(g)v ]\sim v$$ 
Inspired by this, we define the following:
\begin{Def} \label{def:convexclosure}
    For an (indexing) category $\mathcal{J}$ and a field $\mathbb{k}$, define the \textbf{convex closure} of $\mathcal{J}$ denoted by $\mathcal{J}^{conv}$ (when the underlying field $\mathbb{k}$ is understood) as follows:
    \begin{enumerate}
        \item The objects of $\mathcal{J}^{conv}$ are the same as $\mathcal{J}$.
        \item The maps are given by the following property: For objects $A,B$, the maps are formal linear combinations $\sum_{i} c_{i} f_{i}$ where $f_{i} \in \text{Hom}_{\mathcal{J}}(A,B)$, such that $c_{i} \in \mathbb{k}$ with $\sum_{i} c_{i}=1 $ 
    \end{enumerate}
\end{Def}
\begin{Rmk}
    In the definition \ref{def:convexclosure}, note that compositions in the category $\mathcal{J}^{conv}$ are defined as follows:
    $$\text{Let } \{f_i\}_{i=1}^{n} \in Hom_{\mathcal{J}}(A,B),\  \{g_j\}_{j=1}^{m} \in Hom_{\mathcal{J}}(B,C) $$ then we define the composition of $f=\sum_{i}c_i  f_i \ , \  g=\sum_jb_j g_j$ in $\mathcal{J}^{conv}$ as follows:

    $$ g\circ f : =  \sum_{i,j} c_i b_j( g_j \circ f_i)$$ 
Note that if $\sum_{i}c_i =1 $ and $\sum_j b_j =1$ then $\sum_{i,j}c_i b_j =1$, hence the composition $g\circ f $ as defined above is valid in $\mathcal{J}^{conv}$.
\end{Rmk}
It is easy to see the following:
\begin{Thm} \label{thm:convexclosure}
    There is a natural inclusion functor $i:  \mathcal{J} \rightarrow \mathcal{J}^{conv}$ and  any functor $G:\mathcal{J}\rightarrow Vec_{\mathbb{k}} $ extends to a functor $\overline{G}$ on the convex closure, so that we have an isomorphism: $$ \text{colim}_{\mathcal{J}} G(\mathcal{J}) \cong \text{colim}_{\mathcal{J}^{conv}} \overline{G}(\mathcal{J}^{conv})$$
\end{Thm}

Combining all these definitions and results, we define filtered colimits.
\begin{Def}
    A colimit in a category $C$, taken over a system $\mathcal{J}$, is called a filtered colimit if there is a final and filtered subsystem $\mathcal{I}$ of $\mathcal{J}$. 
\end{Def}
Filtered colimits have special properties that are crucial to finding a spectral sequence structure. 

\begin{Thm}
    Let $C$ be the category of chain complexes over $\mathbb{k}$ with an indexing functor $F:\mathcal{J} \rightarrow C$, and let $H_{*}$ be the functor giving homology. Then the differential on each object of $C$ induces a differential on the colimit \footnote{We shall discuss this induced differential later, for our specific case of lasagna modules} and we get the following \footnote{Note that this colimit is taken over the category of vector spaces after applying the forgetful functor to $C$} $$H_{*}({\text{colim}}_{\mathcal{J}}(F(\mathcal{J}))\cong {\text{colim}}_{\mathcal{J}}(H_{*}\circ F (\mathcal{J}))$$
\end{Thm}
\begin{proof}

    The idea of filtered colimits is that they mimic colimits over the system $\{n\in \mathbb{N}\}$, hence we sketch the proof only for this system. The proof follows from observing that the colimit over a filtered system is an exact functor. 
    That is, a given exact sequence in $\text{Mod}_R$
    $$0\rightarrow A_n \rightarrow B_n \rightarrow C_n \rightarrow 0$$
    $$\text{induces the exact sequence :}$$
     $$0\rightarrow \text{colim}_{n} \ A_n \rightarrow  \text{colim}_{n} \ B_n \rightarrow  \text{colim}_{n}  C_n \rightarrow 0$$
    \par To check this, observe that since the system is filtered, for any two elements in the colimit, one can choose representatives in the same n-level. The rest of the exactness computations follow from exactness at each $n$-level.
\end{proof}
From now on, we will denote filtered colimits as ``$\textbf{Fcolim}$" to highlight their special properties and to point out that we can ``flip" the homology with filtered colimits 

Finally, we observe a relationship between colimits and certain quotient categories that will help us prove the required results later.

Let $\mathcal{J}$ be an indexing category for a colimit with respect to the functor $F$. Then, suppose we take a quotient of $\mathcal{J}$ with respect to an equivalence relation that satisfies the following condition:
$$  f,g\in \text{Hom}_{\mathcal{J}}(A,B) \ \text{and }f \sim g \implies \ F(f)=F(g)$$ (We will refer to this as $F$-invariance.)  Then we have:
\begin{Thm} \label{thm:quotientcategory}
    The functor $F$ induces a functor $\overline{F}$ on the quotient category $\mathcal{J}/\sim$ as well as an isomorphism 
    $$\text{colim}_{\mathcal{J}/\sim}(F(\mathcal{J}/\sim)) \cong{\text{colim}}_{\mathcal{J}}(F(\mathcal{J}))$$ This isomorphism is universal in the sense that the arrows given by the universal property (going away) from both colimits commute with respect to this isomorphism.
\end{Thm}

Inspired by this result, we define a weaker condition than``finality" of a subsystem:
\begin{Def}
    A subsystem $\mathcal{I}\subset \mathcal{J}$ is called \textbf{pseudo-final} with respect to a functor $F:\mathcal{J}\rightarrow Vec_{\mathbb{k}}$ if there exists an $F$-invariant equivalence relation on $\mathcal{J}$ so that the corresponding quotient subsystem $\mathcal{I}/\sim$ is final in  $\mathcal{J}/\sim$.
\end{Def}
In light of the above discussion, the following theorems show that our new definition of skein lasagna modules is isomorphic to the old one. 
\begin{Thm}
The inclusion functor  $I: C(X,L)\rightarrow \overline{C(X,L)}$ is pseudo-final with respect to the functor $\overline{Z}$. 
    
\end{Thm}
\begin{proof}
The proof follows after considering the appropriate quotient: $f_{\overline{C(X,L)}}\circ g_{\overline{C(X,L)}}\sim\text{Id}$ in terms of the definition \ref{def:newcategory}. Note that due to closure under composition, a general morphism in $\overline{C(X,L)}$ is a composition of morphisms that either exist in $C(X,L)$ or inside $G_{\overline{C(X,L)}} $. 
Now we will prove the result by using \ref{def:final}. Firstly, since both categories have the same objects, the first requirement is trivially satisfied. Now, without loss of generality, let's assume the morphism to be $\text{Id}_{[\Sigma]}$\footnote{If we check the required condition for a pair of the Identity morphism and a general morphism, we have checked it for all the possible pairs. } and a morphism of the form $$T_{\overline{C(X,L)}}=  g^{1}_{\overline{C(X,L)}}\circ [S^1 ] \circ \dots g^{n}_{\overline{C(X,L)}}\circ [S^m ]$$ where $[S^i ] $ are morphisms in the subcategory $C(X,L)$. Then the required commutative diagram as in \ref{def:final} can be seen as follows:
Since we are working in a quotient category where the maps in $F_{\overline{C(X,L)}} $  are inverses of some maps in $G_{\overline{C(X,L)}} $,  the above diagram is commutative, proving the result.
\end{proof}
\begin{Rmk}
    Henceforth, we will use the same notation for objects in $C(X,L)$ and $\overline{C(X,L)}$
\end{Rmk}

\section{Khovanov Floer Theories}

We now briefly discuss Khovanov-Floer theories (henceforth KFTs) following the work in \cite{b18}. Since all of the work is done over $\mathbb{F}_2$ in \cite{b18}, we need to check that we can get similar theories over $\mathbb{Q}$ (or any other field of characteristic zero). We note that this hardly changes the proofs given for results in \cite{b18} but rather restricts the class of examples we can work with (since most of these known spectral sequences are over $\mathbb{F}_2$). However, there is a great deal of work in progress to use other functorial fixes of Khovanov homology to extend these spectral sequences over $\mathbb{Q}$. The reader can directly jump to the next section if they would like to see the construction of the spectral sequence of skein modules while assuming the functoriality of the base KFT.

First, choose a version of Khovanov homology that gives a well-defined functor on $\mathbb{k}$ modules, where $\mathbb{k}$ has characteristic zero. This may be any functorial fix, e.g. the webs-and-foams theory of \cite{kw21} or Khovanov-Rozansky homology in \cite{kr04}. We will denote these by $Kh_{\bullet}$. Note that here, we require functoriality with respect to cobordisms in $S^3$; hence the sweep-around move must give the identity map on the Khovanov homology.
\begin{Def}
For a graded vector space $V$, a $\textbf{V}$\textbf{-complex} is a pair $(C,q)$, where $C$ is a filtered chain complex with an isomorphism $q:V\rightarrow E^{2}(C)$. Let $(C,q)$ and $(C^{'},q^{'})$ be $V$- and $W$-complexes, respectively. For a homogeneous degree-$k$ morphism $T:V\rightarrow W$, a map of $V$- and $W$- complexes that \textbf{agrees with $T$} is a degree $k$ map $f$ of filtered chain complexes so that $E^{2}({q^{'}}^{-1} \circ f \circ q)=T$. Such a map is called a quasi-isomorphism if it agrees with $\text{Id}_{V}$.
\end{Def}

Now we list certain algebraic results proved in \cite{b18}. These are only stated for $\mathbb{F}_2$, but the proofs of these results depend on the "cancellation lemma," which is true over $\mathbb{Q}$ (or any field of characteristic zero) as well. First, we give a precise statement of the cancellation lemma.

\begin{Lem}[Cancellation Lemma]
\label{lem:cancel}
Let $(C,d)$ be a chain complex of finite-dimensional vector spaces with a fixed
ordered basis $\{x_1,\dots,x_n\}$.  Write
\[
d(x_i) = \sum_{j=1}^n d_{ij}\,x_j .
\]
Suppose that for some distinct indices $k\neq \ell$ we have
\[
d_{k\ell} \neq 0 .
\]
Then $(C,d)$ is chain homotopy equivalent to a complex $(C',d')$ defined as
follows:
\begin{itemize}
\item $C'=\mathrm{span}\{x_i \mid i\neq k,\ell\}$,
\item for $i\neq k,\ell$,
\[
d'(x_i)
\;=\;
\sum_{j\neq k,\ell}
\left(d_{ij} - d_{i\ell}\, d_{k\ell}^{-1}\, d_{kj}\right)x_j .
\]
\end{itemize}

\end{Lem}

As a consequence of \ref{lem:cancel}, we get the following results.

\begin{Lem} \label{ij}
   Let $f$ be a degree-zero map of filtered chain complexes such that $E^{i}(f)$ is an isomorphism. Then $E^{j}(f)$ is also an isomorphism $\forall j 
   \geq i$. This indeed implies that for a morphism of spectral sequences, an isomorphism on a page gives an isomorphism on all higher pages.
\end{Lem}
\begin{Lem} \label{fg}
   For two morphisms $f,g$ of filtered chain complexes, if $E^{i}(f)=E^{i}(g)$ then $E^{j}(f)=E^{j}(g)$, $\forall j\geq i$ 
\end{Lem}
\begin{Lem} \label{lem:collase}
Suppose $(C,d)$ is a filtered chain complex such that $E^i(C)=E^{\infty}(C)$. Then there is a degree zero map of filtered chain complexes $$\pi : (C,d)\rightarrow(E^i(C),0)$$  that induces the identity map on the $i$-th page: $$E^i(\pi):E^i(C)\rightarrow E^i(C)=\text{Id}_{E^i(C)}$$
\end{Lem}
When $(C,d)$ is a filtered chain complex over a field $\mathbb{k}$ of characteristic zero, the corresponding spectral sequence has a description in terms of the cancellation lemma. 
Inductively we have :
\begin{itemize}
    \item Let $(C_{(0)},d_{(0)})$ be $(C,d)$
    \item $(C_{(i)},d_{(i)})$ is obtained from $(C_{(i-1)},d_{(i-1)})$ by canceling the part of $d_{(i-1)}$ that shifts the filtration degree exactly by $(i-1)$
    \item Then the corresponding spectral sequence $(E^i,d^i)$ consists of $E^i \coloneqq C_i$ with grading inherited from $C$, while $d_i$ is the sum of components of $d_{(i)}$ that shift the filtration degree exactly by $i$.
\end{itemize}

We now define Khovanov-Floer theories in a way that will give us the right framework to work with skein lasagna modules. This definition is almost identical to the one given by \cite{b18}.
\begin{Def} \label{KFT}
    A Khovanov-Floer theory $\mathcal{A}$ is a rule that assigns to a link diagram $\mathcal{D}$, a $Kh_{\bullet}(\mathcal{D})$-complex $\mathcal{A}(\mathcal{D})$ that satisfies the following properties:
\begin{enumerate}
    \item For a planar isotopy taking  $\mathcal{D}$ to $\mathcal{D^{'}}$ , we get a morphism $\mathcal{A}(\mathcal{D}) \rightarrow \mathcal{A}(\mathcal{D^{'}})$, that agrees on the $E^2$ page with the isomorphism on the map on $Kh_{\bullet}$. (Note that Lemma \ref{ij} implies that this map gives an isomorphism on all higher pages.)
    \item For $\mathcal{D^{'}}$ obtained from $\mathcal{D}$ by 1-handle attachment, then it induces a map on $\mathcal{A}$ that agrees on the $E^2$ page with the map on $Kh_{\bullet}$.
    \item For disjoint union of diagrams $\mathcal{D} \sqcup \mathcal{D^{'}}$, we have a morphism $\mathcal{A}(\mathcal{D} \sqcup \mathcal{D^{'}}) \rightarrow \mathcal{A}(\mathcal{D}) \otimes  \mathcal{A}(\mathcal{D^{'}})$ that agrees on the $E^2$ page with the standard isomorphism on $Kh_{\bullet}$. (This implies isomorphism on all higher pages as well).
    \item For every diagram $\mathcal{U_{n}}$ of the unlink (with n components), $E^{2}(\mathcal{U_{n}})=E^{\infty}(\mathcal{U_{n}})$.
    \item For the empty link, we have $\mathcal{A}(\emptyset)=\mathbb{k}$ considered as a complex centered at degree zero.
\end{enumerate}
\end{Def}

\begin{Def}
    A Khovanov-Floer theory is called \textbf{functorial} if for every movie $M$, we have a map on the complexes $ F_{M}:\mathcal{A}(\mathcal{D}) \rightarrow \mathcal{A}( \mathcal{D^{'}})$ that agrees on the $E^2$ page with the map on $Kh_{\bullet}$.
    
\end{Def}
We now state a theorem from \cite{b18}, which allows us to view each page of a KFT as a well-defined monoidal link homology, thereby providing a necessary framework for defining skein lasagna modules. We emphasize, however, that in the present context we work over an arbitrary field.

\begin{Thm}
    For a Khovanov-Floer theory $\mathcal{A}$, for every Reidemeister move, there is a map $Kh_{\bullet}(D)$-complexes that agrees with the Reidemeister map on $Kh_{\bullet}$. Hence, equivalent movies induce the same map on the spectral sequence $(E^n,d^n)_{n \geq 2}$
\end{Thm}
\begin{Rmk}
    Note that Khovanov-Rozansky homology is a functorial theory for oriented links. Making resolutions orientation sensitive is a crucial aspect of making maps functorial without a sign ambiguity. However, we ignore the orientations for the following proof. For the oriented case, we simply need to consider all the Reidemeister moves of the oriented links. The proofs are identical in the oriented case.
\end{Rmk}
\begin{proof}

  We sketch the proof from \cite{b18}, noting that at each step, the ground field can be arbitrary. First, we show that each page satisfies invariance under Reidemeister moves. To do that, we sketch the realization of Reidemeister moves as the birth of an unknot and handle attachments following \cite{b18}. We have diagrammatically illustrated it in the figures \ref{fig:R1}, \ref{fig:R2} and \ref{fig:R3}. We realize the Reidemeister moves as a composition of the following maps (which, from the axioms of KFTs, we know are isomorphisms)
\begin{itemize}
    \item Birth of an unknot disjoint from the link.
     \item Reidemeister moves of the unlink to go to a different diagram of the unlink. Note that since the spectral sequence for any diagram of the unknot collapses on the 2nd page, this map is an isomorphism and is the map on $Kh_{\bullet}$
     \item Diagrammatic 1-handle attachment.
     \item Connect sum map and its inverse, which represent isomorphisms on all pages.
 \end{itemize}
 Now, by Lemma \ref{lem:collase}, for any diagram $U$ of the unlink, the corresponding $Kh_{\bullet}(U)$-complex is in the quasi-isomorphism class of $(E^2(U),0)$. Hence for every Reidemeister move of the unlink, there is a map of $Kh_{\bullet}$-complexes that agrees on the $E^2$ page with the maps on $Kh_{\bullet}$
 Finally, using \ref{fg} and \ref{ij} and the fact that the maps on $E^{2}$, equivalent movies give the same maps, we get that each page $E^{n}$ of a KFT is a functor from links in $\mathbb{R}^3$ to modules.
\end{proof}
   
\begin{figure}[ht] 

\centering
\includegraphics[height=5.7cm]{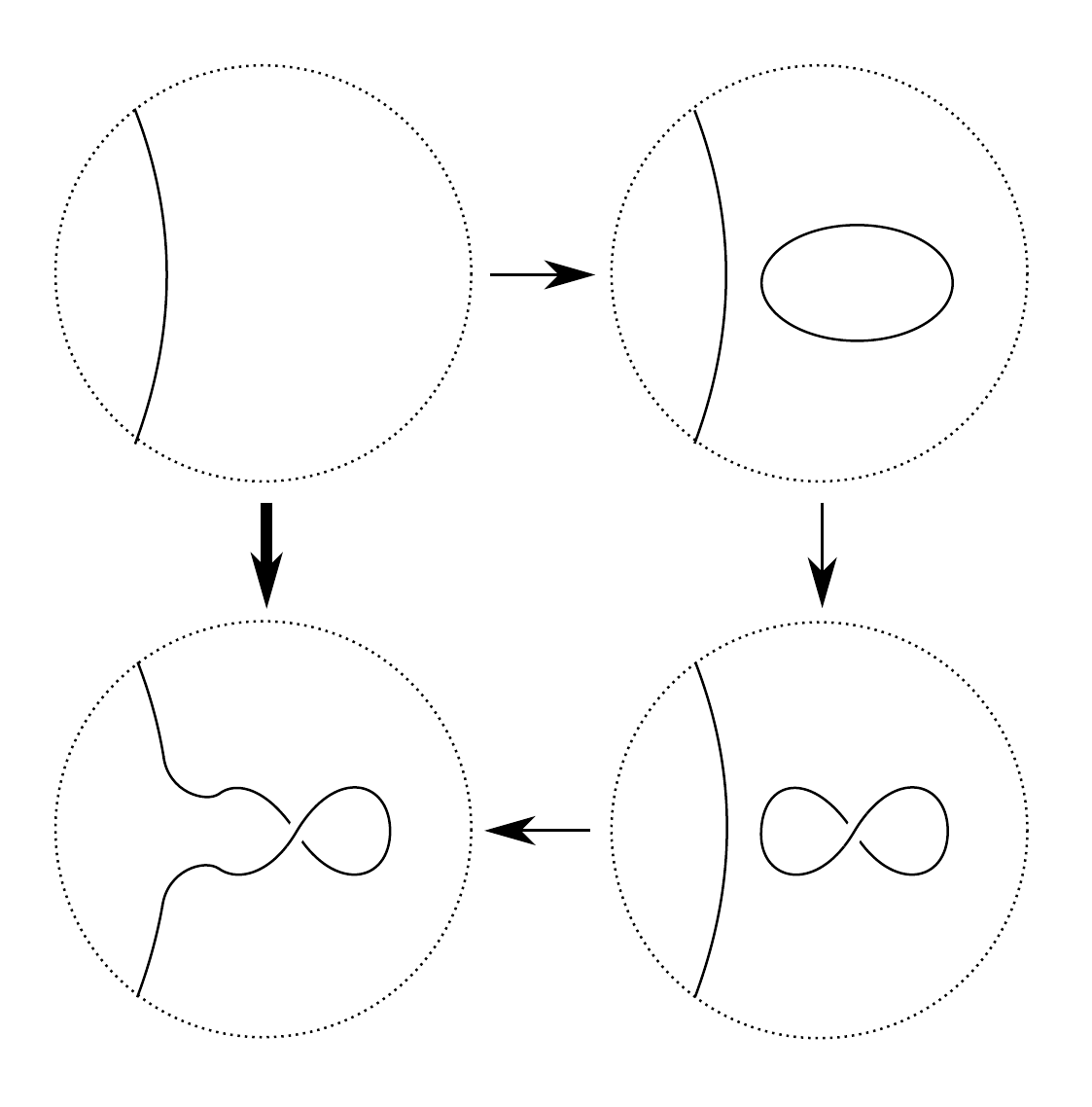}
\caption{ Reidemeister 1 move realized as a birth followed by handle attachments}
\label{fig:R1}
\end{figure}

\begin{figure}[ht] 

\centering
\includegraphics[height=6cm]{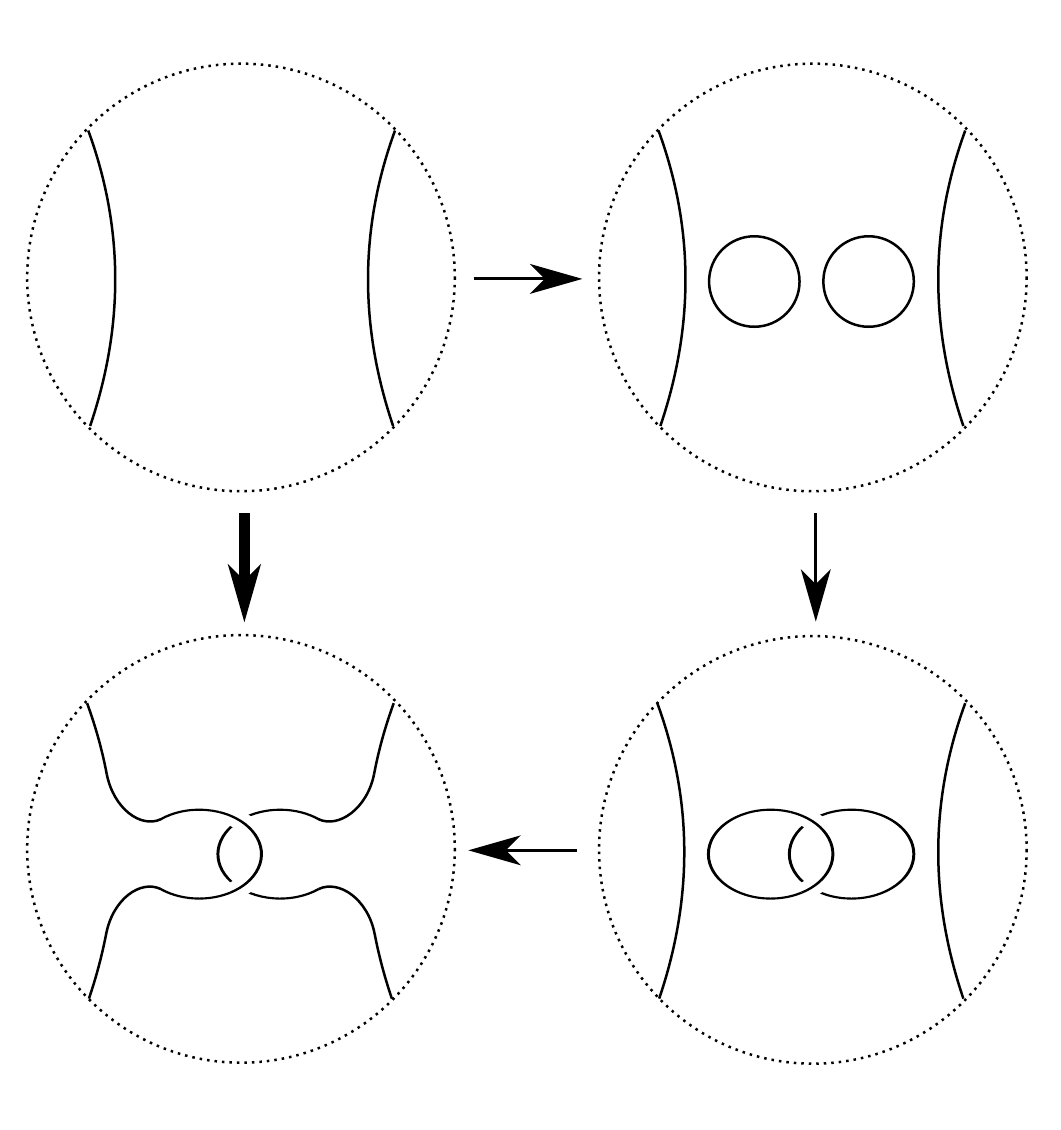}
\caption{ Reidemeister 2 move realized as a birth followed by handle attachments}
\label{fig:R2}
\end{figure}

\begin{figure}[ht] 

\centering
\includegraphics[height=6 cm]{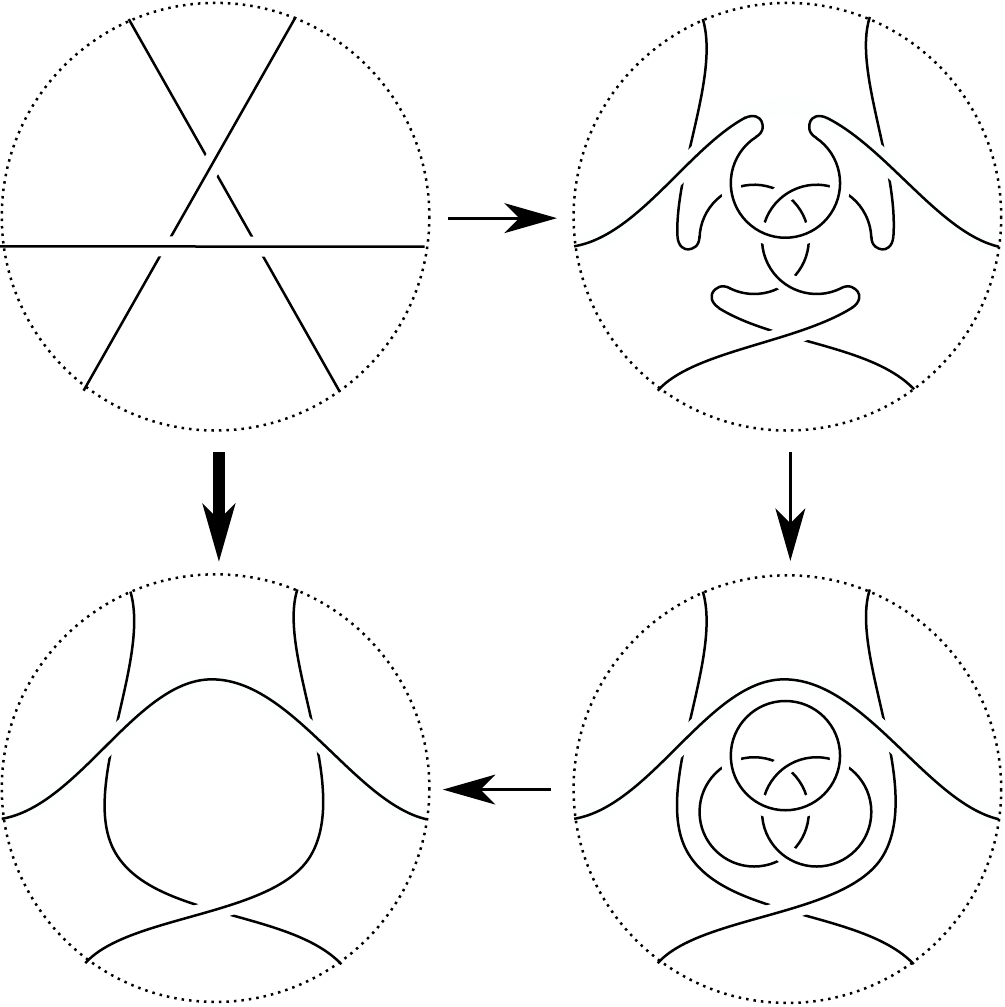}
\caption{ Reidemeister 3 move, realized as a composition of Reidemeister 2 moves with a death of the (3-component) unlink }
\label{fig:R3}
\end{figure}

Note that the first arrow in Figure \ref{fig:R2} is the Reidemeister 2 move proved beforehand.

Hence, we have the following corollary:
\begin{Cor}
    Every Khovanov-Floer theory is functorial.
\end{Cor}
\begin{Cor}
    Every page of a KFT defines a functor  \\
\[
\begin{Bmatrix}
\textrm{link embeddings in oriented } \mathbb{R}^3
\\
\textrm{link cobordisms in oriented }
Y \cong \R^3\times \mathcal{I} 
\\
\textrm{ up to isotopy rel } \partial
\end{Bmatrix}
\xrightarrow{Z}
\begin{Bmatrix}
\ \mathbb{k}\textrm{-vector spaces} 
\\
\ ( \textrm{possibly} \ \mathbb{Z}\textrm{-graded, } \Gamma\textrm{-graded/filtered }  
\\
\   \ \mathbb{Z}\textrm{-pres., } \Gamma\textrm{-hom./filt.} R\textrm{-morphisms})
\end{Bmatrix}
\]
\end{Cor}

Before proceeding, we define the action of the braid group $B_n$ on a given TQFT $Z$ of a cable of a knot $K$ as follows: \par
First, let us familiarize ourselves with the \textbf{tangle notation for cobordisms between cables of a knot $K$}.  The process is depicted in the figure \ref{fig:st}. First, cross the tangle diagram with a circle. Hence, a tangle in $D^2 \times I $ induces a cobordism between links in a solid torus, which, by using the framing of a knot $K$, induces a cobordism between cables. Finally, let $[\sigma_i]_n$ denote the map induced by the cobordism on the n-th page.
For a generator $\sigma_i$ of the braid group switching adjacent strands $(i)$ and $(i+1)$, consider the cobordism as depicted (on the right side of the arrow) in the figure\ref{fig:st}. One can imagine this cobordism to be ``switching" $(i)-$th and $(i+1)-$strands of the cable. 

\begin{figure}
\includegraphics[height=2in]{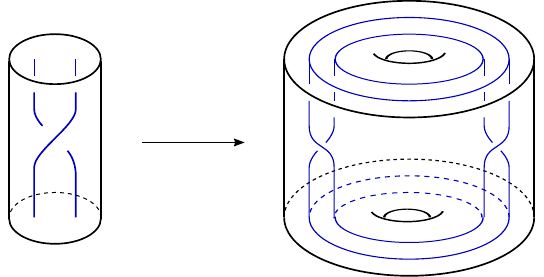}
\caption{\label{fig:st}Schematic depiction of considering tangle $\sigma_i$ (left) and then using revolution to get a cobordism (right) on the cable.}
\end{figure}

We also consider $Z(i)$: the cobordism that forms an annulus between the i-th and the i+1-th strand. Let $Z^{r}(i)$ be its reverse cobordism. In tangle notation depicted in figure \ref{fig:st}, it is given by a cup and a cap.

\begin{figure}
\centerline{\includegraphics[height=1.4in]{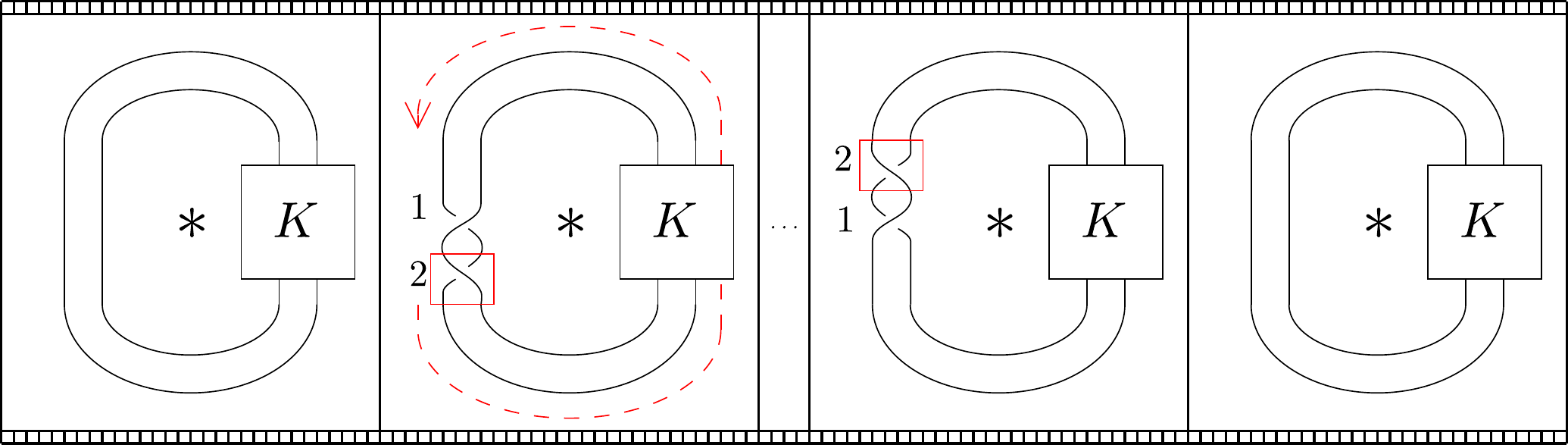}}
\caption{\label{fig:Braid_Movie_Knot}Movie moves of the cobordism depicted in figure \ref{fig:st}}
\end{figure}
Now, we list all the desired properties of Khovanov-Floer theories as Theorem \ref{thm:properties}. This will let us define skein modules for each page of a KFT, as well as provide us with properties to construct spectral sequences of skein modules.
\begin{Thm} \label{thm:properties} Khovanov-Floer Theories satisfy the following properties
\begin{enumerate}  
   
    \item  The map $E^{n}(Z(i))$ is injective and $E^{n}(Z^{r}(i))$ is surjective.  \label{injsurj}
    \item Let $[\sigma_i]_n$ be the map induced on the nth page, by the cobordism given by an element of the braid group that switches the $i$ and $ i+1$st strand. 
    Then 
    \begin{equation}
    [\sigma_{i}]_{n}(v)= v - E^{n}(Z(i))\circ E^{n}(Z^{r}(i))(v) =[{\sigma_{i}^{-1}}]_{n}   \label{eq:braid}
    \end{equation}
    \item For every $n$, $Kh_{\bullet}(\emptyset)=E^{n}(\emptyset)=\mathbb{k}$, $Kh_{\bullet}(U)=E^{n}(U)= \mathbb{k}[X]/(X^2)$. 
    \item   The following cobordisms induce the same map on each page $E^n$, hence the same as $Kh_{\bullet}$ \label{disc}
\begin{equation} 
	\begin{tikzpicture} [fill opacity=0.2,anchorbase, scale=.375]
	\path[fill=blue, opacity=.2] (1,0) arc[start angle=0, end angle=180,x radius=1,y radius=.5] 
		to (-1,4) arc[start angle=180, end angle=0,x radius=1,y radius=.5] to (1,0);
	\path[fill=blue, opacity=.2] (1,0) arc[start angle=360, end angle=180,x radius=1,y radius=.5] 
		to (-1,4) arc[start angle=180, end angle=360,x radius=1,y radius=.5] to (1,0);
	\draw [very thick] (0,4) ellipse (1 and 0.5);
	\draw [very thick] (0,0) ellipse (1 and 0.5);
	\draw[very thick] (1,4) -- (1,0);
	\draw[very thick] (-1,4) -- (-1,0);
\end{tikzpicture}
\,\ \ \ \ 
\begin{tikzpicture} [fill opacity=0.2,anchorbase, scale=.375,rotate=180]
	\path[fill=blue,opacity=.2] (1,4) arc[start angle=0, end angle=180,x radius=1,y radius=.5] 
		to [out=270,in=180] (0,2.5) to [out=0,in=270] (1,4);
	\path[fill=blue,opacity=.2] (1,4) arc[start angle=360, end angle=180,x radius=1,y radius=.5] 
		to [out=270,in=180] (0,2.5) to [out=0,in=270] (1,4);
	\draw[very thick] (0,4) ellipse (1 and 0.5);
	\draw[very thick] (-1,4) to [out=270,in=180] (0,2.5) to [out=0,in=270] (1,4);
    \end{tikzpicture}
    \,\ \ \ \ 
    \begin{tikzpicture} [fill opacity=0.2,anchorbase, scale=.375,rotate=180]
	\path[fill=blue,opacity=.2] (1,0) arc[start angle=0, end angle=180,x radius=1,y radius=.5] 
		to [out=90,in=180] (0,1.5) to [out=0,in=90] (1,0);
	\path[fill=blue,opacity=.2] (1,0) arc[start angle=360, end angle=180,x radius=1,y radius=.5] 
		to [out=90,in=180] (0,1.5) to [out=0,in=90] (1,0);
	\draw[very thick] (0,0) ellipse (1 and 0.5);
	\draw[very thick] (-1,0) to [out=90,in=180] (0,1.5) to [out=0,in=90] (1,0);
	\node[opacity=1] at (0,1) {\footnotesize$\bullet$};
\end{tikzpicture}
\end{equation}
    \item Each $E^n$, satisfies the relations given by the Figure \ref{fig:cobordismrelations}.

\end{enumerate}
\end{Thm}

\begin{equation} \label{fig:cobordismrelations}
	\begin{tikzpicture} [fill opacity=0.2,anchorbase, scale=.375]
	\path[fill=blue, opacity=.2] (1,0) arc[start angle=0, end angle=180,x radius=1,y radius=.5] 
		to (-1,4) arc[start angle=180, end angle=0,x radius=1,y radius=.5] to (1,0);
	\path[fill=blue, opacity=.2] (1,0) arc[start angle=360, end angle=180,x radius=1,y radius=.5] 
		to (-1,4) arc[start angle=180, end angle=360,x radius=1,y radius=.5] to (1,0);
	\draw [very thick] (0,4) ellipse (1 and 0.5);
	\draw [very thick] (0,0) ellipse (1 and 0.5);
	\draw[very thick] (1,4) -- (1,0);
	\draw[very thick] (-1,4) -- (-1,0);
\end{tikzpicture}
\, = \,
\begin{tikzpicture} [fill opacity=0.2,anchorbase, scale=.375,rotate=180]
	\path[fill=blue,opacity=.2] (1,4) arc[start angle=0, end angle=180,x radius=1,y radius=.5] 
		to [out=270,in=180] (0,2.5) to [out=0,in=270] (1,4);
	\path[fill=blue,opacity=.2] (1,4) arc[start angle=360, end angle=180,x radius=1,y radius=.5] 
		to [out=270,in=180] (0,2.5) to [out=0,in=270] (1,4);
	\draw[very thick] (0,4) ellipse (1 and 0.5);
	\draw[very thick] (-1,4) to [out=270,in=180] (0,2.5) to [out=0,in=270] (1,4);
	\path[fill=blue,opacity=.2] (1,0) arc[start angle=0, end angle=180,x radius=1,y radius=.5] 
		to [out=90,in=180] (0,1.5) to [out=0,in=90] (1,0);
	\path[fill=blue,opacity=.2] (1,0) arc[start angle=360, end angle=180,x radius=1,y radius=.5] 
		to [out=90,in=180] (0,1.5) to [out=0,in=90] (1,0);
	\draw[very thick] (0,0) ellipse (1 and 0.5);
	\draw[very thick] (-1,0) to [out=90,in=180] (0,1.5) to [out=0,in=90] (1,0);
	\node[opacity=1] at (0,1) {\footnotesize$\bullet$};
\end{tikzpicture}
+
\begin{tikzpicture} [fill opacity=0.2,anchorbase, scale=.375]
	\path[fill=blue,opacity=.2] (1,4) arc[start angle=0, end angle=180,x radius=1,y radius=.5] 
		to [out=270,in=180] (0,2.5) to [out=0,in=270] (1,4);
	\path[fill=blue,opacity=.2] (1,4) arc[start angle=360, end angle=180,x radius=1,y radius=.5] 
		to [out=270,in=180] (0,2.5) to [out=0,in=270] (1,4);
	\draw[very thick] (0,4) ellipse (1 and 0.5);
	\draw[very thick] (-1,4) to [out=270,in=180] (0,2.5) to [out=0,in=270] (1,4);
	\path[fill=blue,opacity=.2] (1,0) arc[start angle=0, end angle=180,x radius=1,y radius=.5] 
		to [out=90,in=180] (0,1.5) to [out=0,in=90] (1,0);
	\path[fill=blue,opacity=.2] (1,0) arc[start angle=360, end angle=180,x radius=1,y radius=.5] 
		to [out=90,in=180] (0,1.5) to [out=0,in=90] (1,0);
	\draw[very thick] (0,0) ellipse (1 and 0.5);
	\draw[very thick] (-1,0) to [out=90,in=180] (0,1.5) to [out=0,in=90] (1,0);
	\node[opacity=1] at (0,1) {\footnotesize$\bullet$};
\end{tikzpicture}
\, , \quad
\begin{tikzpicture}[anchorbase, scale=.375]
	\path [fill=blue,opacity=0.3] (0,0) circle (1);
	\draw (-1,0) .. controls (-1,-.4) and (1,-.4) .. (1,0);
	\draw[dashed] (-1,0) .. controls (-1,.4) and (1,.4) .. (1,0);
	\draw[very thick] (0,0) circle (1);
\end{tikzpicture}
= 0
\, , \quad
\begin{tikzpicture}[anchorbase, scale=.375]
	\path [fill=blue,opacity=0.3] (0,0) circle (1);
	\draw (-1,0) .. controls (-1,-.4) and (1,-.4) .. (1,0);
	\draw[dashed] (-1,0) .. controls (-1,.4) and (1,.4) .. (1,0);
	\draw[very thick] (0,0) circle (1);
	\node at (0,0.6) {\footnotesize$\bullet$};
\end{tikzpicture}
= 1
\, , \quad
\begin{tikzpicture}[fill opacity=.3, scale=.5, anchorbase]
	\filldraw [very thick,fill=blue] (-1,-1) rectangle (1,1);
	\node [opacity=1] at (0,-.25) {$\bullet$};
	\node [opacity=1] at (0,.25) {$\bullet$};
	\end{tikzpicture}
= 0 \, .
\end{equation}

\begin{proof}
    First, to prove \ref{injsurj}, note that 
    \begin{equation*}
        E^{2}(Z^{r}(i))\circ E^{2}(Z(i))= Kh_{\bullet}(Torus)=2 \ \text{Id}_{\mathbb{k}}
    \end{equation*}\footnote{It is important for our computation that the torus, when considered as a cobordism, induces an isomorphism on the TQFT. The actual isomorphism varies depending on which functorial version of Khovanov homology one uses.}
    Hence, by Lemma \ref{ij}, we have for every $n$,
    \begin{equation*}
        E^{n}(Z^{r}(i))\circ E^{n}(Z(i))= 2\mathbb{I}
    \end{equation*} 
    which is an isomorphism over $\mathbb{k}$ (since 2=1+1 $\in \mathbb{k}$ is invertible). Hence $E^{n}(Z^{r}(i))$ is surjective, while $E^{n}(Z(i)$ is injective.

Equation\ref{eq:braid} follows by considering the functorial morphism a KFT induces on its spectral sequence; combined with the fact that the equation is true for $Kh_{\bullet}$, again by Lemma \ref{ij}. This was originally proved in \cite{Grigsby_2017} for the annular Khovanov homology. We give an outline of the proof below in Theorem\ref{braidgrp}. 
Finally, to prove \ref{disc}, we use the last two axioms of a KFT \ref{KFT}, hence for the mentioned maps, the spectral sequence collapses on the second page, and the claim follows. Every equation in Figure \ref{fig:cobordismrelations} is satisfied by $Kh_{\bullet}$, hence by Lemma \ref{ij} every $E^n$. 
\end{proof}

\begin{Thm}{(\cite{Grigsby_2017}, Proposition 9)} \label{braidgrp}
 \begin{equation}
    [\sigma_{i}]_{*}(v)= v \pm Kh_{\bullet} (Z(i))\circ Kh_{\bullet} (Z^{r}(i))(v) =[{\sigma_{i}^{-1}}]_{*}(v) \footnote{The sign on the RHS in this formula changes according to which functorial version of the Khovanov homology one uses. Consistency is more important. In our convention, the torus cobordism represents multiplication by "2", hence our sign is negative. One would get a positive sign if $KhR_2$ is used.}
 \end{equation} 

\end{Thm}
\begin{proof}[Sketch of Proof]
    We start by considering the movie moves given by cobordisms corresponding to the elements of the braid group as in\ref{fig:Braid_Movie_Knot}.  
\end{proof}
Note that the ``$\dots$" in the Figure \ref{fig:Braid_Movie_Knot} and the Figure \ref{fig:TL_Movie_Knot} correspond to Reidemeister moves 3 and 2 respectively. 
To prove the claim we consider the following resolution:

\[
\xymatrix@R=-0.05in@C=0.4in{
&*+={\includegraphics[height=0.6in]{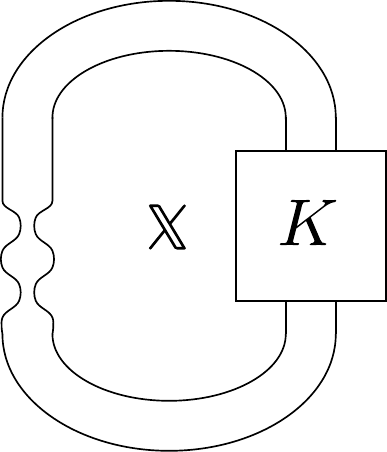}}\ar[r]^{\operatorname{id}}&*+<0.5in>{\cdots}\ar[r]^{\operatorname{id}}&*+={\includegraphics[height=0.6in]{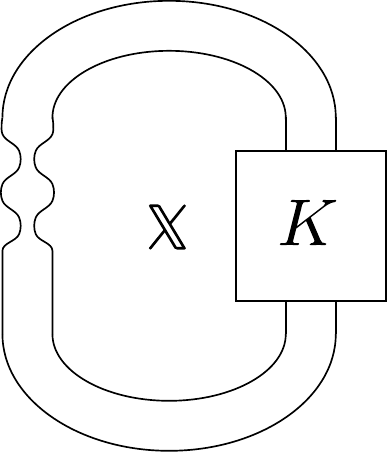}}\ar[dr]^{\operatorname{id}}&\\
*+={\includegraphics[height=0.7in]{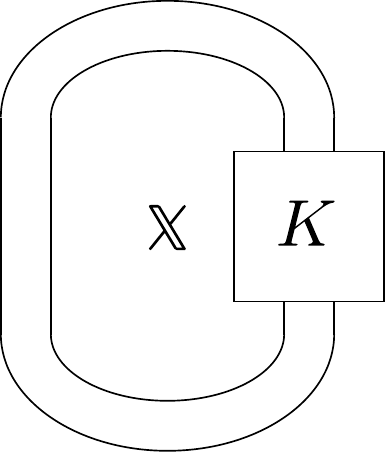}}\ar[ur]^{\operatorname{id}}\ar[dr]_{f}&\oplus&&\oplus&*+={\includegraphics[height=0.6in]{figures/sBraid_Movie_Knot_Chain_Map1n_X.pdf}}\\
&*+={\includegraphics[height=0.6in]{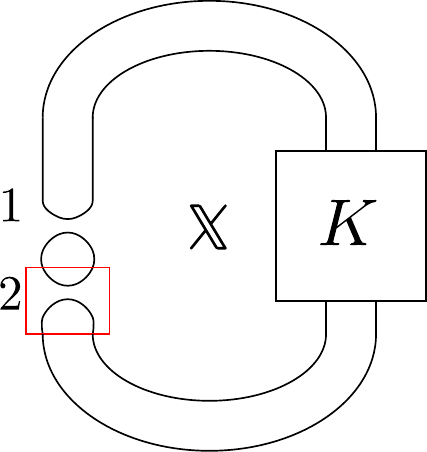}}\ar[r]_{\psi_1}\ar[uur]^{\nu_1}&*+<0.5in>{\cdots}\ar[r]_{\psi_\ell}\ar[uur]^{\nu_\ell}&*+={\includegraphics[height=0.6in]{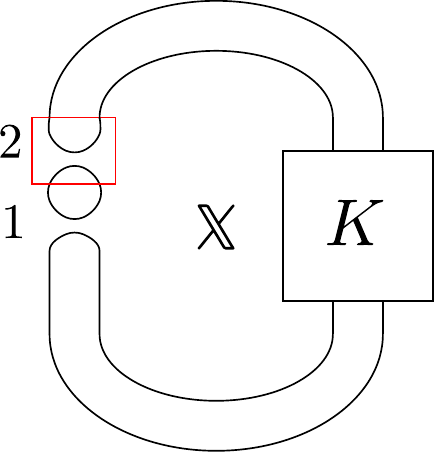}}\ar[ur]_g&
}
\]

\begin{figure}
\centerline{\includegraphics[height=1.4in]{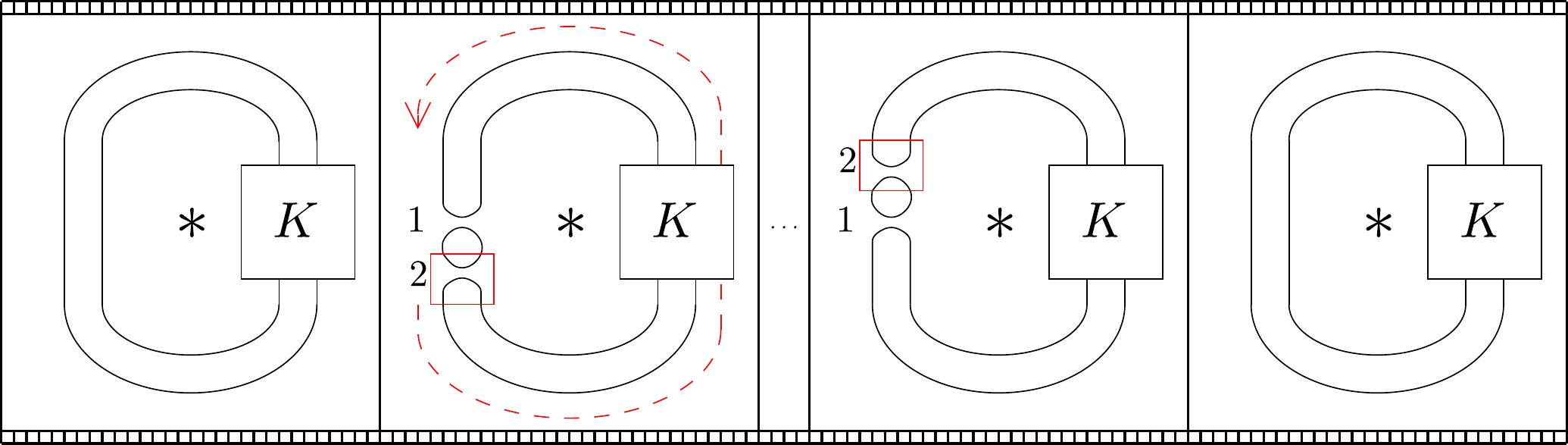}}
\caption{\label{fig:TL_Movie_Knot}Movie for $Z(i)$ followed by $Z(i)^r$}
\end{figure}

\section{Spectral Sequence of Skein Modules}

In this section, we use a functorial spectral sequence associated to a KFT: $\ (E^{n},d_n)$, with $E^{2}=Kh_{\bullet}$ to construct skein lasagna modules corresponding to each page. We will define a differential on these lasagna modules and address the existence of spectral sequences. 
To define lasagna modules, we need a TQFT as in the definition \ref{def:s3tqft}. However, most of the TQFTs we know, including different versions of Khovanov homology, are defined only for links in the standard $S^3$ (more accurately, a fixed copy of a manifold diffeomorphic to $S^3$). However, we need TQFTs defined over an abstract $S^3$, to be able to define skein lasagna modules. This issue is cleverly fixed for Khovanov-Rozansky homologies in \cite{Morrison_2022}. Their construction is stated as the following theorem in \cite{morrison2024invariantssurfacessmooth4manifolds}. This result sets up an axiomatic framework which allows us to jump between a standard and an abstract TQFT.

\begin{Thm}[\cite{morrison2024invariantssurfacessmooth4manifolds}, Theorem 2.1]\label{thm:skeinfromLH} Let $R$ be a commutative ring, optionally graded resp. filtered by
an abelian group $\Gamma$. Given a functorial \emph{link homology theory} for
links in $\R^3$ with values in $R$, i.e. a functor:
\[
\begin{Bmatrix}
\textrm{link embeddings in oriented } \R^3
\\
\textrm{link cobordisms in oriented }
Y \cong \R^3\times \mathcal{I} \textrm{ up to isotopy rel } \partial
\end{Bmatrix}
\xrightarrow{Z}
\begin{Bmatrix}
\Z\textrm{-graded, } \Gamma\textrm{-graded/filtered }  R\textrm{-modules}
\\
\Z\textrm{-pres., } \Gamma\textrm{-hom./filt.} R\textrm{-morphisms}
\end{Bmatrix}
\]
which additionally
\begin{enumerate}
  \item is (lax) monoidal under disjoint union: $Z(L_1)\otimes_R Z(L_2)\to Z(L_1 \sqcup L_2)$, and
  \item satisfies the sweep-around move , and
  \item the trace of the $2\pi$ rotation of $\R^3$, which generates
  $\pi_1(SO(3))$, acts by the identity on $H$, 
\end{enumerate}
then $H$ extends to:
\begin{itemize}
  \item a link homology for links in $3$-spheres (as in definitions \ref{def:s3tqft} and \ref{def:TQFTextends})  with
  values in $\Z$-graded, $\Gamma$-graded/filtered $R$-modules,
  \item an algebra for the lasagna operad  valued in
  $\Z$-graded, $\Gamma$-graded/filtered $R$-modules,
  \item a $(4+\epsilon)$-dimensional TQFT of skein modules associated to pairs
  $(W,L)$ of oriented $4$-manifolds $W$ with links $L\subset \partial W$, valued
  in  $\Z$-graded, $\Gamma$-graded/filtered $R$-modules.
\end{itemize}
\end{Thm}

We know that on $E^2=Kh_{\bullet}$, the map induced by the sweep-around move gives the identity morphism. Hence, from Lemma \ref{fg} it follows that on every page, the sweep-around move induces the identity morphism.

It follows that for a given Khovanov-Floer theory $\mathcal{A}$, each page of the associated spectral sequence is a  TQFT as in the definition \ref{def:TQFT}. Hence, all the necessary axioms to be able to define a lasagna module are satisfied. Now, we will prove that all these pages satisfy an analogue of the 2-handlebody formula \ref{prop:2_hdby}. Henceforth, we will use the notation $K(r)$ to denote the cable $K(\alpha^{+}+r,\alpha^{-}+r)$ of a framed link $K$.

\begin{Prop} \label{thm:2hdlbdy}
Let $(X,L)$ be a 2-handlebody with a link $L\subset \partial X$ away from the 2-handles and $\alpha \in H_{2}(X,L)$, let $2HC \subset \overline{C(X,L)}$ \footnote{This is where one can see the advantage of working with the quotient category. The braid group morphisms can be defined as self-automorphisms, just as in the case of the cabled Khovanov homology. Intuitively, we are saying that along with the surface of the skein, the input balls are also flexible up to isotopy. } be the following subsystem, 
 \begin{enumerate}
    \item  Skeins of the form $\Sigma_{K(r)}$ made with one input ball with input link $K(\alpha^++\textbf{r},\alpha^-+\textbf{r})$, and the surface made with discs parallel to the cores of handle attachments, along with a copy of $L\times I$. \label{special_filling}
    These are depicted in the Figure \ref{fig:generator} (\cite{mn22}, Figure 1)
    \item Morphisms  \begin{equation}
                          S_{\sigma_{i}}:[\Sigma_{K(r)}] \rightarrow [\Sigma_{K(r)}]
                     \end{equation}
          given by the generators $\sigma_i$ of the braid group on $|\alpha|+2r$ strands. These correspond to the cobordism in \ref{fig:st}.
    \item  Morphisms $S_{Z^{\bullet}(r)}:[\Sigma_{K(r)}] \rightarrow [\Sigma_{K(r+1)}]$ given by adding a dotted annulus between two newly introduced strands
 \end{enumerate} \footnote{Some of the morphisms mentioned above are not really self-automorphisms. They become self-morphisms after post composing with maps in $ G_{\overline{C(X,L)}}$. For now, we abuse the notation for convenience. }
Then the corresponding system $2HC^{conv}$ is pseudo-final in $\overline{C(X,L)}^{conv}$. 

\begin {figure}
\begin {center}
\begin{picture}(0,0)%
\includegraphics{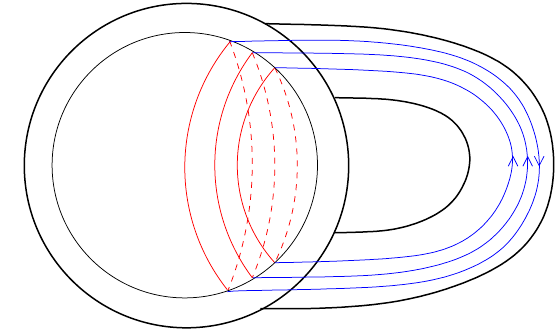}%
\end{picture}%
\setlength{\unitlength}{3158sp}%
\begingroup\makeatletter\ifx\SetFigFont\undefined%
\gdef\SetFigFont#1#2#3#4#5{%
  \reset@font\fontsize{#1}{#2pt}%
  \fontfamily{#3}\fontseries{#4}\fontshape{#5}%
  \selectfont}%
\fi\endgroup%
\begin{picture}(5559,3272)(1753,-7176)
\put(6384,-5903){\makebox(0,0)[lb]{\smash{{\SetFigFont{10}{12.0}{\rmdefault}{\mddefault}{\updefault}{\color[rgb]{0,0,1}$C_j^{i,\pm}$}%
}}}}
\put(6278,-4231){\makebox(0,0)[lb]{\smash{{\SetFigFont{10}{12.0}{\rmdefault}{\mddefault}{\updefault}{\color[rgb]{0,0,0}$2$-handle}%
}}}}
\put(2508,-6201){\makebox(0,0)[lb]{\smash{{\SetFigFont{10}{12.0}{\rmdefault}{\mddefault}{\updefault}{\color[rgb]{0,0,0}$B$}%
}}}}
\put(1768,-6951){\makebox(0,0)[lb]{\smash{{\SetFigFont{10}{12.0}{\rmdefault}{\mddefault}{\updefault}{\color[rgb]{0,0,0}$0$-handle}%
}}}}
\put(3058,-4771){\makebox(0,0)[lb]{\smash{{\SetFigFont{10}{12.0}{\rmdefault}{\mddefault}{\updefault}{\color[rgb]{1,0,0}$K(2,1)$}%
}}}}
\end{picture}%

\caption {Skein of the form $\Sigma_{K(r)}$.}
\label{fig:generator}
\end {center}
\end {figure}
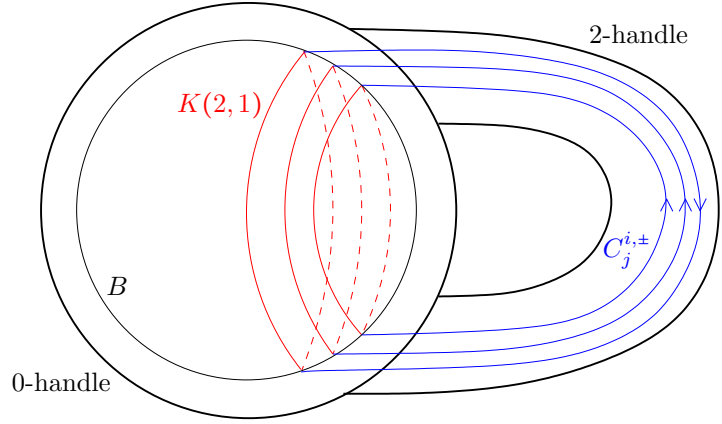

\end{Prop}

An immediate corollary is the 2-handlebody formula as stated in \cite{mn22}
\begin{Cor}
    \begin{equation} \label{eq:fcolim}
    S^{E^{n}}(X,L) \cong {\text{Fcolim}}_{r\rightarrow \infty }{\text{colim}} _{S_{|\alpha|+2r}}E^{n}(K(r))  \footnote{Note that ${\text{colim}} _{S_{|\alpha|+2r}}E^{n}(K(r))$ contains an abuse of notation, simply indicating that the system over which we are considering the colimit has a single object, along with morphisms given by the symmetric group. }
 \end{equation} 
\end{Cor}
\begin{proof}
The proof follows from the various results in Section 3.

 \begin{align*} \label{eq:fcolim}
    S^{E^{n}}(X,L) &\cong {\text{colim}}_{[\Sigma]\in C(X,L)} E^{n} \\ &\cong {\text{colim}}_{[\Sigma]\in \overline{C(X,L)}} E^{n} \\ & \cong {\text{colim}}_{[\Sigma]\in {\overline{C(X,L)}}^{conv}} E^{n}([\Sigma]) \\ & \cong {\text{colim}}_{[\Sigma]\in 2HC^{conv}}E^{n}([\Sigma]) \\ & \cong{\text{colim}}_{[\Sigma]\in 2HC}E^{n}([\Sigma]) \\ & \cong {\text{Fcolim}}_{r\rightarrow \infty }{\text{colim}} _{S_{|\alpha|+2r}}E^{n}(K(r)) 
 \end{align*}
\end{proof}

\begin{proof}[\textbf{Proof of Theorem }\ref{thm:2hdlbdy}]

\par Fix a Kirby diagram for $X$, with 2-handles attached to a single $B^4$, along a framed link $K$.
Start with an element $[\Sigma]\in \overline{C(X,L)}$ and map it to a skein with smaller input balls using a map in $G_{\overline{C(X,L)}}$, so that none of the input balls intersect co-cores of the 2-handles. We will abuse notation and still call it $[\Sigma]$. Then consider a larger ball $B$, containing all the input balls of $[\Sigma]$ and all the non-trivial topology of the surface. We will get a morphism from $[\Sigma]$ to a new skein with $B$ as the only input ball. Then we consider a morphism given by the necessary isotopy, and obtain a surface transversal to the co-cores. Finally, consider a smaller ball and observe that the surfaces are transversal to the co-cores. Hence, we get a morphism from $[\Sigma]$ to a filling of the form $[\Sigma _{K(r)}]$ \footnote{The procedure mentioned above is similar in spirit to the one in \cite{mn22}. However, note that the skeins in $C(X,L)$ as well as $\overline{C(X,L)}$ , are not as flexible as the lasagna fillings in \cite{mn22} to be able to perform isotopies. Hence, we perform this procedure of adding and deleting input balls of a convenient size to be able to perform the required isotopy that makes the skein transversal to the co-core.}.
Now, let us see that any two such morphisms can be related by the morphisms mentioned in the proposition. It turns out that this cannot be achieved in $\overline{C(X,L)}$, but can be achieved in the convex closure. We can conveniently switch between $\overline{C(X,L)}$ and $\overline{C(X,L)}^{cov}$ due to Theorem \ref{thm:convexclosure}.
Away from the singularities \footnote{The term singularity is being used in the context of Theorem 1.1 in \cite{mn22}. To be precise, these singularities appear as critical points of a Morse function on a 3-manifold whose level sets are skeins in the isotopy}, these morphisms are simply related by the ``braid group action". At the singularities, we go from the filling on the left in Figure \ref{fig:Ew} to one on the right side. 

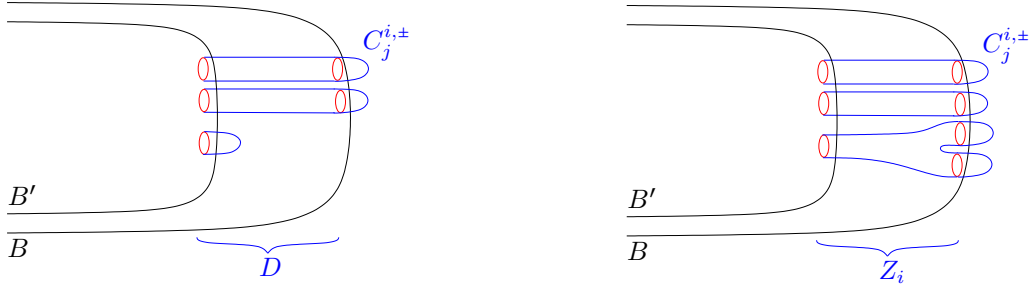
\begin {figure}
\begin {center}
\begin{picture}(0,0)%
\includegraphics{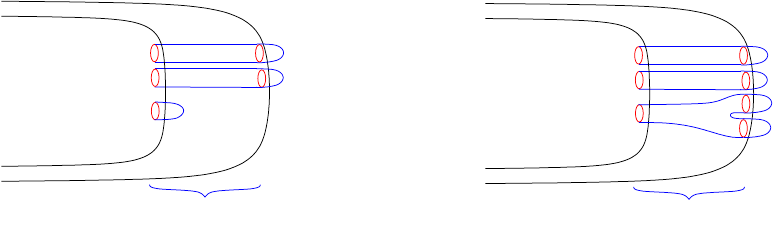}%
\end{picture}%
\setlength{\unitlength}{3158sp}%
\begingroup\makeatletter\ifx\SetFigFont\undefined%
\gdef\SetFigFont#1#2#3#4#5{%
  \reset@font\fontsize{#1}{#2pt}%
  \fontfamily{#3}\fontseries{#4}\fontshape{#5}%
  \selectfont}%
\fi\endgroup%
\begin{picture}(7728,2250)(889,-3049)
\put(5745,-2824){\makebox(0,0)[lb]{\smash{{\SetFigFont{10}{12.0}{\rmdefault}{\mddefault}{\updefault}{\color[rgb]{0,0,0}$B$}%
}}}}
\put(2870,-2948){\makebox(0,0)[lb]{\smash{{\SetFigFont{10}{12.0}{\rmdefault}{\mddefault}{\updefault}{\color[rgb]{0,0,1}$D$}%
}}}}
\put(3680,-1178){\makebox(0,0)[lb]{\smash{{\SetFigFont{10}{12.0}{\rmdefault}{\mddefault}{\updefault}{\color[rgb]{0,0,1}$C_{j}^{i, \pm}$}%
}}}}
\put(907,-2395){\makebox(0,0)[lb]{\smash{{\SetFigFont{10}{12.0}{\rmdefault}{\mddefault}{\updefault}{\color[rgb]{0,0,0}$B'$}%
}}}}
\put(904,-2801){\makebox(0,0)[lb]{\smash{{\SetFigFont{10}{12.0}{\rmdefault}{\mddefault}{\updefault}{\color[rgb]{0,0,0}$B$}%
}}}}
\put(7711,-2971){\makebox(0,0)[lb]{\smash{{\SetFigFont{10}{12.0}{\rmdefault}{\mddefault}{\updefault}{\color[rgb]{0,0,1}$Z_i$}%
}}}}
\put(8521,-1201){\makebox(0,0)[lb]{\smash{{\SetFigFont{10}{12.0}{\rmdefault}{\mddefault}{\updefault}{\color[rgb]{0,0,1}$C_{j}^{i, \pm}$}%
}}}}
\put(5748,-2418){\makebox(0,0)[lb]{\smash{{\SetFigFont{10}{12.0}{\rmdefault}{\mddefault}{\updefault}{\color[rgb]{0,0,0}$B'$}%
}}}}
\end{picture}%

\caption {A schematic of how two fillings differ at singularities. Here $C^{i}_{j}$ denote the discs as in \ref{fig:generator}}
\label{fig:Ew}
\end {center}
\end {figure}

Now, let us define the proper quotient category so that we can show that the $2HC^{conv}$ is pseudo-final.
First, we declare the following relation: for a sphere bounding a 3-ball in a given four-manifold, we declare 
\begin{equation*}
\begin{tikzpicture}[anchorbase, scale=.375]
	\path [fill=blue,opacity=0.3] (0,0) circle (1);
	\draw (-1,0) .. controls (-1,-.4) and (1,-.4) .. (1,0);
	\draw[dashed] (-1,0) .. controls (-1,.4) and (1,.4) .. (1,0);
	\draw[very thick] (0,0) circle (1);
	\node at (0,0.6) {\footnotesize$\bullet$};
\end{tikzpicture}
\cong 1 
\end{equation*}
To be more precise, for a skein given by surface $\Sigma \ \sqcup S $, where S is a contractible dotted sphere, let $\overline{\Sigma}$ be the surface formed by removing a collar neighborhood of $\partial \Sigma$. Then the aforementioned quotient is saying: 
$$\{[\overline{\Sigma}] \xrightarrow{\text{appropriate map in } G_{\overline{C(X,L)}}} [\Sigma] \xrightarrow{[S]} [\overline{\Sigma}] \} \cong \text{Id}_{[\overline{\Sigma}]}$$

Before we proceed, we introduce a new representation of skeins similar to the ones in Figure \ref{fig:Ew}. We mix the standard diagram for surfaces with the dTL notation introduced in \cite{hogancamp2022kirbycolorkhovanovhomology}. To understand these 1-dimensional objects, simply rotate the diagram 360 degrees to get the corresponding surface (dots are not rotated under this diagrammatic representation). This procedure is the same as the one depicted in figure \ref{fig:st}. For example under this notation the diagrams in Figure \ref{fig:Ew} can be represented as follows: 
\
\begin{equation}\label{eq:newdiagram}
 \begin{tikzpicture}[fill opacity=.3, scale= 1.5, anchorbase]
 \filldraw [very thick,fill=red] (-1,0) rectangle (1,1);

	\filldraw [very thick,fill=blue] (-1,-1) rectangle (1,0);
    \filldraw [very thick,fill=blue] (-1,-3) rectangle (0,-1);
     \filldraw [very thick,fill=blue] (0,-3) rectangle (1,-1);
     \node [opacity=1] at (0,-0.5) {$\text{Id}_{K(r)}$};
    \node [opacity=1] at (-0.5,-2) {$\text{Id}_{K(r)}$};
    \capfig{0.5}{4}{0.3}[thick]
\end{tikzpicture} 
\ \ \ \ \  \ \ \ \ \ \ \ \ \ \ \ \ \ \ \ \ 
\begin{tikzpicture}[fill opacity=.3, scale= 1.5, anchorbase]
    \filldraw [very thick,fill=red] (-1,0.5) rectangle (1,1);
    \filldraw [very thick,fill=blue] (-1,0) rectangle (1,0.5);
	\filldraw [very thick,fill=blue] (-1,-1) rectangle (1,0);
    \filldraw [very thick,fill=blue] (-1,-3) rectangle (0,-1);
     \filldraw [very thick,fill=blue] (0,-3) rectangle (1,-1);
     \node [opacity=1] at (0,-0.5) {$\text{Id}_{K(r+1)}$};
    \node [opacity=1] at (-0.5,-2) {$\text{Id}_{K(r)}$};
    \capfig{0.5}{3.5}{0.3}[thick]
   
     \node [opacity=1] at (0.5,-1.75) {$\#$};
     \draw[very thick] (0.25,-1) to [in=-90, out=-90] (.75,-1);
\end{tikzpicture} 
\end{equation}
Essentially, we replaced the annular creation map by the symbol ``$\cup$". The parts in pink represent the rest of the skein inside the 2-handles i.e., copies of disks parallel to the co-core. Below the lowest blue boxes, we have the unique input ball. So both of the skeins above have a unique input ball and boundary link: $K(\alpha^++r,\alpha^-+r)\ \cup \  \text{Unknot}$.

\par \textbf{Goal:}  We need to show that there is a way to go in between the following two maps by using maps in $2HC^{conv}$, so as to prove pseudo-finality as in the definition\ref{def:final}.
\begin{equation*}
\ \
\begin{tikzpicture}[scale=2, anchorbase]
\draw (0,-2.75) rectangle (1,-1);
    \capfig{0.5}{3.5}{0.3}[thick]
     \node [opacity=1] at (0.5,-1.75) {$\#$};
     \draw[very thick] (0.25,-1.25) to [in=-90, out=-90] (.75,-1.25);
\end{tikzpicture}
: \ \ \ \ \
\begin{tikzpicture}[fill opacity=.3, scale= 1.5, anchorbase]
 \filldraw [very thick,fill=red] (-1,0) rectangle (1,1);
\filldraw [very thick,fill=blue] (-1,-1) rectangle (1,0);
    \filldraw [very thick,fill=blue] (-1,-3) rectangle (0,-1);
     \filldraw [very thick,fill=blue] (0,-3) rectangle (1,-1);
     \node [opacity=1] at (0,-0.5) {$\text{Id}_{K(r)}$};
    \node [opacity=1] at (-0.5,-2) {$\text{Id}_{K(r)}$};
    \capfig{0.5}{4}{0.3}[thick]
\end{tikzpicture} 
\stackrel{isotopy}{\xrightarrow{\hspace{2cm}}}
\begin{tikzpicture}[fill opacity=.3, scale= 1.5, anchorbase]
    \filldraw [very thick,fill=red] (-1,0.5) rectangle (1,1);
    \filldraw [very thick,fill=blue] (-1,0) rectangle (1,0.5);
	\filldraw [very thick,fill=blue] (-1,-1) rectangle (1,0);
    \filldraw [very thick,fill=blue] (-1,-3) rectangle (0,-1);
     \filldraw [very thick,fill=blue] (0,-3) rectangle (1,-1);
     \node [opacity=1] at (0,-0.5) {$\text{Id}_{K(r+1)}$};
    \node [opacity=1] at (-0.5,-2) {$\text{Id}_{K(r)}$};
    \capfig{0.5}{3.5}{0.3}[thick]
   
     \node [opacity=1] at (0.5,-1.75) {$\#$};
     \draw[very thick] (0.25,-1) to [in=-90, out=-90] (.75,-1);
     
\end{tikzpicture} 
\ \ 
\xrightarrow{\hspace{2cm}}
\ \ 
\begin{tikzpicture}[fill opacity=.3, scale= 1.5, anchorbase]
 \filldraw [very thick,fill=red] (-1,0.5) rectangle (1,1);
     \filldraw [very thick,fill=blue] (-1,0) rectangle (1,0.5);
	\filldraw [very thick,fill=blue] (-1,-1) rectangle (1,0);
     \node [opacity=1] at (0,-0.5) {$\text{Id}_{K(r+1)}$};
     
\end{tikzpicture} \\
\end{equation*}

\begin{equation*}
\begin{tikzpicture}[scale=2, anchorbase]
\draw (0,-2.75) rectangle (1,-1);
    \capfig{0.5}{3.5}{0.3}[thick]
\end{tikzpicture}
\ \ \ \  : \ \ 
\begin{tikzpicture}[fill opacity=.3, scale= 1.5, anchorbase]
    \filldraw [very thick,fill=red] (-1,0.5) rectangle (1,1);
    \filldraw [very thick,fill=blue] (-1,0) rectangle (1,0.5);
	
    \filldraw [very thick,fill=blue] (-1,-3) rectangle (0,0);
     \filldraw [very thick,fill=blue] (0,-3) rectangle (1,0);
     
    \node [opacity=1] at (-0.5,-2) {$\text{Id}_{K(r)}$};
    \capfig{0.5}{3.5}{0.3}[thick]
     \

\end{tikzpicture} 
\ \ 
\xrightarrow{\hspace{7cm}}
\ \ 
\begin{tikzpicture}[fill opacity=.3, scale= 1.5, anchorbase]
 \filldraw [very thick,fill=red] (-1,0.5) rectangle (1,1);
     \filldraw [very thick,fill=blue] (-1,0) rectangle (1,0.5);
	\filldraw [very thick,fill=blue] (-1,-1) rectangle (1,0);
     \node [opacity=1] at (0,-0.5) {$\text{Id}_{K(r)}$};
       
\end{tikzpicture} \\
\end{equation*}
We will achieve the stated goal by considering the maps of the following two forms that exist in $2HC^{conv}$:
\begin{equation*}
\begin{tikzpicture}[scale=2, anchorbase]
\draw (0,-2.75) rectangle (2,-1);
 \braid{0.25}{3.5}{0.5}[thick]
 \lines{1.2}{3.5}{0.5}[thick]
 \node [opacity=1] at (0.2,-2) {$\frac{1}{2}$};
 \node [opacity=1] at (0.9,-2) {$+  \ \  \frac{1}{2}$};
 
\end{tikzpicture}
\ \ \ \ \   \ \ \ \ \ 
 \begin{tikzpicture}[scale=2, anchorbase]
\draw (0,-2.75) rectangle (1,-1);

     \node [opacity=1] at (0.5,-1.4) {$\bullet$};
     \draw[very thick] (0.25,-1.25) to [in=-90, out=-90] (.75,-1.25);
\end{tikzpicture}
\end{equation*}

Finally, we will show that the following two maps induce the same map on a spectral sequence $\{E^n, d_n\}$ coming from a KFT, hence proving pseudo-finality in an appropriate quotient. 

\begin{equation*}
\begin{tikzpicture}[scale=2, anchorbase]
\draw (0,-2.75) rectangle (2,-1);
 \braid{0.25}{3.5}{0.5}[thick]
 \lines{1.2}{3.5}{0.5}[thick]
 \node [opacity=1] at (0.2,-2) {$\frac{1}{2}$};
 \node [opacity=1] at (0.9,-2) {$+  \ \  \frac{1}{2}$};
 
\end{tikzpicture}
\ \
\circ
\ \
\begin{tikzpicture}[scale=2, anchorbase]
\draw (0,-2.75) rectangle (1,-1);
    \capfig{0.5}{3.5}{0.3}[thick]
     \node [opacity=1] at (0.5,-1.75) {$\#$};
     \draw[very thick] (0.25,-1.25) to [in=-90, out=-90] (.75,-1.25);
\end{tikzpicture}
: \ \ 
\end{equation*}
\par
\begin{equation*}
\begin{tikzpicture}[fill opacity=.3, scale= 1.5, anchorbase]
 \filldraw [very thick,fill=red] (-1,0) rectangle (1,1);
 \filldraw [very thick,fill=blue] (-1,-1) rectangle (1,0);
    \filldraw [very thick,fill=blue] (-1,-3) rectangle (0,-1);
     \filldraw [very thick,fill=blue] (0,-3) rectangle (1,-1);
     \node [opacity=1] at (0,-0.5) {${\text{Id}}_{K(r)}$};
    \node [opacity=1] at (-0.5,-2) {${\text{Id}}_{K(r)}$};
    \capfig{0.5}{4}{0.3}[thick];
\end{tikzpicture} 
\stackrel{isotopy}{\xrightarrow{\hspace{2cm}}}
\begin{tikzpicture}[fill opacity=.3, scale= 1.5, anchorbase]
    \filldraw [very thick,fill=red] (-1,0.5) rectangle (1,1);
    \filldraw [very thick,fill=blue] (-1,0) rectangle (1,0.5);
	\filldraw [very thick,fill=blue] (-1,-1) rectangle (1,0);
    \filldraw [very thick,fill=blue] (-1,-3) rectangle (0,-1);
     \filldraw [very thick,fill=blue] (0,-3) rectangle (1,-1);
     \node [opacity=1] at (0,-0.5) {$\text{Id}_{K(r+1)}$};
    \node [opacity=1] at (-0.5,-2) {$\text{Id}_{K(r)}$};
    \capfig{0.5}{3.5}{0.3}[thick];
   
     \node [opacity=1] at (0.5,-1.75) {$\#$};
     \draw[very thick] (0.25,-1) to [in=-90, out=-90] (.75,-1);
\end{tikzpicture} 
\ \ 
\xrightarrow{\hspace{2cm}}
\ \ 
\begin{tikzpicture}[fill opacity=.3, scale= 1.5, anchorbase]
 \filldraw [very thick,fill=red] (-1,0.5) rectangle (1,1);
     \filldraw [very thick,fill=blue] (-1,0) rectangle (1,0.5);
	\filldraw [very thick,fill=blue] (-1,-1) rectangle (1,0);
     \node [opacity=1] at (0,-0.5) {$\text{Id}_{K(r+1)}$};
     
\end{tikzpicture} 
\end{equation*}

\begin{equation*}
    \begin{tikzpicture}[scale=2, anchorbase]
\draw (0,-2.75) rectangle (1,-1);
    \capfig{0.5}{3.5}{0.3}[thick]
    
     \node [opacity=1] at (0.5,-1.4) {$\bullet$};
     \draw[very thick] (0.25,-1.25) to [in=-90, out=-90] (.75,-1.25);
\end{tikzpicture}
\ \ \ \ : \ \ \ \ \ \ 
\begin{tikzpicture}[fill opacity=.3, scale= 1.5, anchorbase]
 \filldraw [very thick,fill=red] (-1,0) rectangle (1,1);

	\filldraw [very thick,fill=blue] (-1,-1) rectangle (1,0);
    \filldraw [very thick,fill=blue] (-1,-3) rectangle (0,-1);
     \filldraw [very thick,fill=blue] (0,-3) rectangle (1,-1);
     \node [opacity=1] at (0,-0.5) {$\text{Id}_{K(r)}$};
    \node [opacity=1] at (-0.5,-2) {$\text{Id}_{K(r)}$};
    \capfig{0.5}{3.5}{0.3}[thick]
    \node [opacity=1] at (0.5,-1.35) {$\bullet$};
     \nsphere{0.5}{4.5}{0.3}[thick];
  
\end{tikzpicture} 
\stackrel{isotopy}{\xrightarrow{\hspace{2cm}}}
\begin{tikzpicture}[fill opacity=.3, scale= 1.5, anchorbase]
    \filldraw [very thick,fill=red] (-1,0.5) rectangle (1,1);
    \filldraw [very thick,fill=blue] (-1,0) rectangle (1,0.5);
	\filldraw [very thick,fill=blue] (-1,-1) rectangle (1,0);
    \filldraw [very thick,fill=blue] (-1,-3) rectangle (0,-1);
     \filldraw [very thick,fill=blue] (0,-3) rectangle (1,-1);
     \node [opacity=1] at (0,-0.5) {$\text{Id}_{K(r+1)}$};
    \node [opacity=1] at (-0.5,-2) {$\text{Id}_{K(r)}$};
    \capfig{0.5}{3.5}{0.3}[thick]
     \node [opacity=1] at (0.5,-1.15) {$\bullet$};
     \draw[very thick] (0.25,-1) to [in=-90, out=-90] (.75,-1);
      
\end{tikzpicture} 
\ \ 
\xrightarrow{\hspace{2cm}}
\ \ 
\begin{tikzpicture}[fill opacity=.3, scale= 1.5, anchorbase]
 \filldraw [very thick,fill=red] (-1,0.5) rectangle (1,1);
     \filldraw [very thick,fill=blue] (-1,0) rectangle (1,0.5);
	\filldraw [very thick,fill=blue] (-1,-1) rectangle (1,0);
     \node [opacity=1] at (0,-0.5) {$\text{Id}_{K(r+1)}$};
       
\end{tikzpicture} \\
\end{equation*}

Finally, to be able to use Theorem \ref{thm:quotientcategory}, we must show that on every $E^n$, these two maps represent the same maps. 

Let us first recall some properties that the cobordisms satisfy under each page $E^n$:

\begin{equation}
	\label{eqn:TLbraiding}
\begin{tikzpicture}[anchorbase]
\draw[very thick] (0,0) to (1,2);
\draw[very thick] (0,2) to (.4,1.2);
\draw[very thick] (.6,0.8) to (1,0);
\end{tikzpicture}
:= \ \ 
\begin{tikzpicture}[anchorbase]
\draw[very thick] (0,0) to (0,2);
\draw[very thick] (1,0) to (1,2);
\end{tikzpicture}
\ \ - \ \ 
\begin{tikzpicture}[anchorbase]
\draw[very thick] (0,0) to [in=100,out=80](1,0);
\draw[very thick] (0,2) to [in=-100, out=-80](1,2);
\end{tikzpicture}
:=
\begin{tikzpicture}[anchorbase]
\draw[very thick] (0,0) to (0.4,.8);
\draw[very thick] (.6,1.2) to (1,2);
\draw[very thick] (0,2) to (1,0);
\end{tikzpicture}
\end{equation}

\begin{equation}
	\label{eq:BNrels}
\begin{tikzpicture} [fill opacity=0.2,anchorbase, scale=.375]
	\path[fill=blue, opacity=.2] (1,0) arc[start angle=0, end angle=180,x radius=1,y radius=.5] 
		to (-1,4) arc[start angle=180, end angle=0,x radius=1,y radius=.5] to (1,0);
	\path[fill=blue, opacity=.2] (1,0) arc[start angle=360, end angle=180,x radius=1,y radius=.5] 
		to (-1,4) arc[start angle=180, end angle=360,x radius=1,y radius=.5] to (1,0);
	\draw [very thick] (0,4) ellipse (1 and 0.5);
	\draw [very thick] (0,0) ellipse (1 and 0.5);
	\draw[very thick] (1,4) -- (1,0);
	\draw[very thick] (-1,4) -- (-1,0);
\end{tikzpicture}
\, = \,
\begin{tikzpicture} [fill opacity=0.2,anchorbase, scale=.375,rotate=180]
	\path[fill=blue,opacity=.2] (1,4) arc[start angle=0, end angle=180,x radius=1,y radius=.5] 
		to [out=270,in=180] (0,2.5) to [out=0,in=270] (1,4);
	\path[fill=blue,opacity=.2] (1,4) arc[start angle=360, end angle=180,x radius=1,y radius=.5] 
		to [out=270,in=180] (0,2.5) to [out=0,in=270] (1,4);
	\draw[very thick] (0,4) ellipse (1 and 0.5);
	\draw[very thick] (-1,4) to [out=270,in=180] (0,2.5) to [out=0,in=270] (1,4);
	\path[fill=blue,opacity=.2] (1,0) arc[start angle=0, end angle=180,x radius=1,y radius=.5] 
		to [out=90,in=180] (0,1.5) to [out=0,in=90] (1,0);
	\path[fill=blue,opacity=.2] (1,0) arc[start angle=360, end angle=180,x radius=1,y radius=.5] 
		to [out=90,in=180] (0,1.5) to [out=0,in=90] (1,0);
	\draw[very thick] (0,0) ellipse (1 and 0.5);
	\draw[very thick] (-1,0) to [out=90,in=180] (0,1.5) to [out=0,in=90] (1,0);
	\node[opacity=1] at (0,1) {\footnotesize$\bullet$};
\end{tikzpicture}
+
\begin{tikzpicture} [fill opacity=0.2,anchorbase, scale=.375]
	\path[fill=blue,opacity=.2] (1,4) arc[start angle=0, end angle=180,x radius=1,y radius=.5] 
		to [out=270,in=180] (0,2.5) to [out=0,in=270] (1,4);
	\path[fill=blue,opacity=.2] (1,4) arc[start angle=360, end angle=180,x radius=1,y radius=.5] 
		to [out=270,in=180] (0,2.5) to [out=0,in=270] (1,4);
	\draw[very thick] (0,4) ellipse (1 and 0.5);
	\draw[very thick] (-1,4) to [out=270,in=180] (0,2.5) to [out=0,in=270] (1,4);
	\path[fill=blue,opacity=.2] (1,0) arc[start angle=0, end angle=180,x radius=1,y radius=.5] 
		to [out=90,in=180] (0,1.5) to [out=0,in=90] (1,0);
	\path[fill=blue,opacity=.2] (1,0) arc[start angle=360, end angle=180,x radius=1,y radius=.5] 
		to [out=90,in=180] (0,1.5) to [out=0,in=90] (1,0);
	\draw[very thick] (0,0) ellipse (1 and 0.5);
	\draw[very thick] (-1,0) to [out=90,in=180] (0,1.5) to [out=0,in=90] (1,0);
	\node[opacity=1] at (0,1) {\footnotesize$\bullet$};
\end{tikzpicture}
\, , \quad
\begin{tikzpicture}[anchorbase, scale=.375]
	\path [fill=blue,opacity=0.3] (0,0) circle (1);
	\draw (-1,0) .. controls (-1,-.4) and (1,-.4) .. (1,0);
	\draw[dashed] (-1,0) .. controls (-1,.4) and (1,.4) .. (1,0);
	\draw[very thick] (0,0) circle (1);
\end{tikzpicture}
= 0
\, , \quad
\begin{tikzpicture}[anchorbase, scale=.375]
	\path [fill=blue,opacity=0.3] (0,0) circle (1);
	\draw (-1,0) .. controls (-1,-.4) and (1,-.4) .. (1,0);
	\draw[dashed] (-1,0) .. controls (-1,.4) and (1,.4) .. (1,0);
	\draw[very thick] (0,0) circle (1);
	\node at (0,0.6) {\footnotesize$\bullet$};
\end{tikzpicture}
= 1
\, , \quad
\begin{tikzpicture}[fill opacity=.3, scale=.5, anchorbase]
	\filldraw [very thick,fill=blue] (-1,-1) rectangle (1,1);
	\node [opacity=1] at (0,-.25) {$\bullet$};
	\node [opacity=1] at (0,.25) {$\bullet$};
	\end{tikzpicture}
= 0 \, .
\end{equation}
Now we do the calculations after applying $\{E^n,d_n\}$ in green color to avoid confusion with calculations in $\overline{C(X,L)}$.

\begin{equation*}
 \ \  \ \ \ \ \ \ \ \ \ \ \ \ \ \ \ 
 \begin{tikzpicture}[scale=2, anchorbase]
 \filldraw[opacity=0.3, green] (0,-2.75) rectangle (2,-1);
 \braid{0.25}{3.5}{0.5}[thick]
 \lines{1.2}{3.5}{0.5}[thick]
 \node [opacity=1] at (0.2,-2) {$\frac{1}{2}$};
 \node [opacity=1] at (0.9,-2) {$+  \ \  \frac{1}{2}$};
 \end{tikzpicture}
\circ \ \ 
\begin{tikzpicture}[scale=2, anchorbase]
\filldraw[opacity=0.3, green](0,-2.75) rectangle (1,-1);
    \capfig{0.5}{3.5}{0.3}[thick]
     \node [opacity=1] at (0.5,-1.75) {$\#$};
     \draw[very thick] (0.25,-1.25) to [in=-90, out=-90] (.75,-1.25);
\end{tikzpicture}
  \ \ \ \ \ \ \ \ \ \ \ \ \  =
\begin{tikzpicture}[scale=2, anchorbase]
 \filldraw[opacity=0.3, green] (0,-2.75) rectangle (2,-1);
 \braid{0.25}{3.5}{0.5}[thick]
 \lines{1.2}{3.5}{0.5}[thick]
 \node [opacity=1] at (0.2,-2) {$\frac{1}{2}$};
 \node [opacity=1] at (0.9,-2) {$+  \ \  \frac{1}{2}$};
\end{tikzpicture}
\circ  \ \ 
\begin{tikzpicture}[scale=2, anchorbase]
\filldraw[opacity=0.3, green] (0,-2.75) rectangle (2,-1);
    \capfig{0.5}{3.5}{0.3}[thick];
    
     \draw[very thick] (0.25,-1.25) to [in=-90, out=-90] (.75,-1.25);
     \capfig{1.5}{3.5}{0.3}[thick];
    \node [opacity=1] at (0.5,-2) {$\bullet$};
     \node [opacity=1] at (1.5,-1.4) {$\bullet$};
 \node [opacity=1] at (1,-1.75) {$+ $};
 
     \draw[very thick] (1.25,-1.25) to [in=-90, out=-90] (1.75,-1.25);
\end{tikzpicture}
\ \ \ \ \ \ \ \ \ \ \ \ \ = 
\begin{tikzpicture}[scale=2, anchorbase]
 \filldraw[opacity=0.3, green] (0,-2.75) rectangle (2,-1);
 \cupcap{0.25}{3.5}{0.5}[thick]
 \lines{1.2}{3.5}{0.5}[thick]
 \node [opacity=1] at (0.2,-2) {$-\frac{1}{2}$};
 \node [opacity=1] at (0.9,-2) {$+$};
\end{tikzpicture}
\circ  \ \ 
\begin{tikzpicture}[scale=2, anchorbase]
\filldraw[opacity=0.3, green](0,-2.75) rectangle (2,-1);
    \capfig{0.5}{3.5}{0.3}[thick];
    
     \draw[very thick] (0.25,-1.25) to [in=-90, out=-90] (.75,-1.25);
     \capfig{1.5}{3.5}{0.3}[thick];
    \node [opacity=1] at (0.5,-2) {$\bullet$};
     \node [opacity=1] at (1.5,-1.4) {$\bullet$};
 \node [opacity=1] at (1,-1.75) {$+ $};
 
     \draw[very thick] (1.25,-1.25) to [in=-90, out=-90] (1.75,-1.25);
\end{tikzpicture}
  \ \ \ \ \ \  \ \ \ \ \ \ = 
\begin{tikzpicture}[scale=2, anchorbase]
\filldraw[opacity=0.3, green] (0,-2.75) rectangle (4,-1);
    \capfig{0.5}{3.5}{0.3}[thick];
    
     \draw[very thick] (0.25,-1.25) to [in=-90, out=-90] (.75,-1.25);
     \capfig{1.5}{3.5}{0.3}[thick];
    \node [opacity=1] at (0.5,-2) {$\bullet$};
     \node [opacity=1] at (1.5,-1.4) {$\bullet$};
 \node [opacity=1] at (1,-1.75) {$+ $};
  \node [opacity=1] at (2,-1.75) {$- $};
  \node [opacity=1] at (2.4,-1.75) {$\frac{1}{2} \times 2$};
 
     \draw[very thick] (1.25,-1.25) to [in=-90, out=-90] (1.75,-1.25);
      \capfig{3}{3.5}{0.3}[thick];
      \draw[very thick] (2.75,-1.25) to [in=-90, out=-90] (3.25,-1.25);
       \node [opacity=1] at (3,-2) {$\bullet$};
        \node [opacity=1] at (3.4,-1.75) {$+ \ \ 0$};
       
\end{tikzpicture}
 = 
\begin{tikzpicture}[scale=2, anchorbase]
\filldraw[opacity=0.3, green] (0,-2.75) rectangle (1,-1);
    \capfig{0.5}{3.5}{0.3}[thick]
    
     \node [opacity=1] at (0.5,-1.4) {$\bullet$};
     \draw[very thick] (0.25,-1.25) to [in=-90, out=-90] (.75,-1.25);
\end{tikzpicture}
\end{equation*}
This proves the finality of $2HC^{conv}$ by the following quotient (i.e., pseudo-finality):

\begin{equation*}
 \begin{tikzpicture}[scale=2, anchorbase]
 \filldraw[opacity=0.3] (0,-2.75) rectangle (2,-1);
 \braid{0.25}{3.5}{0.5}[thick]
 \lines{1.2}{3.5}{0.5}[thick]
 \node [opacity=1] at (0.2,-2) {$\frac{1}{2}$};
 \node [opacity=1] at (0.9,-2) {$+  \ \  \frac{1}{2}$};
 \end{tikzpicture}
\circ \ \ 
\begin{tikzpicture}[scale=2, anchorbase]
\filldraw[opacity=0.3](0,-2.75) rectangle (1,-1);
    \capfig{0.5}{3.5}{0.3}[thick]
     \node [opacity=1] at (0.5,-1.75) {$\#$};
     \draw[very thick] (0.25,-1.25) to [in=-90, out=-90] (.75,-1.25);
\end{tikzpicture}
\ \ \ \ \ \ 
\sim
\ \ \ \ \ 
\begin{tikzpicture}[scale=2, anchorbase]
\filldraw[opacity=0.3] (0,-2.75) rectangle (1,-1);
    \capfig{0.5}{3.5}{0.3}[thick]
    
     \node [opacity=1] at (0.5,-1.4) {$\bullet$};
     \draw[very thick] (0.25,-1.25) to [in=-90, out=-90] (.75,-1.25);
\end{tikzpicture} 
\end{equation*}

\end{proof}

Note that \ref{thm:quotientcategory} implies that not only are these colimits isomorphic, but they also induce the same universal arrows. Coincidentally, the differentials we are about to define are universal arrows, coming from the universal property of the colimits. Hence, the differentials on the original colimit and the colimit given by $2HC^{conv}$ agree.

Let us provide a potential differential to get a spectral sequence of the lasagna modules. For a fixed link, a KFT $\mathcal{A}$ gives a spectral sequence $\{E^{n},d_n\}$. 
Hence, for an object $[\Sigma]\in \overline{C(X,L)}$, we have 
\begin{equation*}
    d_n([\Sigma]): E^{n}([\Sigma])\rightarrow E^{n}([\Sigma])
\end{equation*}
Note that by definition, $E^{n}([\Sigma])$ is  $E^{n}(\sqcup_{i}\partial \Sigma |_{B_{i}})$.\footnote{As we have already discussed, there is no well-defined notion of disjoint union for link diagrams for links in $S^3$. We simply want to point out that $E^{n}([\Sigma])$ comes from an appropriate link.}
Hence, for every object $[\Sigma]$ in $\overline{C(X,L)}$, there is a differential on $E^n ([\Sigma])$. Using this differential along with the universal property of a colimit, we get the following commutative diagram. 
\[\begin{tikzcd}
    E^{n}([\Sigma]) \arrow{r}{E^{n}(S)} \arrow{rrd}{\pi_{\Sigma}} \arrow{dd}{d_{n}([\Sigma])} & E^{n}([\Sigma^{'}]) \arrow{rd} {\pi_{\Sigma^{'}}}  \arrow{dd}{d_n([\Sigma]^{'}}\\ & & S^{E^{n}}(X,L) \arrow[dotted]{dd}{\overline{d_n}} \\
    E^{n}([\Sigma]) \arrow{rrd}{\pi_{\Sigma}}  \arrow{r}{E^{n}(S)} & E^{n}([\Sigma^{'}]) \arrow{rd} \\ & & S^{E^{n}}(X,L)
\end{tikzcd}\]

By the universal property, we have $$\overline{d_n}\circ \pi_{\Sigma}= \pi_{\Sigma}\circ d_{n}$$
At the same time $\overline{d_n} \circ \overline{d_n}$ is a unique arrow that satisfies: $$\overline{d_n} \circ \overline{d_n} \circ \pi_{\Sigma} = \pi_{\Sigma}\circ d_n \circ d_n = 0$$ $$\text{So we have : } \overline{d_n} \circ \overline{d_n} =0 $$ Hence, $\overline{d_n}$ is indeed a differential. 
Such a differential exists for $\overline{C(X,L)}^{conv}$ as well. We ignore the notational distinction between these two at the moment.
Now let us analyze the homology $H_{*}(S^{E^{n}}(X,L))$ with respect to this differential.
Note that this differential being a universal arrow agrees with the differential on $2HC^{conv}\subset \overline{C(X,L)}^{conv}$.
Now to analyze $H_{*}(S^{E^{n}}(X,L))$, consider the following commutative diagram:
\[\begin{tikzcd}
    E^{n+1}([\Sigma]) \arrow{rrd} \arrow{r}{E^{n+1}(S)} \arrow{dd}{:=} & E^{n+1}([\Sigma^{'}]) \arrow{rd} \arrow{dd}{:=}   \\ & & S^{E^{n+1}}(X,L) \arrow[dotted]{dd}{\Phi_{n}} \\
     H_{*}(E^{n}([\Sigma])) \arrow{rrd} \arrow{r}{E^{n}(S)_{*}} & H_{*}(E^{n}([\Sigma^{'}])) \arrow{rd} \\&  & H_{*}(S^{E^{n}}(X,L))
\end{tikzcd}\]

Before proving the main theorem, we will prove a lemma regarding the braid group action.
Let $\mathfrak{S}_i$ be the symmetric group on $i$ elements. As an application of the equation \ref{eq:braid}, one can see that the braid group action factors through the symmetric group $\mathfrak{S}_{\|\alpha \|+2r}$. First, we consider the subsystem\footnote{We prefer the term subsystem to highlight that it is used as an indexing to consider a colimit. In reality, it is simply a subcategory of the indexing category} of $2HC$ corresponding to $ \mathfrak{S}_{\|\alpha \|+2r}$. This subsystem consists of skein corresponding to cables, mentioned in figure\ref{fig:generator}, along with morphisms on them given by the elements of the symmetric group. So the indexing category has a single object corresponding to $[\Sigma_{K(r)}]$ and the arrows correspond to elements of the symmetric group. With some abuse of notation, we will denote the colimit over this indexing category as ${\text{colim}}_{\mathfrak{S}_{\|\alpha \|+2r}}$.

We prove the following lemma:
\begin{Lem}
    \begin{equation*}
        H_{*}({\text{colim}}_{\mathfrak{S}_{\|\alpha \|+2r}}(E^{n}(K(r))))\cong {\text{colim}}_{\mathfrak{S}_{\|\alpha \|+2r}}(E^{n+1}(K(r)))
    \end{equation*}
    where the homology is taken using the differential induced by $d_n$ on the colimit.
\end{Lem}

\begin{proof}
    First, let us unravel the colimit:
    \begin{equation}
      {\text{colim}}_{\mathfrak{S}_{\|\alpha \|+2r}}(E^{n}(K(r))) \cong E^{n}((K(r))/\{v-\sigma(v)\| v\in E^{n}(K(r)), \ \sigma \in   \mathfrak{S}_{\|\alpha \|+2r}\}
    \end{equation}
However, in a field $\mathbb{k}$ of characteristic zero, 
\begin{equation*}
    E^{n}(K(r))^{\mathfrak{S}_{\|\alpha \|+2r}}\oplus \mathbb{k}\{v-\sigma(v)\}\cong  E^{n}(K(r))
\end{equation*}

where the first vector space in the summation is the subspace of vectors fixed under the symmetric group action. This can be seen as follows:
\begin{equation}
    v=\frac{\sum_{\sigma \in  \mathfrak{S}_{\|\alpha \|+2r}} \sigma(v)}{\| \mathfrak{S}_{\|\alpha \|+2r}\|}+ v-\frac{\sum_{\sigma \in  \mathfrak{S}_{\|\alpha \|+2r}} \sigma(v)}{\| \mathfrak{S}_{\|\alpha \|+2r}\|} = \frac{\sum_{\sigma \in  \mathfrak{S}_{\|\alpha \|+2r}} \sigma(v)}{\| \mathfrak{S}_{\|\alpha \|+2r}\|}+ \frac{\sum_{\sigma \in  \mathfrak{S}_{\|\alpha \|+2r}} v-\sigma(v)}{\| \mathfrak{S}_{\|\alpha \|+2r}\|}
\end{equation} 
and note that:
\begin{equation*}
    \sigma(\frac{\sum_{\sigma \in  \mathfrak{S}_{\|\alpha \|+2r}} \sigma(v)}{\| \mathfrak{S}_{\|\alpha \|+2r}\|})=\frac{\sum_{\sigma \in  \mathfrak{S}_{\|\alpha \|+2r}} \sigma(v)}{\| \mathfrak{S}_{\|\alpha \|+2r}\|} \ \ \ \ \forall \ \sigma \in \mathfrak{S}_{\|\alpha \|+2r}
\end{equation*}

Hence we get 
\begin{equation}
    E^{n}(K(r))^{\mathfrak{S}_{\|\alpha \|+2r}} \cong  E^{n}((K(r))/\mathbb{k}\{v-\sigma(v)\} \cong {\text{colim}}_{\mathfrak{S}_{\|\alpha \|+2r}}(E^{n}(K(r)))
\end{equation}
It is easy to see that the first isomorphism in the above equation commutes with the differentials of interest. Hence, it is enough to prove, 
\begin{equation}
    H_{*}( E^{n}(K(r))^{\mathfrak{S}_{\|\alpha \|+2r}})\cong  E^{n+1}(K(r))^{\mathfrak{S}_{\|\alpha \|+2r}}
\end{equation}
To see this, consider a cycle $v \in E^{n}(K(r))$ that is fixed by the symmetric group action. This gives us an element of $[v]\in E^{n+1}(K(r))$. Note that the symmetric group action fixes $[v]$ due to functoriality. To prove the opposite direction, consider $ [v] \in E^{n+1}(K(r))$ fixed by the symmetric group action:
\begin{equation*}
    \sigma_{n+1}([v])=[v]\ \  \forall \ \sigma  \ \in \ \mathfrak{S}_{\|\alpha \|+2r}
\end{equation*}

For any lift $v\in E^{n}(K(r))$, consider the element $\frac{\sum_{\sigma\in \mathfrak{S}_{\|\alpha \|+2r} }\sigma_{n}(v)}{\| \mathfrak{S}_{\|\alpha \|+2r}\|}$, observe that on the $(n+1)$-th page:
\begin{equation*}
    [\frac{\sum_{\sigma\in \mathfrak{S}_{\|\alpha \|+2r} }\sigma_{n}(v)}{\| \mathfrak{S}_{\|\alpha \|+2r}\|}]= \frac{\sum_{\sigma\in \mathfrak{S}_{\|\alpha \|+2r} }[\sigma_{n}(v)]}{\| \mathfrak{S}_{\|\alpha \|+2r}\|}=\frac{\sum_{\sigma\in \mathfrak{S}_{\|\alpha \|+2r} }\sigma_{n+1}([v])}{\| \mathfrak{S}_{\|\alpha \|+2r}\|}=[v]
\end{equation*}
Hence $\frac{\sum_{\sigma\in \mathfrak{S}_{\|\alpha \|+2r} }\sigma_{n}(v)}{\| \mathfrak{S}_{\|\alpha \|+2r}\|}$ is a lift of $[v]$ fixed by the braid group action on the n-th page, hence completing the proof.
\end{proof}
We now prove the main theorem of this paper. 
\begin{proof}{(of the \textbf{Theorem} \ref{thm: main})}

On the n-th page, we have 
\begin{equation}
    S_{\alpha \in H_{2}(X,L) } ^{E^{n}}(X,L)\cong {\text{Fcolim}}_{r\in \mathbb{N}^{k}} {\text{colim}}_{\sigma \in \mathfrak{S}_{\|\alpha \|+2r}} E^{n}(K(r))
\end{equation}
where the first colimit is filtered, taken over maps between cables of $K$ (dotted annulus to be precise).
Now, with respect to the differential induced on the colimit, after using the previous lemma, we get:
\begin{align*} 
    H_{*}(S_{\alpha \in H_{2}(X,L) } ^{E^{n}}(X,L))\cong  H_{*}( {\text{Fcolim}}_{r\in \mathbb{N}^{k}} {\text{colim}}_{\sigma \in \mathfrak{S}_{\|\alpha \|+2r}} E^{n}(K(r)))  \\  \cong  F{\text{colim}}_{r\in \mathbb{N}^{k}} H_{*}({\text{colim}}_{\sigma \in \mathfrak{S}_{\|\alpha \|+2r}} E^{n}(K(r)))\\ \cong  {\text{Fcolim}}_{r\in \mathbb{N}^{k}} {\text{colim}}_{\sigma \in \mathfrak{S}_{\|\alpha \|+2r}} H_{*}(E^{n}(K(r)))\\ \cong   {\text{Fcolim}}_{r\in \mathbb{N}^{k}} {\text{colim}}_{\sigma \in \mathfrak{S}_{\|\alpha \|+2r}} E^{n+1}(K(r)) \\ \cong  S_{\alpha \in H_{2}(X,L) } ^{E^{n+1}}(X,L)
\end{align*} \label{eq:differentdiff}
\begin{Rmk}
The fact that the differentials in the first line of the above equation \ref{eq:differentdiff} agree follows from the setting of the theorem \ref{thm:2hdlbdy} and the naturality part of the theorem \ref{thm:quotientcategory}. 
This can be seen as follows. Consider the pseudo-final inclusion functor $$i: \ 2HC \rightarrow \overline{C(X,L)}^{conv}$$ Let $d_{2HC}\ , \ d_{\overline{C(X,L)}^{conv}} $ be the differentials coming from both the systems using the procedure described above. Now let $[S]:[\Sigma]\rightarrow[\Sigma^{'}]$ be a morphism that is in $2HC$. Let $i_{*}$ be the isomorphism induced by $i$ on $E^n$. Now, consider the following diagram.
\end{Rmk}

\[\begin{tikzcd}
&& E^n(i([\Sigma]))\arrow{rr}{i([S])} \arrow{rdddd} & & E^n(i([\Sigma^{'}])) \arrow{ldddd} \\ \\
E^n([\Sigma])\arrow{rr}{[S]} \arrow{rruu}{i} \arrow{rdddd} & & E^n([\Sigma^{'}]) \arrow{rruu}{i}  \arrow{ldddd} && \\
\\
& & & \text{colim}_{\overline{C(X,L)}^{conv}} \arrow[dotted, color=purple]{dd}{d_{\overline{C(X,L)}^{conv}}}\\
\\
& \text{colim}_{2HC} \arrow[dotted,color=purple]{dd}{d_{2HC}} \arrow[ color=purple]{rruu}{i_{*}} & & \text{colim}_{\overline{C(X,L)}^{conv}} \\ \\
& \text{colim}_{2HC} \arrow[ color=purple]{rruu}{i_{*}} & & \\
\end{tikzcd}\]
Note that the purple square in the diagram above commutes. One can see this by considering the morphism $i_{*}^{-1}\circ d_{\overline{C(X,L)}^{conv}} \circ i_{*} $. One can check that this map satisfies the universal property of colimits over $2HC$, hence by uniqueness it must be equal to $d_{2HC}$.
Hence, the differentials on both colimits must match and give the same homology on each page. We can see that the advantage of this construction is that, even though the proof of the existence of the spectral sequence depends heavily on the choice of the 2-handlebody structure, the spectral sequence itself is independent (up to isomorphism of spectral sequences) of the choice of 2-handlebody structure.

Before we discuss some applications, we define the $E^\infty$ TQFT corresponding to a KFT. 
\begin{Def}\label{def:Einfty}
    For a given KFT $\mathcal{A}$, and some diagram $D$ of a fixed link $L$, if the spectral sequence corresponding to $\mathcal{A}$ converges on the $n$-th page, then we define $E^\infty (L)= E^{n} (L)$ . For a cobordism $\Sigma : L \rightarrow L^{'}$, we define $E^{\infty}([\Sigma])= E^{max(n, n^{'})}([\Sigma])$. Where the spectral sequence corresponding to $L^{'}$ converges on the $n^{'}$-th page.
\end{Def}

\end{proof}
\begin{Cor}
    \begin{equation*}
        rank ( S_{\alpha } ^{E^{n}}(X,L)) \geq rank(S_{\alpha  } ^{E^{n+1}}(X,L)) \geq rank(S_{\alpha  } ^{E^{\infty}}(X,L)) \ \ \forall n\geq2
    \end{equation*}
\end{Cor}
\begin{proof}
    The first inequality follows immediately from the theorem. For the second inequality, choose linearly independent elements 
\begin{equation}
    v_{1},\dots ,v_{k} \in S_{\alpha  } ^{E^{\infty}}(X,L) 
\end{equation}
Choose r large enough and corresponding linearly independent representatives 
\begin{equation}
    v_{1}^{'}, \dots, v_{k}^{'} \in E^{\infty}(K(r))
\end{equation}
Now for any non-zero tuple $\{c_{i}\}$, for any $n\geq 2$,
\begin{equation}
    \sum_{i} c_{i} v_{i}^n =0 \implies \sum_{i} c_{i} v_{i}^{'} =0 
\end{equation}
where $v_{i}^{n}$ is a lift on the nth page. On top of that this lift must survive all maps under the filtered colimit. Hence, 
\begin{equation}
     rank ( S_{\alpha } ^{E^{n}}(X,L)) \geq S_{\alpha  } ^{E^{\infty}}(X,L)
\end{equation}

\end{proof}
\begin{Rmk}
    Suppose a spectral sequence $(E^{n},d^{n})$ coming from a KFT converges to a TQFT $Z$. Then the $E^{\infty}$ TQFT need not be the same as $Z$. For example, in the case of the Lee spectral sequence, the sphere with 3-dots, considered as a cobordism between empty links, induces the $\text{Id}$ map. While following the definition \ref{def:Einfty}, one can see that it gives the zero map on $E^\infty$. Hence, to apply the results proved so far, one needs to investigate the relation between these two TQFTs. We provide one such example in the next section. 
\end{Rmk}
\section{Lee Spectral Sequence and Applications}
The primary goal of this section is to discuss applications of our spectral sequence construction. First, we prove that the rank inequality between the Khovanov and Lee lasagna modules proved in \cite{rw24} comes from a spectral sequence of lasagna modules. Then we show that our construction of the spectral sequence satisfies a connect-sum formula, similar to the one introduced in \cite{mn22}.
\subsection{Lasagna spectral sequence of Lee Homology }
In this subsection, we prove that the rank inequality initially proved in \cite{rw24}, can be seen as a corollary of the results of the previous section; hence showing that the inequality does indeed come from a spectral sequence. For the Lee spectral sequence, we will discuss the relationship between $E^\infty$ TQFT and the Lee TQFT. We will also set up a general framework for deriving similar results for other KFTs and possibly prove various non-vanishing results.\

First, we set up the required background. Let $KhR_2$ be the $gl_2$ Khovanov-Rozansky Homology defined in \cite{kr04}. $KhR_2$ and the standard Khovanov homology $Kh$ are related by the following equivalence:
\begin{equation}\label{eq:KhR_2}
KhR_2^{h,q}(L)\cong Kh^{h,-q-w(L)}(-L),
\end{equation}
The $KhR_2$ homology is defined over the Frobenius algebra $V=\mathbb{k}[X]/(X^2)$. The co-multiplication and co-unit maps are  given by $$\Delta1=X\otimes 1+1\otimes X,\ \Delta X=X\otimes X,\ \epsilon1=0,\ \epsilon X=1.$$ Here $X$ has quantum degree $2$.
The $KhR_{Lee}$ deformation is defined over the Frobenius algebra $V_{Lee}=\mathbb{k}[X]/(X^2-1)$ with co-multiplication and co-unit given by: $$\Delta1=X\otimes1+1\otimes X,\ \Delta X=X\otimes X+1\otimes1,\ \epsilon1=0,\ \epsilon X=1.$$
Now, note that the underlying chain complex of both $KhR_2$ and $KhR_{Lee}$ is the same while the differential $d^{Lee}$ relates to the differential $d^{KhR_2}$ as follows:
$$d^{KhR_2} : CKhR_2 (L)_{q,n}  \longrightarrow  CKhR_2 (L)_{q,n+1}$$
$$d^{Lee} : CKhR_2 (L)_{q,n}  \rightarrow  CKhR_2 (L)_{q,n+1} \bigoplus CKhR_2 (L)_{q-4,n+1} $$ with
\begin{equation} \label{eq:leedifferential}
d^{Lee}= d^{KhR_2} + \delta
\end{equation}

Here $\delta$ shifts the degree down by 4.

\subsection*{Spectral Sequence from Filtered Chain complex : }
For most of the KFTs, the $E^\infty$ page as a TQFT is quite different from the TQFT it converges to. But in the presence of a special filtration such as the one for Lee homology, there is a way to relate these two TQFTs. To understand this, we first discuss the basics of such spectral sequences. Readers familiar with constructions of spectral sequences may skip this section. 
Note that we will be dealing with a decreasing filtration to be able to deal with $KhR_2$. Similar constructions and results for increasing filtration can be found in standard texts like \cite{mac1998categories}.

Now, consider a (decreasing) filtered chain complex:
\[\begin{tikzcd}
    \dots \arrow{r} & F_{p+1} C_{\bullet} \arrow{r} & F_{p}C_{\bullet} \arrow{r} & F_{p-1}C_{\bullet} \arrow{r} & \dots\\
\end{tikzcd}\]
Such a filtered complex has a differential $d$ satisfying $d(F_p C_n)\subset F_p C_{n-1}$

For such a filtered complex, we define :
\begin{itemize}
    \item  $Z^{r}_{p,q}= \{ x \in F_p C_{p+q} | dx \in F_{p+r}C_{p+q-1}\}/ F_{p+1}C_{p+q}$
    \item $B^{r}_{p,q}= d(F_{p-r-1}C_{p+q-1})$
    \item $E^{r}_{p,q}= Z^{r}_{p,q}/B^{r}_{p,q}  $
    
\end{itemize}

\begin{Thm}
A filtered chain complex gives rise to a spectral sequence as follows:
\begin{itemize}
\item The differential $d$ on $C_{\bullet}$ induces a differential $d^r$ on   $Z^{r}_{p,q}$ and on $E^{r}_{p,q}$ $(d^{r}: E^{r}_{p,q} \rightarrow E^{r}_{p+r,q-r-1} )$
\item $\ker(d^r)= Z^{r+1}_{p,q}$
\item We get a spectral sequence by considering $H_{*}(E^{r}_{p,q}, d^{r})\cong E^{r+1}_{p,q}$
\end{itemize}
\end{Thm}

Now, coming back to equation \ref{eq:leedifferential}, noting that $d^{Lee}$ gives a filtration, we get a corresponding spectral sequence. Recall that the Khovanov differential preserves the grading, hence we get $$d^{KhR_2}(F_{p}C_{\bullet}) \not\subset  F_{p+1}C_{\bullet}$$ and $$\delta(F_{p}C_{\bullet})\subset F_{p-4}C_{\bullet} \subset F_{p-2}C_{\bullet}$$ Here $d^{KhR_2}$ preserves the quantum degree while $\delta$ reduces the quantum degree by 4.
$$ Z^{2}_{p,q}= \{ x \in F_p C_{p+q} | d^{Lee}x \in F_{p-2}C_{p+q-1}\}/ F_{p+1}C_{p+q} =  \{ x \in F_p C_{p+q} | d^{KhR_2}(x) =0 \}$$ Hence the $E^{2}$ page of our spectral sequence is $KhR_2$ while, $$ E^{\infty}_{p,\bullet}\cong G_{p}(KhR_{Lee})$$ 
Consider a cobordism $S: L_{1} \rightarrow L_2 $ representing a map between two skeins. Then the map $KhR_2 (S) $ is degree preserving (up to a few pre-adjusted degree shifts depending only on the genus of the skeins and the number of dots on the cobordism) \footnote{To be precise, to a skein $[\Sigma] $ with surface $\Sigma$ and input links $L_{i}$, on the n-th page we assign $\otimes_{i}E^{n}(L_{i}) $ while shifting the homological degree by $(-\chi(\Sigma)+2(\text{\# dots on }\Sigma))$.} Note that the required degree shift is already incorporated into the 2-handle-body formula \ref{prop:2_hdby}. Hence, on any page $E^n$, a map in the skein category induces a degree-preserving map.
Hence, by the definition of $E^{\infty}$ TQFT, the map $E^{\infty}(S)$ is the degree preserving part of $KhR_{Lee}(S)$. Hence, we get the following correspondence:

\begin{Thm}
    For the functorial spectral sequence of TQFTs from $KhR_2$ to $KhR_{Lee}$ we have:
    \begin{equation*}
        S_{\alpha, q}^{E^{\infty}}(X,L) \cong gr_{q} S_{\alpha  } ^{KhR_{Lee}}(X,L) \ \ \forall q \in \mathbb{Z}
    \end{equation*} 
    \begin{equation*}
        S_{\alpha  } ^{E^{\infty}}(X,L) \cong \bigoplus_{h,q\in \mathbb{Z}} gr_{q} S_{\alpha  } ^{KhR_{Lee}}(X,L)
    \end{equation*}
\end{Thm}
\begin{proof}
    We prove this by a careful examination of the Lee spectral sequence. After an appropriate grading shift as in the 2-handlebody formula for skein lasagna modules, we see that for a cobordism $\Sigma:L\rightarrow L^{'}$, the map induced on $E^{\infty}$ page of the Lee Spectral sequence corresponds to the filtration degree preserving map on $KhR_{Lee}$. 
    Now, note that the filtration function on $S^{Lee}$ is the minimum of filtration on $[\Sigma,\nu]$ over all the possible equivalence classes. 
    So if $q(v)\neq -\infty,\ v\in S^{Lee}(X,L) $, then $\exists$ a lift $v^{'}\in KhR_{Lee}(K(r_{0}))$ for some $r_{0}$, so that $q(v^{'})=q(v)$ and the degree of all equivalent elements in $r\geq r_{0}$ is also $q(v)$, since otherwise the filtration degree must go to $-\infty$. 
    We consider the obvious map :
    \begin{equation*}
        F: \bigoplus_ {h,q\in \mathbb{Z}} gr_{q} S_{\alpha  } ^{KhR_{Lee}}(X,L) \rightarrow S_{\alpha  } ^{E^{\infty}}(X,L)
    \end{equation*}

By the degree preservation argument above, the map is clearly surjective. Now, consider a set of elements $v_{1}^{'}\dots v_{k}^{'}\in gr_q S^{Lee}_{\alpha,h} $, linearly independent, for a fixed $q$. 
For any non-zero tuple $\{c_i\}$, consider  $\sum c_iv_{i}^{'}$ for some lifts $v_{i}^{'}\in KhR_{Lee}(K(r_{0}))$ so that $q(v_{i}^{'})$ is the same as the quantum filtration degree in the Lee-lasagna Module. The corresponding element in $E^{\infty}(K(r_{0}))$ is non-zero and must include all the terms in the filtered colimit, otherwise contradicting the linear independence of $\{v_i\}_{i=1}^{k}$
Hence $\text{Ker}(F)=0$, proving that $F$ is an isomorphism.
\end{proof}
\begin{Cor}
    \begin{equation*}
        rank(S^{2}(X,L,\mathbb{k}))\geq rank (\bigoplus_{h,q\in\mathbb{Z}} gr_{q}S^{Lee}(X,L,\mathbb{k}))
    \end{equation*}
\end{Cor}
It is easy to check that this inequality is true in individual quantum degrees as well, as long as it's not the $q=-\infty$ grading.
\subsection{Connect Sum Formula}
For manifolds with boundary $X_1$ and $X_2$  and links $L_i \subset \partial X_i$, the following formula was proved in \cite{mn22}: $$S^{Z}(X_1 \natural X_2 , L_1 \sqcup L_2 , \mathbb{F}) \ \cong \  S^{Z}(X_1 , L_1, \mathbb{F}) \bigotimes S^{Z}(X_2 , L_2, \mathbb{F})$$
We prove a similar formula for our spectral sequence construction. To be more precise, we will prove a pseudo-finality result for the category of skeins. We will construct a pseudo-final functor that commutes with differentials. 
Before we state the relevant result, we define a functor:  $$F: \overline{C(X_1 , L_1)} \ \bigtimes \ \overline{C(X_2 , L_2)} \longrightarrow \overline{C(X_1 \natural X_2 , L_1 \sqcup L_2 )} $$  First take the boundary connect sum away from the links $L_i$. For each $X_i$, the subspace $X_i - N(\partial X_i) $ embeds into $X_1 \natural X_2$. Now, for $([\Sigma_{X_1}],[\Sigma_{X_2}])\in\overline{C(X_1 , L_1)} \ \bigtimes \ \overline{C(X_2 , L_2)}$, we embed each $[\Sigma_{X_i}]$ in the subspace corresponding to  $X_i -N(\partial X_i )$  inside $X_1 \natural X_2$. Then we add copies of $L_i\times I \subset N(\partial X_i)$ giving us an element of $\overline{C(X_1 \natural X_2 , L_1 \sqcup L_2 )}$. Henceforth, we will refer to the functor $F$ as the ``inclusion functor".
\begin{Thm}
    For a pair of four manifolds and boundary links $(X_i , L_i) \ , \ i=1,2 $ The inclusion functor  $$\overline{C(X_1 , L_1)} \ \bigtimes \ \overline{C(X_2 , L_2)} \longrightarrow \overline{C(X_1 \natural X_2 , L_1 \sqcup L_2 )} $$ is pseudo-final.
\end{Thm}
\begin{proof}
 Throughout this proof, we will use the following convention for the diagrams: The manifold $X_1$ will be in green while $X_2$ will be in blue, and the intersection of the two will be a line representing the $B^3$ along which we will take the boundary connect sum. \\
\begin{figure}
 \centering

\begin{tikzpicture}[scale=2, anchorbase]
\filldraw[opacity=0.3, green] (-2,0) rectangle (2,2);
\filldraw[opacity=0.3, blue] (2,0) rectangle (6,2);
\draw[very thick] (2,0) to (2,2);
\node[opacity=1] at (4,1) {$X_2$};
\node[opacity=1] at (0,1) {$X_1$};
\node[opacity=1] at (3,1.7) {$ \text{connect sum } B^3$};
\node[opacity=1] at (2.2,1.6) {$\swarrow$};
    
\end{tikzpicture}
\caption{Convention for the connect sum computations}
\label{fig:bluegreen}
\end{figure}
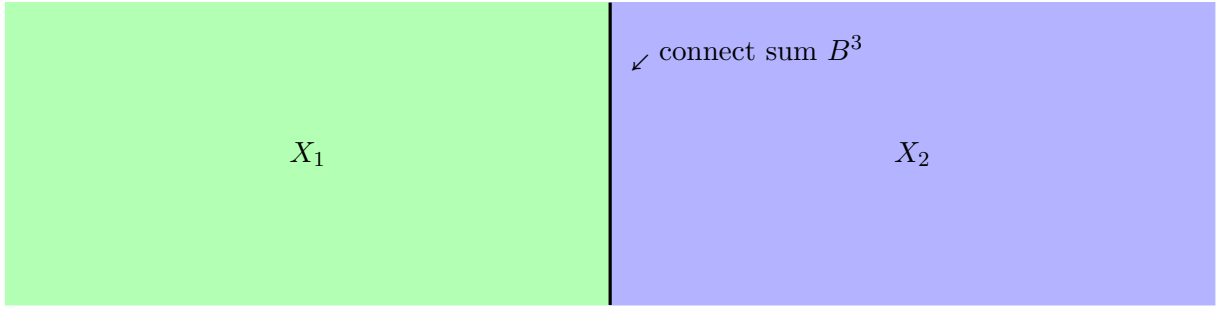
Let us start with a skein $[\Sigma] \in \overline{C(X_1 \natural X_2 , L_1 \sqcup L_2 )}$ . We first use a map in $G_{\overline{C(X_1 \natural X_2 , L_1 \sqcup L_2 )}}$ to make the input balls smaller and make sure none of them intersect the connected sum $B^3$ (henceforth denoted $B$). We then use a morphism given by an isotopy to make the surface transversal to $B$. Hence, locally, at the intersection, the surface will be of the form ``$L\times I$". Then, we use another map in $G_{\overline{C(X_1 \natural X_2 , L_1 \sqcup L_2 )}}$ to ``neck cut" the surface and introduce two input balls. The procedure is demonstrated in the figure \ref{fig:connectsum1}. Note that the last two maps are from the subclass $G_{\overline{C(X_1 \natural X_2 , L_1 \sqcup L_2 )}}$ depicted in the figure \ref{fig:newcatfg}. 
Finally, we prove that the choices we made during this procedure are related through maps in the product category, thereby establishing the first step of proving finality as required by Definition \ref{def:final}.  Assume that we started with a skein $[\Sigma]$ with input balls disjoint from $B$. Suppose the first morphism is given by an isotopy that takes us to $[\Sigma_{0}]$, while the other morphism takes us to $[\Sigma_1]$. Concatenating these two, consider $\Sigma_{(t)}$ be the isotopy from $\Sigma_0$ to $\Sigma_1$. Then, consider the cobordism $J=(\Sigma_{(t)} \ \cap \ B  \ )\times \ \{t\}$ between the two links $\Sigma_0 \  \cap \ B$ and $\Sigma_1 \  \cap \ B$.  Let $\Sigma_j^i$  be $\Sigma_j |_{X_i}$. Then one can see the following two are isotopic by an isotopy supported in $X_1$:
$$\Sigma_0^1\ \cup \ J  \ \cong_{isotopy} \Sigma_1^1 $$To be precise, by $\Sigma_0^1\ \cup \ J $, we mean the following : 
\par Fix a tubular neighborhood $B\times [-1,1]$ of the connect sum ball, $I_{t}$ be the isotopy in $X_1$ taking $X_1$ to $X_1-(B\times [-1,0])$ . Then embed J inside the collar neighborhood $B \times [-1,0]$. The surface of interest is $I_{\{1\}}(\Sigma_{0}^{1})\ \cup \ J$. 
\par The explicit isotopy at time $t$ is given by the surface : $$I_{\{t\}}(\Sigma_0^1) \ \bigcup \ (\cup_{s=0}^{t} (\Sigma_{\{s\}} \cap B)\ \times \ \{s\})$$
Similarly, by embedding $J$ in $B \ \times  \ [0,1]$, one can prove :
$$\Sigma_0^2 \ \cong_{isotopic}  J \ \cup \Sigma_1^2 $$
Hence, we can go from both the $\Sigma_j$'s  to $\Sigma_0^1 \ \cup \ J \ \cup \ \Sigma_1^2$ through a map given by an isotopy in a fixed  $X_i$.  Hence, after neck cutting, we get a commutative diagram as required by the finality condition in \ref{def:final}. These two maps are depicted in the Figure \ref{fig:welldefbluegreen}, where we are denoting $\Sigma_0 \ \cap \ B$ by $L$ and $\Sigma_1 \ \cap \ B$ by $L^{'}$.

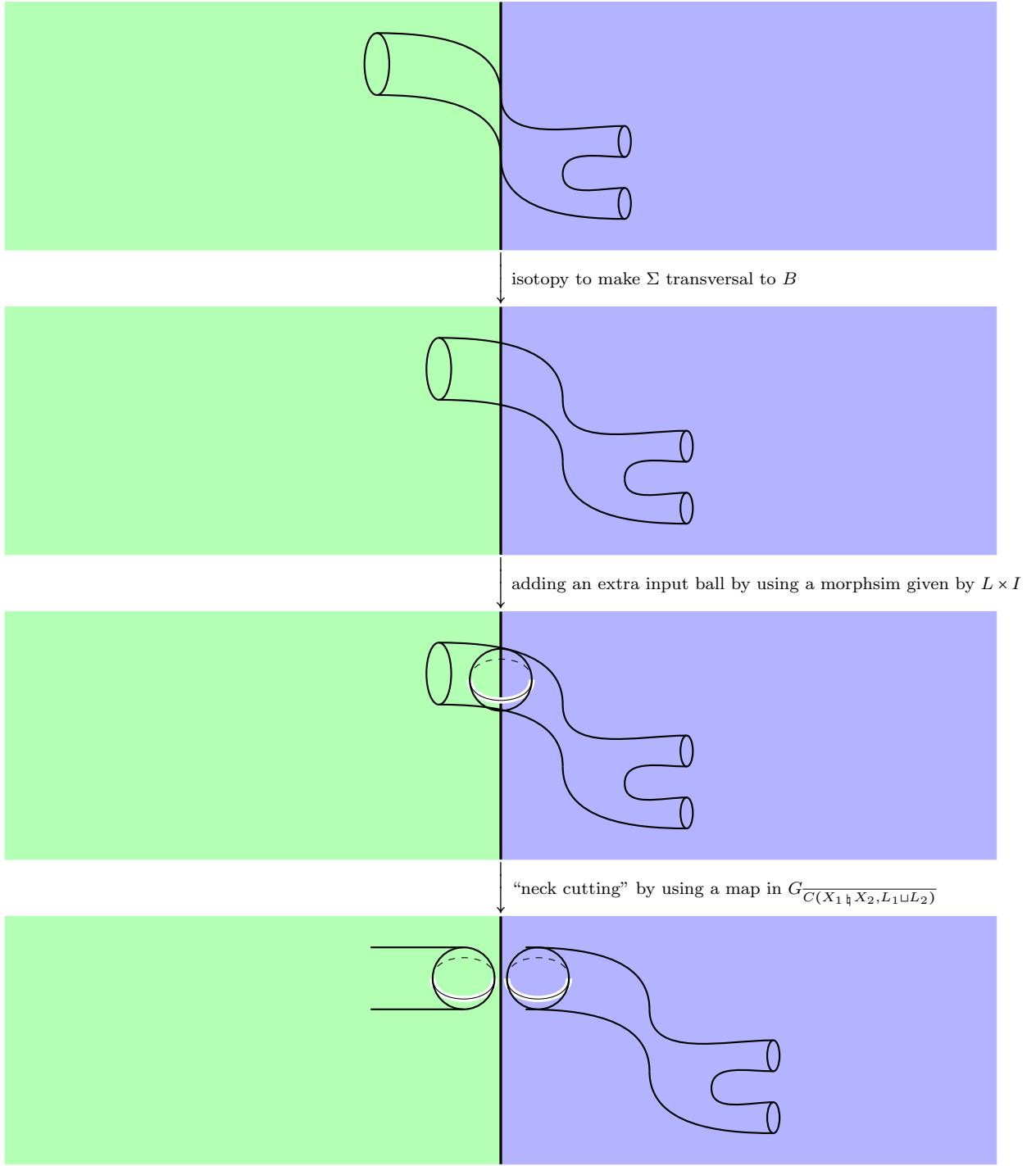
\begin{figure}
    \centering
    
\begin{tikzpicture}[scale=2, anchorbase]
\filldraw[opacity=0.3, green] (-2,0) rectangle (2,2);
\filldraw[opacity=0.3, blue] (2,0) rectangle (6,2);
\draw[very thick] (2,0) to (2,2);
\draw[thick]   (1,1.75) to [in =90 , out= 0 ](2,1.25);
\draw[thick]    (2,1.25) to [in =180 , out =270]   (3,1);
\draw[thick]   (1,1.25) to [in =90 , out= 0 ](2,0.75);
\draw[thick]    (2,0.75) to [in =180 , out =270]   (3,0.25);
\draw[thick]    (3,0.75) to [in=90, out=180] (2.5,0.6125);
\draw[thick]    (2.5,0.6125) to [in=180, out=-90] (3,0.5);
\draw[thick] (1,1.5) ellipse (0.1 and 0.25);
\draw[thick] (3,0.875) ellipse (0.05 and 0.125);
\draw[thick] (3,0.375) ellipse (0.05 and 0.125);
\end{tikzpicture} 
\\
$\bigdownarrow{\text{isotopy to make $\Sigma$ transversal to $B$ } }$  \\
\begin{tikzpicture}[scale=2, anchorbase]
\filldraw[opacity=0.3, green] (-2,0) rectangle (2,2);
\filldraw[opacity=0.3, blue] (2,0) rectangle (6,2);
\draw[very thick] (2,0) to (2,2);
\begin{scope}[shift={(0.5,0)}]
\draw[thick]   (1,1.75) to [in =90 , out= 0 ](2,1.25);
\draw[thick]    (2,1.25) to [in =180 , out =270]   (3,1);
\draw[thick]   (1,1.25) to [in =90 , out= 0 ](2,0.75);
\draw[thick]    (2,0.75) to [in =180 , out =270]   (3,0.25);
\draw[thick]    (3,0.75) to [in=90, out=180] (2.5,0.6125);
\draw[thick]    (2.5,0.6125) to [in=180, out=-90] (3,0.5);
\draw[thick] (1,1.5) ellipse (0.1 and 0.25);
\draw[thick] (3,0.875) ellipse (0.05 and 0.125);
\draw[thick] (3,0.375) ellipse (0.05 and 0.125);

\end{scope}

\end{tikzpicture} \\ $\bigdownarrow{\text{adding an extra input ball by using a morphsim given by $L\times I$}}$ \\
\begin{tikzpicture}[scale=2, anchorbase]
\filldraw[opacity=0.3, green] (-2,0) rectangle (2,2);
\filldraw[opacity=0.3, blue] (2,0) rectangle (6,2);
\draw[very thick] (2,0) to (2,2);
\begin{scope}[shift={(0.5,0)}]
   \draw[thick]   (1,1.75) to [in =90 , out= 0 ](2,1.25);
\draw[thick]    (2,1.25) to [in =180 , out =270]   (3,1);
\draw[thick]   (1,1.25) to [in =90 , out= 0 ](2,0.75);
\draw[thick]    (2,0.75) to [in =180 , out =270]   (3,0.25);
\draw[thick]    (3,0.75) to [in=90, out=180] (2.5,0.6125);
\draw[thick]    (2.5,0.6125) to [in=180, out=-90] (3,0.5);
\draw[thick] (1,1.5) ellipse (0.1 and 0.25);
\draw[thick] (3,0.875) ellipse (0.05 and 0.125);
\draw[thick] (3,0.375) ellipse (0.05 and 0.125);

\end{scope}
\sphere{2}{1.45}{0.25}[thick]
\end{tikzpicture} \\ $\bigdownarrow{\text{``neck cutting" by using a map in $G_{\overline{C(X_1\natural X_2, L_1 \sqcup L_2)}}$}}$ \\
\begin{tikzpicture}[scale=2, anchorbase]
\filldraw[opacity=0.3, green] (-2,0) rectangle (2,2);
\filldraw[opacity=0.3, blue] (2,0) rectangle (6,2);
\draw[very thick] (2,0) to (2,2);
\begin{scope}[shift={(1.2,0)}]
\draw[thick]   (1,1.75) to [in =90 , out= 0 ](2,1.25);
\draw[thick]    (2,1.25) to [in =180 , out =270]   (3,1);
\draw[thick]   (1,1.25) to [in =90 , out= 0 ](2,0.75);
\draw[thick]    (2,0.75) to [in =180 , out =270]   (3,0.25);
\draw[thick]    (3,0.75) to [in=90, out=180] (2.5,0.6125);
\draw[thick]    (2.5,0.6125) to [in=180, out=-90] (3,0.5);

\draw[thick] (-0.25,1.75) to (0.5,1.75);
\draw[thick] (-0.25,1.25) to (0.5,1.25);
\draw[thick] (3,0.875) ellipse (0.05 and 0.125);
\draw[thick] (3,0.375) ellipse (0.05 and 0.125);
\end{scope}
\begin{scope}[shift={(1.2,0.1)}]

\sphere{1.1}{1.4}{0.25}[thick]
\sphere{0.5}{1.4}{0.25}[thick]
\end{scope}

\end{tikzpicture} 
    
\caption{Connect sum formula algorithm}
 \label{fig:connectsum1}
\end{figure}
\begin{figure}
\centering

\begin{tikzpicture}[scale=1, anchorbase]
\filldraw[opacity=0.3, green] (-2,0) rectangle (2,2);
\filldraw[opacity=0.3, blue] (2,0) rectangle (6,2);
\draw[very thick] (2,0) to (2,2);
\draw[thick]   (1,1.75) to [in =90 , out= 0 ](2,1.25);
\draw[thick]    (2,1.25) to [in =180 , out =270]   (3,1);
\draw[thick]   (1,1.25) to [in =90 , out= 0 ](2,0.75);
\draw[thick]    (2,0.75) to [in =180 , out =270]   (3,0.25);
\draw[thick]    (3,0.75) to [in=90, out=180] (2.5,0.6125);
\draw[thick]    (2.5,0.6125) to [in=180, out=-90] (3,0.5);
\draw[thick] (1,1.5) ellipse (0.1 and 0.25);
\draw[thick] (3,0.875) ellipse (0.05 and 0.125);
\draw[thick] (3,0.375) ellipse (0.05 and 0.125);

\end{tikzpicture}  \\ $\bigdownarrow{\text{isotopy to get transversal intersection: locally $L\times I$}} $\\
\begin{tikzpicture}[scale=1, anchorbase]
\filldraw[opacity=0.3, green] (-2,0) rectangle (2,2);
\filldraw[opacity=0.3, blue] (2,0) rectangle (6,2);
\draw[very thick] (2,0) to (2,2);
\begin{scope}[shift={(0.5,0)}]
   \draw[thick]   (1,1.75) to [in =90 , out= 0 ](2,1.25);
\draw[thick]    (2,1.25) to [in =180 , out =270]   (3,1);
\draw[thick]   (1,1.25) to [in =90 , out= 0 ](2,0.75);
\draw[thick]    (2,0.75) to [in =180 , out =270]   (3,0.25);
\draw[thick]    (3,0.75) to [in=90, out=180] (2.5,0.6125);
\draw[thick]    (2.5,0.6125) to [in=180, out=-90] (3,0.5);
\draw[thick] (1,1.5) ellipse (0.1 and 0.25);
\draw[thick] (3,0.875) ellipse (0.05 and 0.125);
\draw[thick] (3,0.375) ellipse (0.05 and 0.125);

\end{scope}

\end{tikzpicture} 
\\ $\bigdownarrow{\text{``neck cutting" by using a map in $G_{\overline{C(X_1\natural X_2, L_1 \sqcup L_2)}}$}}$ \\
\begin{tikzpicture}[scale=1, anchorbase]
\filldraw[opacity=0.3, green] (-2,0) rectangle (2,2);
\filldraw[opacity=0.3, blue] (2,0) rectangle (6,2);
\draw[very thick] (2,0) to (2,2);
\begin{scope}[shift={(1.2,0)}]
\draw[thick]   (1,1.75) to [in =90 , out= 0 ](2,1.25);
\draw[thick]    (2,1.25) to [in =180 , out =270]   (3,1);
\draw[thick]   (1,1.25) to [in =90 , out= 0 ](2,0.75);
\draw[thick]    (2,0.75) to [in =180 , out =270]   (3,0.25);
\draw[thick]    (3,0.75) to [in=90, out=180] (2.5,0.6125);
\draw[thick]    (2.5,0.6125) to [in=180, out=-90] (3,0.5);
\draw[thick] (-0.25,1.75) to (0.5,1.75);
\draw[thick] (-0.25,1.25) to (0.5,1.25);
\draw[thick] (3,0.875) ellipse (0.05 and 0.125);
\draw[thick] (3,0.375) ellipse (0.05 and 0.125);
\end{scope}
\begin{scope}[shift={(1.2,0.1)}]

\sphere{1.1}{1.4}{0.25}[thick]
\sphere{0.5}{1.4}{0.25}[thick]
\end{scope}
\end{tikzpicture} \\ $\bigdownarrow{\text{Another ``neck cut" at $L^{'}\times I$ }} $\\

\begin{tikzpicture}[scale=1, anchorbase]
\filldraw[opacity=0.3, green] (-2,0) rectangle (2,2);
\filldraw[opacity=0.3, blue] (2,0) rectangle (6,2);
\draw[very thick] (2,0) to (2,2);
\begin{scope}[shift={(1.2,0)}]
   \draw[thick]   (1,1.75) to [in =90 , out= 0 ](2,1.25);
\draw[thick]    (2,1.25) to [in =180 , out =270]   (3,1);
\draw[thick]   (1,1.25) to [in =90 , out= 0 ](2,0.75);
\draw[thick]    (2,0.75) to [in =180 , out =270]   (3,0.25);
\draw[thick]    (3,0.75) to [in=90, out=180] (2.5,0.6125);
\draw[thick]    (2.5,0.6125) to [in=180, out=-90] (3,0.5);

\draw[thick] (-0.25,1.75) to (0.5,1.75);
\draw[thick] (-0.25,1.25) to (0.5,1.25);
\draw[thick] (4.25,1) to (4.7,1);
\draw[thick] (4.3,0.75) to (4.7,0.75);
\draw[thick] (4.25,0.5) to (4.7,0.5);
\draw[thick] (4.3,0.25) to (4.7,0.25);
\end{scope}
\begin{scope}[shift={(1.2,-5.9)}]
\sphere{3.15}{6.75}{0.2}[thick]
\sphere{3.15}{6.25}{0.2}[thick]

\sphere{4.15}{6.75}{0.2}[thick]
\sphere{4.15}{6.25}{0.2}[thick]
\sphere{1.1}{7.4}{0.25}[thick]
\sphere{0.5}{7.4}{0.25}[thick]
\end{scope}
\end{tikzpicture} \\ $\bigdownarrow{\text{map of skeins given by cobordism from $L$ to $L^{'}$}} $\\
\begin{tikzpicture}[scale=1, anchorbase]
\filldraw[opacity=0.3, green] (-2,0) rectangle (2,2);
\filldraw[opacity=0.3, blue] (2,0) rectangle (6,2);
\draw[very thick] (2,0) to (2,2);
\begin{scope}[shift={(0,0)}]
\draw[thick] (4.25,1) to (4.7,1);
\draw[thick] (4.3,0.75) to (4.7,0.75);
\draw[thick] (4.25,0.5) to (4.7,0.5);
\draw[thick] (4.3,0.25) to (4.7,0.25);
\draw[thick] (-0.25,1.75) to (0.5,1.75);
\draw[thick] (-0.25,1.25) to (0.5,1.25);
\end{scope}
\begin{scope}[shift={(0,-5.9)}]
\sphere{4.15}{6.75}{0.2}[thick]
\sphere{4.15}{6.25}{0.2}[thick]  
\sphere{0.5}{7.4}{0.25}[thick]
\end{scope}

\end{tikzpicture}
\\ $\biguparrow{\text{map of skeins using the cobordism from $L^{'}$ to $L$}} $\\
\begin{tikzpicture}[scale=1, anchorbase]
\filldraw[opacity=0.3, green] (-2,0) rectangle (2,2);
\filldraw[opacity=0.3, blue] (2,0) rectangle (6,2);
\draw[very thick] (2,0) to (2,2);
\begin{scope}[shift={(-3,0)}]
\begin{scope}[shift={(1.2,0)}]
   \draw[thick]   (1,1.75) to [in =90 , out= 0 ](2,1.25);
\draw[thick]    (2,1.25) to [in =180 , out =270]   (3,1);
\draw[thick]   (1,1.25) to [in =90 , out= 0 ](2,0.75);
\draw[thick]    (2,0.75) to [in =180 , out =270]   (3,0.25);
\draw[thick]    (3,0.75) to [in=90, out=180] (2.5,0.6125);
\draw[thick]    (2.5,0.6125) to [in=180, out=-90] (3,0.5);
\draw[thick] (-0.25,1.75) to (0.5,1.75);
\draw[thick] (-0.25,1.25) to (0.5,1.25);
\draw[thick] (4.25,1) to (4.7,1);
\draw[thick] (4.3,0.75) to (4.7,0.75);
\draw[thick] (4.25,0.5) to (4.7,0.5);
\draw[thick] (4.3,0.25) to (4.7,0.25);
\end{scope}
\begin{scope}[shift={(1.2,-5.9)}]
\sphere{3.15}{6.75}{0.2}[thick]
\sphere{3.15}{6.25}{0.2}[thick]

\sphere{4.15}{6.75}{0.2}[thick]
\sphere{4.15}{6.25}{0.2}[thick]
\sphere{1.1}{7.4}{0.25}[thick]
\sphere{0.5}{7.4}{0.25}[thick]
\end{scope}
\end{scope}
\end{tikzpicture} \\ $\biguparrow{\text{``neck cut" at $L$}} $\\

\begin{tikzpicture}[scale=1, anchorbase]
\filldraw[opacity=0.3, green] (-2,0) rectangle (2,2);
\filldraw[opacity=0.3, blue] (2,0) rectangle (6,2);
\draw[very thick] (2,0) to (2,2);
\begin{scope}[shift={(-1.5,0)}]
   \draw[thick]   (1,1.75) to [in =90 , out= 0 ](2,1.25);
\draw[thick]    (2,1.25) to [in =180 , out =270]   (3,1);
\draw[thick]   (1,1.25) to [in =90 , out= 0 ](2,0.75);
\draw[thick]    (2,0.75) to [in =180 , out =270]   (3,0.25);
\draw[thick]    (3,0.75) to [in=90, out=180] (2.5,0.6125);
\draw[thick]    (2.5,0.6125) to [in=180, out=-90] (3,0.5);
\draw[thick] (4.25,1) to (4.7,1);
\draw[thick] (4.3,0.75) to (4.7,0.75);
\draw[thick] (4.25,0.5) to (4.7,0.5);
\draw[thick] (4.3,0.25) to (4.7,0.25);
\end{scope}
\begin{scope}[shift={(-1.5,-5.9))}]
\sphere{3.15}{6.75}{0.2}[thick]
\sphere{3.15}{6.25}{0.2}[thick]
\sphere{4.15}{6.75}{0.2}[thick]
\sphere{4.15}{6.25}{0.2}[thick]
\end{scope}
\end{tikzpicture} \\  $\biguparrow{\text{``neck cut" at $L^{'}$}}$ \\
\begin{tikzpicture}[scale=1, anchorbase]
\filldraw[opacity=0.3, green] (-2,0) rectangle (2,2);
\filldraw[opacity=0.3, blue] (2,0) rectangle (6,2);
\draw[very thick] (2,0) to (2,2);
\begin{scope}[shift={(-0.8,0)}]
   \draw[thick]   (1,1.75) to [in =90 , out= 0 ](2,1.25);
\draw[thick]    (2,1.25) to [in =180 , out =270]   (3,1);
\draw[thick]   (1,1.25) to [in =90 , out= 0 ](2,0.75);
\draw[thick]    (2,0.75) to [in =180 , out =270]   (3,0.25);
\draw[thick]    (3,0.75) to [in=90, out=180] (2.5,0.6125);
\draw[thick]    (2.5,0.6125) to [in=180, out=-90] (3,0.5);
\draw[thick] (1,1.5) ellipse (0.1 and 0.25);
\draw[thick] (3,0.875) ellipse (0.05 and 0.125);
\draw[thick] (3,0.375) ellipse (0.05 and 0.125);

\end{scope}

\end{tikzpicture} \\ $\biguparrow{\text{isotopy to get transversality: locally $L^{'}\times I$}}$ \\
\begin{tikzpicture}[scale=1, anchorbase]
\filldraw[opacity=0.3, green] (-2,0) rectangle (2,2);
\filldraw[opacity=0.3, blue] (2,0) rectangle (6,2);
\draw[very thick] (2,0) to (2,2);
\draw[thick]   (1,1.75) to [in =90 , out= 0 ](2,1.25);
\draw[thick]    (2,1.25) to [in =180 , out =270]   (3,1);
\draw[thick]   (1,1.25) to [in =90 , out= 0 ](2,0.75);
\draw[thick]    (2,0.75) to [in =180 , out =270]   (3,0.25);
\draw[thick]    (3,0.75) to [in=90, out=180] (2.5,0.6125);
\draw[thick]    (2.5,0.6125) to [in=180, out=-90] (3,0.5);
\draw[thick] (1,1.5) ellipse (0.1 and 0.25);
\draw[thick] (3,0.875) ellipse (0.05 and 0.125);
\draw[thick] (3,0.375) ellipse (0.05 and 0.125);
\end{tikzpicture}  
\caption{Checking independence of the connect sum procedure from the choice of isotopy}
\label{fig:welldefbluegreen}
\end{figure}

\end{proof}
As an immediate corollary, we get the following result:
\begin{Cor}
    For 2-handlebodies $X_1$ and $X_2$, the inclusion functor $$ \overline{C(X_1, L_1)} \times \overline{C(X_2, L_2)} \longrightarrow \overline{C(X_1 \natural X_2 , L_1 \sqcup L_2 )} $$ induces an isomorphism of the spectral sequences:
    $$S^{E^n}(X_1,L_1 , \mathbb{F}) \bigotimes S^{E^n}(X_2,L_2, \mathbb{F}) \longrightarrow S^{E^n}(X_1 \natural X_2,L_1 \sqcup L_2, \mathbb{F})$$\footnote{Recall, that our construction of the differential is for all four manifolds and not just for 2-handlebodies. Hence, this result is slightly more general in that it says if the spectral sequence exists for each $X_i$ (not necessarily 2-handlebodies), then it exists for their boundary connect sum as well. }
\end{Cor}
Let us see what happens to the spectral sequences under four-handle attachment. This is an identical construction to the one in \cite{mn22}.
\begin{Thm}
    Let $X^{'}$ be a closed four-manifold obtained by attaching a four-handle to $X$. Then the inclusion $X\rightarrow X^{'}$ induces a pseudo-final map $$\overline{C(X,\emptyset)} \rightarrow \overline{C(X^{'},\emptyset)}$$
\end{Thm}
\begin{proof}
Since the co-core of a four-handle is zero-dimensional and the skeins are 2-dimensional, one can map the skein through an isotopy to another skein that is disjoint from the co-core. One can similarly use the maps in $G_{\overline{C(X,\emptyset)}}$ to make the input balls disjoint from the co-core. This proves the first condition necessary for finality as in \ref{def:final}. The other condition follows from the fact that any two surfaces in $X$ that are isotopic in $X^{'}$ are also isotopic through an isotopy supported in $X$.
    
\end{proof}
Finally, since the connect-sum of closed manifolds can be realized by first removing the four handles, followed by taking a boundary connect sum and then finally attaching a four-handle.
That is, $$X_1\# X_2 = \{(X_1 -B^4 )\ \natural \  (X_2-B^4)\}\ \cup \ \text{4-handle}$$
Hence, we get the following result:
\begin{Thm}
    For closed manifolds $X_1$ and $X_2$, we have a pseudo-final map given by inclusion:
    $$ \overline{C(X_1, \emptyset)} \times \overline{C(X_2, \emptyset)} \longrightarrow \overline{C(X_1 \# X_2 , \emptyset )} $$
    In addition, if both $X_i$'s can be given a handlebody structure without 1- and 3-handles, then we have an isomorphism of spectral sequences:
    $$S^{E^n}(X_1,\emptyset, \mathbb{F}) \bigotimes S^{E^n}(X_2,\emptyset, \mathbb{F}) \longrightarrow S^{E^n}(X_1 \# X_2,\emptyset, \mathbb{F})$$
\end{Thm}

\subsection{Future Directions:} Throughout this paper, we point out a general strategy for relating lasagna modules of different TQFTs
\begin{enumerate}
    \item Find a functorial fix of a known KFT over $\mathbb{Q}$ or some other field of characteristic zero.
    \item Apply the techniques and theorems developed in this paper.
    \item Try developing (possibly weaker) notion of quantum degree for other KFTs so one can relate $E^\infty$ TQFT with the convergence TQFT (in other words, the homology of the filtered chain complex corresponding to the spectral sequence).
\end{enumerate}
There is already work in progress to fix the instanton spectral sequence in \cite{kronheimer2019instantonsbarnatanhomology} over $\mathbb{Q}$. One may also consider other perturbations of the Khovanov differential to get various lasagna modules. Lastly, we point out that various Floer homologies behave much better under the braid group action. Hence, they potentially will give much better computations and non-vanishing results for lasagna modules. This can be seen by observing that the Alexander polynomial (whose categorification is given by a Floer theory) has a much nicer cabling formula than the Jones polynomial (which corresponds to Khovanov homology).
We also wish to address whether a similar spectral sequence construction exists for the lasagna theory with 1-dimensional inputs defined in \cite{ren2025khovanovskeinlasagnamodules}. For this new lasagna theory, a 2-handlebody formula was recently proved in \cite{montague2026handledecompositions1dimensionalinputs}, one can analyze this formula to see if our constructions can be replicated for the 1-dimensional lasagna.
Finally, as an immediate next project, we wish to address the various formulas satisfied by our spectral sequence construction; a connect-sum formula being one of them. Another possible result would be to relate this spectral sequence to a spectral sequence of the Rozansky-Willis invariant defined in \cite{willis2019khovanovhomologylinksrs2times}.
In \cite{sullivan2024kirbybeltscategorifiedprojectors}, the authors proved the following equality:
$$S^{KhR_{2}}(S^2 \times D^2 , \ L \ )=RW(L)\otimes S^{KhR_{2}}(S^2 \times D^2,\emptyset)$$ where $RW(L)$ is the invariant defined in \cite{willis2019khovanovhomologylinksrs2times}.
We expect a similar result to be true for the spectral sequence construction. 

\printbibliography

\end{document}